\def\Bbb{\mathbb}

\def\dist{{\rm{dist}}}
\def\area{{\rm{area}}}
\def\dist{{\rm{dist}}}
\def\diam{{\rm{diam}}}

\def\Bbb{\mathbb}
\def\reals{\Bbb R}
\def\complex{\Bbb C}
\def\disk{\Bbb D}
\def\circle{\Bbb T}

\def\integers{\Bbb Z}
\def\rhp{{\Bbb H}_r}

\def\uhp{{\Bbb H}_u}

\def\classS{\cal S}
\def\class2{{\cal S}_2,0}
\def\class21{{\cal S}^*}
\def\classB{\cal B}

\def\disk{\Bbb D}
\def\cal{\mathcal}

\def\Julia{{\cal J}}
\def\zbar{{\overline{z}}}

\documentclass[12pt]{amsart}  % use this with LaTeX2e 
\usepackage{amssymb}    %use with LaTeX2e

\usepackage{graphicx,amssymb,amsthm,verbatim,amsfonts}

\theoremstyle{plain}                    %------- 'regular' theorem types
\newtheorem{thm}{Theorem}[section]
\newtheorem{cor}[thm]{Corollary}
\newtheorem{lemma}[thm]{Lemma}

\newcounter{ques}
   
\numberwithin{equation}{section}
\begin{document}
\baselineskip=18pt

%%%%%%%%%%%%%%%%%%   title  %%%%%%%%%%%%%%%%%%%%%%%%%%%%%%

\title [    Models for the Speiser class ]
          {  Models for the Speiser class}


\author {Christopher J. Bishop}
\address{C.J. Bishop\\
         Mathematics Department\\
         Stony Brook University \\
         Stony Brook, NY 11794-3651}
\email {bishop@math.sunysb.edu}
\thanks{The  author is partially supported by NSF Grant DMS 16-08577.
        }

%\date
\date{January 2017}
\maketitle

%%%%%%%%%%%%%%%%%%%%%%%% abstract %%%%%%%%%%%%%%%%%%%%%%%%

\begin{abstract}
The Eremenko-Lyubich  class $\classB$
consists of  transcendental entire functions 
with bounded singular set
and the Speiser class $\classS \subset \classB$   is made up 
of functions with a finite singular set.
 In  \cite{Bishop-EL-models}  I gave 
a method for constructing Eremenko-Lyubich functions 
that approximate certain simpler functions  called models.
 In this paper, I show that all models can be approximated
in a weaker sense by Speiser class functions, and  
that the stronger approximation of \cite{Bishop-EL-models}
can fail  for the Speiser class.
In particular, I give geometric restrictions on 
the geometry of a Speiser class function that need
not be satisfied by general Eremenko-Lyubich functions.
\end{abstract}

\clearpage

%%%%%%%%%%%%%%%%%%%%%%%% Introduction %%%%%%%%%%%%%%%%%%%%

\setcounter{page}{1}
\renewcommand{\thepage}{\arabic{page}}
\section{Introduction} \label{Intro} 

If $f$ is an  entire function, we say $f$ is transcendental 
if it is not a polynomial. 
The singular set of an entire function $f$ is the closure
of its finite critical values and finite asymptotic values, and
will be denoted $S(f)$.  The Eremenko-Lyubich class
$\classB$ consists of transcendental entire
 functions such that $S(f)$ is a bounded set.
% (such functions are also called  bounded type). 
 The Speiser class $\classS \subset \classB$
consists of  functions for which $S(f)$ is a finite set.
We let $\classS_{n,k} \subset \classS$ denote the 
sub-collection of functions with at most $n$ finite critical values
and $k$ finite  asymptotic values. In this paper, we 
will be particularly concerned with $\classS_{2,0}$. 

The Eremenko-Lyubich and Speiser
 classes are important in the study of transcendental
dynamics %(iteration of entire  functions),
and it is known that the dynamical behavior in the   Speiser
class is more restricted than in the Eremenko-Lyubich class.
 For example, a Speiser class
 function cannot have a wandering domain
(proved by Eremenko and Lyubich  in
\cite{MR1196102},  and  Goldberg and Keen in \cite{MR857196})
whereas an Eremenko-Lyubich function can have a
wandering domain \cite{Bishop-classS}.
On the other hand, various types of pathological
behavior, such as a  Julia set  with no 
non-trivial path components
can be constructed in either class (see 
 \cite{Bishop-classS} and \cite{MR2753600}).

In this paper we prove an approximation  theorem 
involving the Speiser class that is analogous to 
a result proven for the Eremenko-Lyubich class in 
\cite{Bishop-EL-models}.  
However, the function we construct here
 fails to satisfy some of  the side 
conditions that  could be imposed in 
\cite{Bishop-EL-models}. Comparing the two 
results helps illustrate 
the  differences between the two classes of functions. 
To state our results precisely, we need to introduce some 
notation.

Suppose $\Omega = \bigcup_j \Omega_j $ is a disjoint union
of unbounded simply connected domains 
so that sequences of components of $\Omega$ accumulate only
     at infinity.
 Also suppose  there exists  a map $\tau : \Omega \to \rhp + \rho_0= 
\{ x+iy: x>\rho_0\}$ that  is holomorphic   and such  that
\begin{enumerate}
\item the restriction of $\tau$ to each
   $\Omega_j$  is a conformal map $\tau_j:\Omega_j \to\rhp+\rho_0$, and
\item if $\{ z_n \} \subset \Omega$ and
        $\tau(z_n) \to \infty$ then $z_n \to \infty$ .
\end{enumerate}

An open set $\Omega$ as above will be called a model domain and 
$F = e^\tau$ will be called a model function. 
Note that $F: \Omega \to \{z:|z|>e^{\rho_0}\}$  is a 
covering map. 
A choice of both a model domain $\Omega$
and a model function $F$ on $\Omega$ will be called
a model.
If $\rho_0=0$ we say the model is normalized; this is 
the main case we will consider. 

We call the 
connected components, $\{ \Omega_j\}$, of  a model domain 
$\Omega$ the  tracts of $\Omega$.
In many cases of interest, the tracts will be Jordan domains
on the Riemann sphere with the point $\infty$ on the boundary.
The number of tracts can be either finite or infinite.
(Usually a domain  refers to an open connected set, so 
using  ``model domain''  for regions that 
 may have several connected components  might be confusing.  
We  are using the phrase to abbreviate
``the domain of definition of  the model function'' rather 
than invent a new term for this --
 terrain, territory, archipelago, \dots.
Except for this usage, the term domain  will retain its usual meaning).

Given a normalized model $(\Omega, F)$  we let
$$ \Omega(\rho) = \{ z \in \Omega : |F(z)| > e^\rho\} = 
   \tau^{-1}(\{ x+iy: x > \rho\}),$$
and
$$ \Omega(\rho,\delta) = \{ z \in  \Omega : 
 e^\rho < |F(z)| <  e^\delta \} = 
   \tau^{-1}(\{ x+iy:  \rho <  x  < \delta \}). 
$$
Given a tract $\Omega_j$ of $\Omega $,  we
let $\Omega_j(\rho) = \Omega(\rho) \cap \Omega_j  $
%\tau_j^{-1}(\{ x+iy: x>0\})
 and similarly for $\Omega_j( \rho, \delta )$.

\begin{figure}[htb]
\centerline{
\includegraphics[height=3.5in]{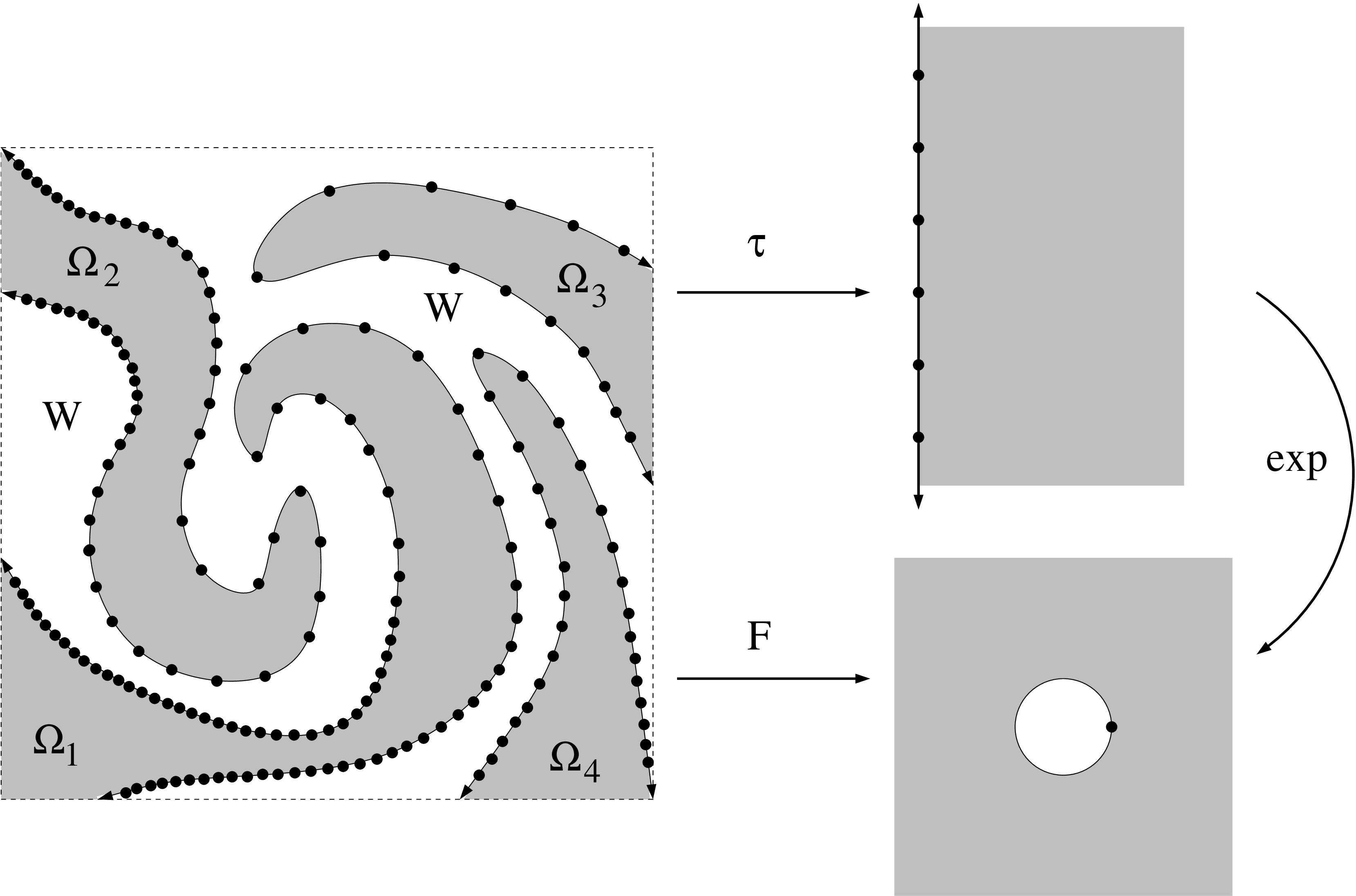}
}
\caption{ \label{LevelSet1}
A normalized  model consists of an open set $\Omega$ 
with possibly several tracts, each of which is 
mapped conformally by $\tau$ to  $\rhp$ and 
then by $e^z$ to $\{|z|> 1\}$, giving the model 
 function $F$ on $\Omega$. 
The points $F^{-1}(1)$ partition each boundary component 
into arcs. 
In the Eremenko-Lyubich class, $\tau$ can be rescaled 
independently on different tracts, so that 
 the partitions on different 
tract boundaries  are unrelated, but we will prove that for 
the Speiser class, the partitions for different 
tracts satisfy certain geometric relations. 
}
\end{figure}

Suppose $\Omega$ is a normalized model domain and 
$\rho >0$. The boundary of $\Omega_j(\rho)$ has a natural partition 
into sub-arcs  with endpoints that satisfy $\tau_j(z) 
\in \rho+ \pi i \integers$.  We call this a $\tau$-partition
or conformal partition of $\partial \Omega (\rho)$.
It is easy to see from 
the distortion theorems for conformal maps
(e.g., see Section \ref{review modulus} of this paper or
 Theorem I.4.5 of \cite{MR2150803}) 
that these sub-arcs of $\partial \Omega_j(\rho)$ 
are smooth  with bounds  depending only on $\rho$,
 and that adjacent arcs have 
comparable lengths (again with a constant
depending only on $\rho$).

Suppose $f$ is a transcendental entire function 
and that  $S(f) \subset \disk_R = \{ z: |z| < R\}$
(when $R=1$ we write $\disk = \disk_1$).
In  \cite{MR1196102}, Eremenko and Lyubich
observed  that  $\Omega=f^{-1}(\{z: |z| >R\})$
is a disjoint  union  of analytic, unbounded simply
connected domains and that $f$ acts a covering map
$f: \Omega_j \to \{|z| > R\}$  on each tract $\Omega_j$
of $\Omega$.
Thus each function $f$ in the Eremenko-Lyubich  class that 
satisfies $S(f) \subset \disk$  gives rise 
to a normalized model domain $\Omega = \{z : |f(z)|> 1\}$ and 
a model function $F = f|_\Omega$ (hence $\tau(z)$ is 
a branch of  $\log f(z)$).
The components of $\Omega$ are called the tracts of $f$.
We call a model arising in this way an Eremenko-Lyubich model. 
If $f$ is in the Speiser class, we call it a Speiser model.

The purpose  of this paper is  to 
quantify the differences between $\classB$ and $\classS$
in terms of models.
In \cite{Bishop-EL-models}, I  
 showed that Eremenko-Lyubich 
functions can essentially behave like arbitrary  
models near $\infty$; the tracts can have any   
shape and the choice of $\tau$ on each tract is  
independent of the choice in other tracts.  In this paper,
 I show that in Speiser   models the choice of $\tau$
in different tracts  must satisfy certain
geometric  constraints (e.g. Theorems \ref{not thm}
and  \ref{near constant});  however, given any model
$\Omega$ it is always possible to add extra tracts
and define $\tau$ on these new tracts so that the geometric
conditions are satisfied. Thus informally we say 
``every model is an Eremenko-Lyubich model'' and 
``every model is a sub-model of a  Speiser model''.
More precisely, 
the following theorem is proved in \cite{Bishop-EL-models}:

\begin{thm}[All models occur  in $\classB$] 
\label{QR}
Suppose $(\Omega,F)$ is a normalized  model
and $\rho >0$. 
Then there is a $f \in \classB$
and a quasiconformal homeomorphism $\varphi: \complex 
\to \complex$ so that  
$ F = f \circ \varphi$ on $  \Omega(\rho)$.
In addition,
\begin{enumerate}
 \item we have $|f \circ \varphi| \leq  e^{\rho}$ off $\Omega(\rho)$
       (i.e., $f$ is bounded off $\varphi(\Omega)$),
 \item  the singular set satisfies $S(f)  \subset D(0,e^{\rho})$,
\item the maximal dilatation
     $K$ of $\varphi$ depends only on $\rho$, 
\item the map $\varphi$ is conformal  except on $\Omega( \rho, 2\rho)$.
\end{enumerate} 
\end{thm} 

We will review the definition and basic properties of 
quasiconformal mappings in Section \ref{review modulus}. 
One of the main goals of this paper is to 
prove the following  analog of Theorem \ref{QR} 
 for the Speiser class:

\begin{thm} [All models  occur as  sub-models in $\classS$]
\label{App Omega} 
Suppose $(\Omega,F)$ is a normalized model and $\rho >0$. 
Then there is a $f \in \classS$ and a quasiconformal 
homeomorphism $\varphi: \complex 
\to \complex$ so that  $ F = f \circ \varphi$ on $  \Omega(\rho)$.
In addition, 
\begin{enumerate}
 \item  the function $f$  has no finite asymptotic  values and 
          two critical values,  $\pm  e^\rho$,
  \item  every critical  point of $f$  has degree $\leq 12$,
\item the maximal dilatation
      $K$ of $\varphi$ depends only on $\rho$, 
\item  the  map $\varphi$ is conformal on $\Omega(2\rho)$.
\end{enumerate} 
\end{thm}

The   maximal dilatation
bound for $\varphi$  remains bounded as $\rho \to \infty$, 
but blows up as $\rho \to 0$ (we will not be explicit 
about the dependence of $K$ on $\rho$, but estimates 
could be derived from a careful reading of \cite{Bishop-classS}).
It is not true that the  maximal dilatation $K$ tends to 
$1$ as $\rho \to \infty$, at least for the construction 
given here, since the use of the folding maps 
from \cite{Bishop-classS} introduce a fixed amount of 
distortion, independent of $\rho$.

The degree of a critical point $z$ of a holomorphic map $f$ 
is taken to be the local valence of $f$ near $z$, e.g., 
$f(z) =z^3$ has  a critical point of degree three  at $0$.
The bound in (2) follows immediately from the proof of the 
folding theorem in \cite{Bishop-classS}.
However, by making some simple changes to the
construction in \cite{Bishop-classS}, the 12
 can be improved to 4. 
 This will be discussed in more detail 
at the end of Section \ref{QC folding}.

The crucial  difference between
 Theorems \ref{QR} and  \ref{App Omega}
 is that the latter  omits the  conclusion
 ``$ |f \circ \varphi| \leq e^{\rho}$  off $\Omega(\rho)$''.
The function $f \in \classB$ constructed in Theorem \ref{QR} 
is only large where the model is large (inside $\Omega$),
 so it has the same number of tracts
as the model has. However, the function $f \in \classS$ in Theorem 
\ref{App Omega} might also be large outside $\Omega$, and 
so it can have ``extra'' tracts.
This is the sense in which approximation by Speiser functions 
is weaker than approximation by Eremenko-Lyubich functions.  
In fact, our proof will always introduce 
extra tracts; we will first give a construction  
that creates an infinite number of
extra tracts, and then give a more intricate construction 
that shows: 

\begin{thm} \label{details} 
The function $f$ in Theorem \ref{App Omega} may be chosen 
so that the number of tracts of $f$ is at most %$2 \#(\Omega)$.
 twice the number of tracts of the model $(\Omega,F)$.
\end{thm} 

Simple examples show that  some models with $n$ tracts 
require the approximating Speiser class function  to 
have $2n$ tracts, so the bound in Theorem \ref{details} 
is sharp.  Roughly speaking, if $\Omega$ 
has $n$ tracts, then the domain 
$W = \complex \setminus \Omega(\rho)$ has $n$  distinct
``ends'' at infinity. If these ends are each ``large''  
 compared to the tracts of the model, then each end must 
contain at least one  extra  tract of the approximating
Speiser class function.  A very 
concrete example is:  

\begin{thm} \label{not thm} 
The half-strip
 $S = \{ x+iy : x>0, |y|< 1\}$ cannot be mapped to any  
Speiser class  model domain  by any quasiconformal homeomorphism
of the plane.
\end{thm} 

In other words, there is no Speiser class function with 
a single tract, so that this tract is
the image of a half-strip under a 
quasiconformal map of the plane.
However, there are   Speiser
class  functions with two tracts, one of which  can 
be sent to a half-strip by a quasiconformal map of the 
plane; moreover, this tract can approximate 
the half-strip  in the Hausdorff metric on the plane 
 as closely as we wish.
See Figure \ref{TwoExtra} and the remarks in Section \ref{not}.
On the other hand, Theorem \ref{QR} implies
there are Eremenko-Lyubich functions with
single tracts that approximate the half-strip as 
closely as we wish in the Hausdorff metric.

The referee of this paper asked if Theorem \ref{not thm}
also holds for any tract that is contained in the
half-strip $S$. While our proof of Theorem \ref{not thm}
 extends to cover  many 
cases of this type,  and it  is not hard to 
see that  no 
subdomain of $S$ can  itself be a Speiser class 
model domain, there might be such   a subdomain
 that can be mapped to a Speiser class  model
 domain by some quasiconformal map of the plane.
Deciding this would be an 
interesting problem.  It would  also be 
very interesting to have a geometric characterization (even up 
to quasiconformal maps) of the tracts of Speiser class
functions that have a single tract.

Another difference between Theorems \ref{QR} and \ref{App Omega}
concerns the proofs. The proof of Theorem \ref{QR} given 
in \cite{Bishop-EL-models} is mostly self-contained  and 
depends on constructing a Blaschke product in the disk 
that approximates a certain inner  function  
arising from the model. On the other hand, 
the proof of Theorem \ref{App Omega} in this paper  depends 
on the more difficult quasiconformal folding
construction of Speiser class functions in \cite{Bishop-classS}.
The precise statement we use will be reviewed in 
Section \ref{QC folding}.

 Finally, we mention an application of 
Theorem \ref{App Omega} to dynamics. 
We call a  model $(\Omega,F)$  disjoint type if 
it is normalized and 
$\overline{\Omega} \cap \overline{\disk} = 
\emptyset$. 
An entire function is usually called 
disjoint type if (1) it is hyperbolic (the singular
set is bounded and every point in it tends to 
an attracting periodic cycle of $f$ under iteration)
and (2) the Fatou set is connected (the Fatou set is 
the largest open set on which the iterates of $f$ form 
a normal family; its complement is called the Julia 
set of $f$).
Alternatively, Proposition 2.1 of \cite{Rempe-arc-like}
  states that 
a transcendental entire function is disjoint type 
if and only if there is a Jordan domain $D$ so that 
$S(f) \subset D$ and $f(\overline{D}) \subset D$.
This implies that if  $(\Omega, F)$ is an 
disjoint type Eremenko-Lyubich  model,
then $F = f|_\Omega$  where $f$ 
is an Eremenko-Lyubich  entire function that is 
disjoint type  in the sense above (just take $D = \disk$).

We can iterate a model function  
$F$ as long as the iterates keep landing
in $\Omega$, and we define the Julia set of a model as
$$ \Julia(F) = \bigcap_{n\geq 0} \{ z\in \Omega: F^n(z) 
         \in \Omega \}.$$
If $F$ is a disjoint type Eremenko-Lyubich model,
% (i.e., $\{z: |z|\;|f|\geq 1\} \ca,
 then this is the same as  the usual Julia set of the 
extension of $F$.
Lasse Rempe-Gillen has pointed out 
that Theorem \ref{QR} implies that
 any disjoint type  model function 
is conjugate on its domain to a  disjoint type 
$f \in \classB$, in particular, 
the Julia set and the escaping set  for the model  function $F$
are homeomorphic via a quasiconformal mapping of the 
whole plane to the corresponding sets for $f$. Thus
various pathological examples in $\classB$ can be 
 constructed  simply by exhibiting  a model with the
desired property, e.g., see \cite{Rempe-arc-like}.

For the Speiser class, the approximating function $f$ 
may have extra tracts that do not correspond to 
tracts of the model. In this case, Rempe-Gillen's argument 
 implies the model function  $F$ restricted to its   Julia set  
can be conjugated to  a Speiser class function $f$ 
restricted to a certain   closed subset $A \subset \Julia(f)$. 
More precisely, 

\begin{thm} \label{Speiser models}
Suppose that $(\Omega, F)$ is any normalized, 
disjoint type 
model, that $f$ is a Speiser class function,
and that $\varphi$ is a
quasiconformal mapping  of the plane
so that $f = F \circ \varphi$ on 
$U= \varphi^{-1}(\Omega) $ (this is a sub-collection
of tracts of $f$). Assume that $(U, f|_U)$ is also a
normalized,  disjoint type model.
Then there is a quasiconformal map $\Phi : \complex \to 
\complex$  so that $\Phi \circ f = F \circ \Phi$ on $U$. 
\end{thm}

In other words, the Julia set of the model function 
$F$ is quasiconformally conjugate to a closed subset
$A$ of the Julia set of the Speiser class function $f$.
The set  $A$  consists of those points whose orbits stay
 within $U$ forever, 
where $U$ is the sub-collection of $f$'s tracts
corresponding to the tracts of $F$ via $\Phi$.
This result is  a straightforward application of 
Theorem 9.1 in \cite{Bishop-EL-models}
(which itself is simply a summary of an argument
of Rempe-Gillen  from \cite{MR2570071}).

The use of quasiconformal techniques to build and
understand entire functions with finite singular 
sets has a long history with its roots in the 
work of Gr{\"o}tzsch, Speiser,  Teichm{\"u}ller, Ahlfors,
Nevanlinna, Lavrentieff  and many others. 
The earlier 
work was often phrased in terms of  Riemann surfaces
and deciding if  a simply connected surface built 
by branching over a finite singular set was conformally 
equivalent to the plane or to the disk (the type 
problem; in the first case the uniformizing map 
gives a Speiser class function). Such constructions play
an important role in  value distribution theory;
see \cite{MR2121873} for an excellent survey of these  
methods and a very useful guide to this literature. Also see
Chapter VII of \cite{MR2435270}.
More recent work (including this paper) is motivated
by applications to dynamics, where the Speiser
class provides  an interesting  mix of structure (like 
polynomials, the quasiconformal 
equivalence classes are finite dimensional \cite{MR1196102})
and flexibility (as indicated 
by the results of  \cite{Bishop-classS}, \cite{Bishop-order}, 
\cite{Rempe-arc-like}
and the current paper). 

Many thanks to Simon Albrecht, Adam Epstein, Alex Eremenko
 and Lasse Rempe-Gillen
for  numerous helpful discussions about the 
content of this paper and
about the quasiconformal folding construction 
and its applications.
The introduction of the paper and the formulation of
the main result in terms  of models
was inspired by a lecture of Lasse Rempe-Gillen at
an ICMS conference on transcendental dynamics in Edinburgh, May 2013.
The results of both this paper and \cite{Bishop-EL-models}
originally appeared in a single  2013
 preprint titled ``The geometry of bounded
type entire functions''. Based partly on a
referee's report, I decided to split that manuscript  
in order to improve the exposition and separate the  self-contained
arguments  for the Eremenko-Lyubich class
 (now contained in \cite{Bishop-EL-models})
from the proofs for the Speiser class that depend  crucially  on 
the quasiconformal folding  techniques in \cite{Bishop-classS}.
Theorem  \ref{details} is new and did not appear 
in the earlier manuscript.
Malik Younsi read the revised manuscript and I 
greatly appreciate his comments and suggestions.
The referee of the current paper
produced two detailed and thoughtful reports 
that contained
numerous comments and suggestions that improved the 
exposition, and I  am most thankful for  the great deal 
of time and effort  that went into   these reports.
Finally, I am indebted to Aimo Hinkkanen for
for a great deal of encouragement  and 
constructive  advice that substantially improved
both this paper and its companion \cite{Bishop-EL-models}.

In this paper, the notation $A\lesssim B$ means that 
$A \leq  C B$ where $A,B$ are quantities that depend on 
some parameter and $C < \infty$ is a constant that is 
independent of the parameter. The notation means the 
same as $A = O(B)$. Similarly,   $A \gtrsim B$ is 
equivalent to $B\lesssim A$ or $B=O(A)$ . If 
$A\lesssim B$ and $A \gtrsim B$ then we say $A \simeq B$, 
i.e., $A$ are $B$ are comparable, independent of the 
parameter.

%--------------------------------------------------------------
\section{Modulus and quasiconformal maps }  \label{review modulus}

Many of our arguments involve the modulus of path 
families, conformal maps and quasiconformal maps, 
 so we briefly review the basic facts 
here for the convenience of the reader. Everything in 
this section  can be found (in greater detail and 
with proofs)   in standard references such 
as \cite{MR2241787} or \cite{MR2150803}.

An orientation preserving 
 homeomorphism $\varphi$ of the plane to itself is 
quasiconformal if it is absolutely continuous on all 
lines and $|\varphi_{\zbar}| \leq k |\varphi_z|$ almost
everywhere (with respect to area measure)  for
some $k < 1$. At points of differentiability, this means 
that the tangent map  of $\varphi$ sends circles to 
ellipses of eccentricity  at most $K = (1+k)/(1-k) \geq 1$. 
The smallest $K$ that works for $\varphi$ at almost 
every point  is called 
the maximal dilatation of $\varphi$;
% (or quasiconstant for brevity);
 such a map is also called $K$-quasiconformal.
A $K$-quasiconformal map $\varphi$ satisfies 
a Beltrami equation $f_\zbar = \mu f_z$ almost everywhere
for some bounded measurable  function $\mu$ called 
the dilatation of $f$ and $ \|\mu\|_\infty \leq k =(K-1)/(K+1) $.
A $1$-quasiconformal map is conformal.  The family of $K$-quasiconformal 
maps of the plane to itself that fix two finite points 
(usually taken to be $0,1$) is compact.

The measurable Riemann mapping theorem (e.g., see \cite{MR2241787})
says that given any measurable $\mu$ with 
$\|\mu\|_\infty = k < 1$, there is a $K$-quasiconformal 
map with dilatation $\mu$ almost everywhere.
An important consequence of this is that if $f$ is entire 
and $\varphi$ is quasiconformal, then there exists a 
quasiconformal $\psi$ so that $g = \varphi \circ f \circ \psi$ 
is  entire. Two entire functions $f$ and $g$ that 
are related in this 
way are called quasiconformally equivalent.
Eremenko and Lyubich proved 
that if $f$ has $q$ singular values, 
then the collection of entire 
functions  that are quasiconformally equivalent to $f$ 
forms a $(q+2)$-dimensional complex manifold (see Section 3 of 
\cite{MR1196102}).

Suppose $\Omega$ is a planar open set. A 
non-negative Borel function $\rho$ on $\Omega$
is called a metric on $\Omega$.
Suppose $\Gamma$ is a collection of locally rectifiable 
curves in $\Omega$.  We say  a metric $\rho$ is 
an admissible metric  for $\Gamma$ if 
$$ \inf_{\gamma \in \Gamma}  \int_\gamma \rho ds \geq 1,$$
and we define the modulus of $\Gamma$ as 
$$ M(\Gamma) = \inf_\rho \int_\Omega \rho^2 dxdy,$$
where the infimum is over all admissible metrics for $\Gamma$.
The reciprocal of $M(\Gamma)$ is called the extremal 
length of $\Gamma$ and is denoted $\lambda(\Gamma)$. 
The most important facts that we will need are:
\newline 
{\bf Conformal invariance:} if $f:\Omega \to \Omega'$ is 
conformal, $\Gamma$ is a path family in $\Omega$ and 
$\Gamma' =f(\Gamma)$,  
then $M(\Gamma') = M(\Gamma)$.
\newline 
{\bf Quasi-invariance:} If $f:\Omega \to \Omega'$ is 
$K$-quasiconformal, $\Gamma$ is a path family in $\Omega$
and $\Gamma'=f(\Gamma)$,
then $ M(\Gamma)/K \leq M(\Gamma') \leq K  \cdot  M(\Gamma)$.
\newline
{\bf Extension:} If $\Gamma, \Gamma'$ are path families such 
that each path in $\Gamma'$ contains a sub-path in $\Gamma$ 
then $M(\Gamma') \leq M(\Gamma)$.
In particular, if $\Gamma' \subset \Gamma$, then 
$M(\Gamma') \leq M(\Gamma)$.
\newline 
{\bf Parallel Rule:}  If $\Gamma_1, \dots, \Gamma_n$ are 
defined on disjoint open sets, and every $\gamma \in 
\cup_j \Gamma_j$ contains some curve in $\Gamma$ then 
$M(\Gamma ) \geq \sum_j M(\Gamma_j)$.
\newline 
{\bf Round  Annuli:} the modulus of the path family separating the 
two boundary components of the round annulus
 $A(r, R) = \{ z: r < |z| < R\}$ is 
$ (\log R/r)/ 2 \pi$. We call this the modulus of the 
annulus. Every topological annulus $\Omega \subset \complex$
 is conformally 
equivalent to a round annulus, and its modulus is equal 
to the modulus of the corresponding round annulus. 
\newline
{\bf Topological Annuli:} There is a $M_0 < \infty$ so 
that if $\Omega$ is a  topological annulus
with modulus $M \geq M_0$ 
then $\Omega$ contains 
a round annulus of modulus $M' >1$ and $M'$ tends to 
$\infty $ as $M$ tends to $\infty$.  
\newline
{\bf Reciprocity:} the modulus of the path family separating 
the two boundary components of a topological   annulus $\Omega$ is 
the reciprocal of the modulus of the path family 
in $\Omega$ that connects the two boundary components. 
%(paths that have one endpoint on each of the boundary components). 
\newline
{\bf Rectangles:} the modulus of the path family 
connecting the sides of length $a$ in a $a \times b $
rectangle is $a/b$.

 Another fact that we shall use repeatedly is:

\begin{lemma} \label{separation lemma} 
Suppose $e,f \subset \complex$ are disjoint Jordan arcs and let 
$\Gamma$ be the family  of closed curves in $\complex \setminus 
(e \cup f)$ that separates them. Let $M$ be the 
modulus of $\Gamma$. Then 
\begin{eqnarray} \label {mod est}
 \dist (e, f) \geq  \epsilon \cdot \min(\diam(e), \diam(f)),
\end{eqnarray} 
where $\epsilon >0$ depends only on a lower bound for $M$.
Conversely,  if (\ref{mod est}) holds, then $M$ is bounded 
away from zero with an estimate depending only on $\epsilon$.  
Moreover, $\epsilon$ tends to infinity  if and only if 
 $M$ tends to infinity.
\end{lemma} 

\begin{proof}
This is fairly standard.
Let $r = \min(\diam(e), \diam(f))$. 
If there are points $x \in e$ and 
$y \in f$ with $|x-y| \leq \epsilon $, 
then we define a metric $\rho$ on 
$\{ z:  \epsilon r < |x-z| < r/2\}$ 
by setting $\rho(z) = (|z-x| \log \frac 2 \epsilon)^{-1}$.
It is a standard exercise to show that $\rho$ is 
admissible and integrating $\rho^2$ gives 
$M \leq  ( \log \frac 2 \epsilon)^{-1}$  which 
tends to $0$ with $\epsilon$.  This proves the
first claim. For the other direction, 
suppose 
 $\dist (e, f) \geq  \epsilon r.$
Then setting $\rho(z) = (\epsilon r)^{-1}$ 
on an $\epsilon r$-neighborhood 
of $e$ (if $\diam(e) = r$) or $f$  (otherwise) 
gives an admissible metric for the path 
family connecting $e$ to $f$. Since this 
neighborhood has area at most $\pi (\epsilon r + r)^2$, 
computing the integral 
of $\rho^2$ shows this family has modulus at most 
$$   (\epsilon r)^{-2}   \pi  (\epsilon r + r)^{2} 
\leq     \pi (1+ \epsilon^{-2}).$$
Since this modulus is the reciprocal of the modulus 
of the path family separating $e$ and $f$ we get
a lower bound for the latter modulus in terms of $\epsilon$.
If $M$ is large, then by the topological annuli
property there is a large round annulus separating $e$
and $f$, and hence $\epsilon$ is large. Conversely, if 
$\epsilon$ is large, then  there is  clearly a large round 
annulus separating the curves and so the modulus $M$ is large.
\end{proof}

We will use the following in Section \ref{lower sec}. 

\begin{lemma} \label{two intervals}
If $I,J$ are disjoint intervals on $ \reals$, let
$M(I,J)$ be the modulus of the path family in 
$\uhp  =\{ x +i y : y >0\}$ (the upper half-plane)
  separating $I$ and $J$. If $I,J$ have unit 
length and are distance $r \geq 2$ apart, then $M(I,J) 
\simeq \log r$.
\end{lemma} 

\begin{proof}
There are several ways to estimate this modulus, but 
we will use a conformal map. 
Without loss of generality, assume $I=[-1,0], J=[r,r+1]$. 
The Schwarz-Christoffel formula (e.g., 
see \cite{MR1908657} and its references) that says $\uhp$  is conformally 
mapped to an $a \times b$ rectangle  with $I, J$ going to the 
sides of length $a$ by the map 
$$
 f(z) = \int^z \frac {dw} {(w+1)^{1/2} w^{1/2} (w-r)^{1/2}(w-r-1)^{1/2}}. 
$$
Moreover, 
$$a  =
  \int_{-1}^0 \frac {dx} {|x+1|^{1/2} |x|^{1/2} 
            |x-r|^{1/2}|x-r-1|^{1/2}} 
     \simeq 
   \frac 1r \int_{-1}^0 \frac {dx} {|x+1|^{1/2} |x|^{1/2} }
    \simeq \frac 1r
,$$
and  similarly 
\begin{eqnarray*}
b  &=&
  \int_{0}^r \frac {dx} {|x+1|^{1/2} |x|^{1/2} 
            |x-r|^{1/2}|x-r-1|^{1/2}} \\
  &   \simeq & 
   \frac 1r \int_{0}^{r/2} \frac {dx} {|x+1|^{1/2} |x|^{1/2} }\\
  &  \simeq & 
\frac 1r + \frac 1r \int_{1}^{r/2} \frac {dx} {x}  \\
  &  \simeq & \frac 1r  (1+\log   r) .
\end{eqnarray*}
Therefore, by conformal invariance and the rectangle rule, 
$M(I,J) = b/a \simeq 1+\log r $ (and $1+\log r \simeq 
\log r$ since $r \geq 2$). 

\end{proof} 

Other proofs of the lemma are possible. For example, one 
can use a M{\"o}bius transformation  
to map $I$ to $[-1,1]$, map $J$ to the complement 
of $[-y,y]$ for some $y \simeq r$, and then estimate the 
modulus of the planar complement of these using explicit
metrics.

Several times in this paper we will use Koebe's $\frac 14$-theorem 
and its consequences. Koebe's theorem says that if 
$f: \disk \to \Omega$ is conformal (holomorphic and $1$-to-$1$) 
then 
$$ \frac 14 |f'(z)|(1-|z|^2)
   \leq \dist(f(z), \partial \Omega) \leq
  |f'(z)|(1-|z|^2).$$
See Theorem  I.4.3 of \cite{MR2150803}. 
A consequence of this is that if $f$ is conformal 
on a   region $W$ and $E\subset W$ is compact, then 
$|f'|$ is comparable at any two points of $E$ with 
a constant that depends only on $E$ and $W$ (in fact,
it only depends on the diameter of $E$ in the hyperbolic 
metric for $W$).

The image  $\gamma$ of a line under a quasiconformal mapping of the plane 
to itself is called a quasi-line. Such curves $\gamma$ are 
exactly characterized by the three-point condition:  there 
is a $M < \infty$ so that given any 
three points $x,y,z \in \gamma$ with $x,y$ in different 
connected components of $\gamma \setminus \{ z\}$,  
we have $|x-z|\leq M|x-y|$. Equivalently, the subarc of $\gamma$ 
connecting $x$ and $y$ has diameter $O(|x-y|)$. 
We will use this in the following way. 

A quasidisk is the image of $\disk$ under a quasiconformal 
map of the plane. Abusing notation slightly, we will 
say 
$\Omega$ is an unbounded quasidisk if it the image of 
a half-plane under a quasiconformal map of the plane
(this sounds better than ``quasi-half-plane'', and
would be technically correct if we simply considered 
quasiconformal maps of the Riemann sphere to itself, rather 
than just maps that fix $\infty$).

\begin{lemma}  \label{unbounded quasidisk} 
Suppose $\Omega$  an unbounded quasidisk.
Then there is a $C < \infty$ so that 
given any $x \in \partial \Omega$, there is a curve $\gamma$ 
in $\Omega$ that connects $x$ to $\infty$ and satisfies 
$$ \dist(z,\partial \Omega) \geq   |z-x|/C ,$$
for every $z \in \gamma$.  If $|z|\geq 2|x|$, then 
$ \dist(z,\partial \Omega) \geq   |z|/(2C) ,$
for every $z \in \gamma$.  
\end{lemma} 

\begin{proof}
Suppose $\Omega = f(\rhp)$. Without loss of generality 
we may assume $x = f(0)$.  The right half-plane  can easily 
be quasiconformally mapped to a quarter-plane by a 
quasiconformal  map $g$ of the plane  
(leave radii fixed and
contract angles by a factor of two in one half-plane 
and expand them by a factor of $3/2$ in the remaining half-plane; 
we leave the details to the reader). Thus if 
$\Omega_1 \subset \Omega$  is
the image of the first quadrant under $f$, then it is also an unbounded 
quasidisk, and hence  $\partial \Omega_1$  satisfies the three point 
condition with some constant $C$.

 Suppose that there  was a  point on $\gamma= f(\reals^+)$
that was ``too close'' to $f(i\reals^+) \subset \partial \Omega$, 
i.e.,   suppose there were 
$s,t > 0$ so that
$$ |f(s) - f(it)| < \epsilon |f(s)-f(0)|. $$
Then the arc  of $\partial \Omega_1$ connecting 
$f(s)$ and $f(t)$ must have diameter $\leq C \epsilon |f(s)-f(0)|$
by the three-point condition, but it contains both 
$x= f(0)$ and $f(s)$ so it has diameter at least 
$|f(s)- f(0)|$. Thus $C \epsilon \geq 1$.  The same 
argument applies to the image of the fourth quadrant and the 
negative imaginary axis, and this proves the first part 
of the  lemma. The final claim  follows easily.
\end{proof}

Mori's theorem states that $K$-quasiconformal maps 
of the plane are bi-H{\"o}lder, i.e.,
$$  \frac 1{C|z-w|^\alpha} \leq   |f(z)-f(w)| \leq  C|z-w|^\alpha,$$
where $\alpha$ depends only on $K$.
Quasiconformal maps of the plane are also quasisymmetric:
there is a homeomorphism  $\eta$ from $[0, \infty)$ to  
itself such that $|x-y| \leq t|a-b|$ implies
 $|f(x)-f(y)|\leq \eta (t)|f(x)-f(y)|$. 
See \cite{MR1800917} and its references.

\begin{lemma} \label{QC edge nbhds}
Given a Jordan arc $\gamma \subset \complex$ define 
$$ \gamma(r) =\{ z \in \complex :  
\dist(z,\gamma) \leq r \cdot \diam( \gamma) \} .$$
If $f$ is a $K$-quasiconformal map of the plane to 
itself, then there are $ 0< s < t < \infty$ depending 
only on $r$ and $K$ so that if $\sigma=f(\gamma) $ then  
$$ \sigma(s)\subset f( \gamma(r))  \subset  \sigma(t).$$ 
\end{lemma} 

\begin{proof}
Without loss of generality we may assume 
$\diam(\gamma) = \diam(\sigma) =1$.
Taking the metric $\rho =1/r$ on $\gamma(r)$ we see that
the modulus of the path family connecting $\gamma$
to $\partial \gamma(r)$ is bounded above by 
$$ \frac {\area(\gamma(r))}{r^2} \leq \frac {\pi(1+r)^2}{r^2}
   = \pi ( 1+ \frac 1{r^2}).$$
Hence the modulus of the path family separating 
$\gamma$ and $\partial \gamma(r)$ is bounded below by the 
reciprocal of this upper bound.
Thus the  $f$-image of this family therefore also has modulus 
bounded below (by quasi-invariance) and thus  the 
distance between $\sigma$ 
and $f(\partial \gamma(r))$ is  bounded 
below by a constant $s$  times $\diam (\sigma)$. This gives 
the left-hand inclusion of the lemma.

The other inclusion is easier. If $t$ is large then the 
modulus of the path family surrounding $\sigma$ in 
$\sigma(t)$ is also large and hence its pre-image under 
$f$ is also large. This means that the pre-image contains a large 
round annulus that surround $\gamma$ and hence contains 
$\gamma(r)$ if $t$ is large enough compared to $r$ (depending 
on $K$).  
\end{proof}

We shall also use the following well known result of
Teichm{\"u}ller, Wittich,  Belinski\u\i\  and Lehto
(e.g., 
\cite{MR869798}, 
Theorem 7.3.1 of 
\cite{MR2435270},  
 \cite{MR0125963},  
\cite{Teichmuller}).

\begin{thm} \label{TWB thm} 
Suppose $\varphi: \complex \to \complex$ is  $K$-quasiconformal 
with dilatation $\mu$ and 
$$ \iint_{|z|>R} |\mu(z)| \frac {dxdy}{|z|^2} < \infty,$$
for some $R < \infty$. Then there is a non-zero, finite 
complex constant $A$ so that $\varphi(z)/Az \to 1$ as 
$|z| \to \infty$.
\end{thm}

%-----------------------------------------------------------------
\section{Quasiconformal folding }   \label{QC folding} 

In this section,
 we review notation and results from \cite{Bishop-classS}.
Recall that  $\classS_{2,0} \subset \classS$ is
 the sub-collection of 
 Speiser class functions that have $2$ critical
 values and no finite asymptotic values. 
We  will start   by describing 
how an element of $\classS_{2,0}$ gives rise to a locally 
finite,  infinite planar tree; we then describe how to start with 
such a  tree (satisfying some geometric 
regularity conditions) and obtain an element of $\classS_{2,0}$.
This construction is the  main 
result of \cite{Bishop-classS}, and contains much of 
the work needed to prove Theorem \ref{App Omega}.

Suppose $f \in \classS_{2,0}$ and that the critical values 
of $f$ are exactly $\{ -1, 1\}$. Let $T = f^{-1}([-1,1])$. 
Let $U = \complex \setminus [-1,1]$ and let $\Omega= f^{-1}(U)$.
Then each component of $\Omega$ is simply connected and 
$f$ acts as a covering map from each component of $\Omega$
to $U$. The boundary of $\Omega$ is an infinite tree where 
the vertices are the  pre-images of $\{ -1,1\}$.
%We will call this a Speiser tree.
For each  connected component  of $\Omega$ there is 
  a  conformal map 
$\tau $ to $\rhp$  so that $f = \cosh \circ \, \tau$. 
The edges of $\partial \Omega$ are mapped to intervals 
of length $\pi$ on $\partial \rhp$.
See Figure \ref{Tracts5}.

\begin{figure}[htb]
\centerline{
\includegraphics[height=2.5in]{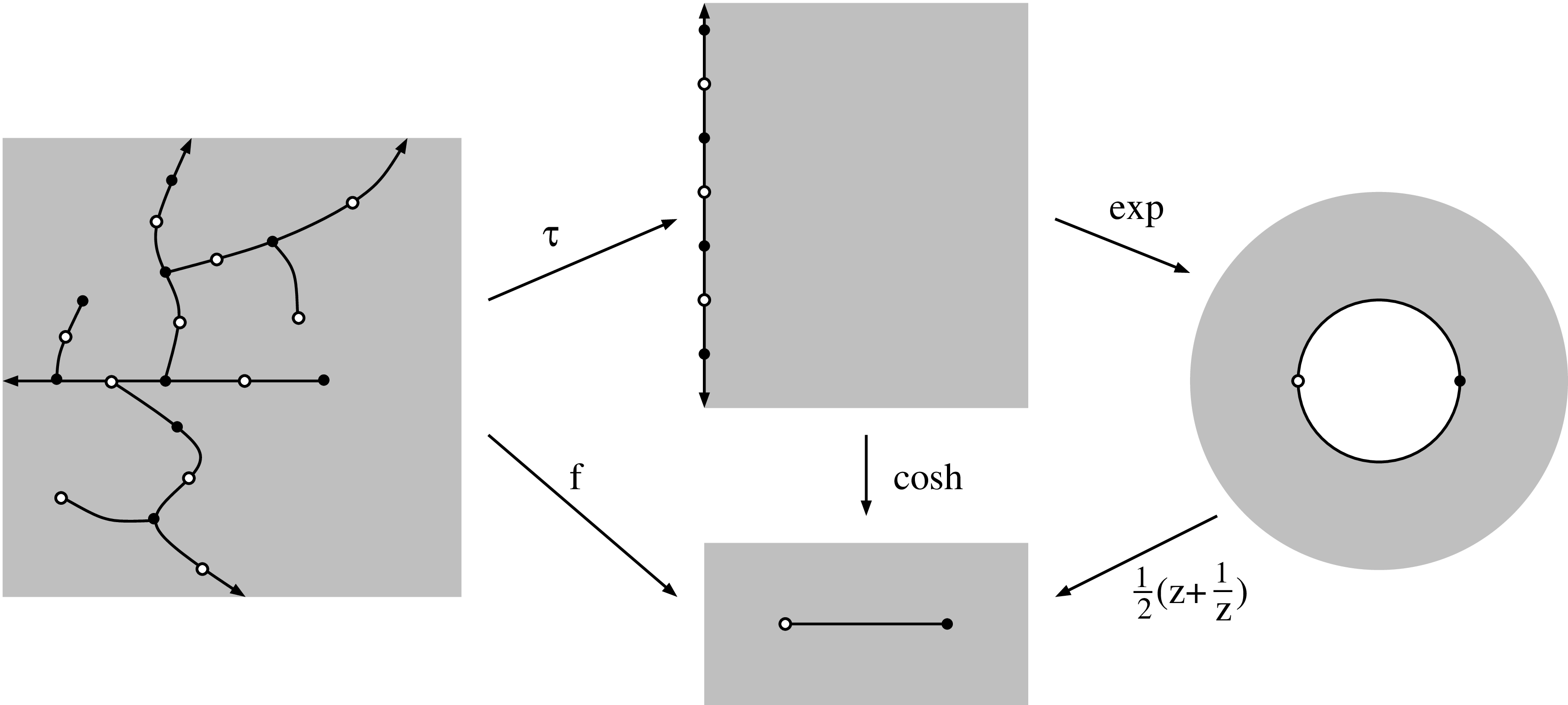}
}
\caption{ \label{Tracts5}
A function with two critical values at $\{ -1, 1\}$ and 
no finite asymptotic values. $T=f^{-1}([-1,1])$ is a tree
with vertices  mapping to  $\pm 1$ (shown as black and white dots).
 $\tau$ is a conformal map 
from each complementary component of $T$ to the right half-plane
and $f = \cosh \circ \tau$.
}
\end{figure}

Given $r >0$ and an edge $e$ on $\partial \Omega$
 we define a neighborhood 
$$ e(r) =  \{ z: \dist(z, e) <  r \cdot\diam(e) \} ,$$
and define a neighborhood of $T=\partial \Omega$ by taking 
the union over all edges. This neighborhood will 
be denoted $T(r)$.
See Figure \ref{Nbhd3}.

\begin{figure}[htb]
\centerline{
\includegraphics[height=2.5in]{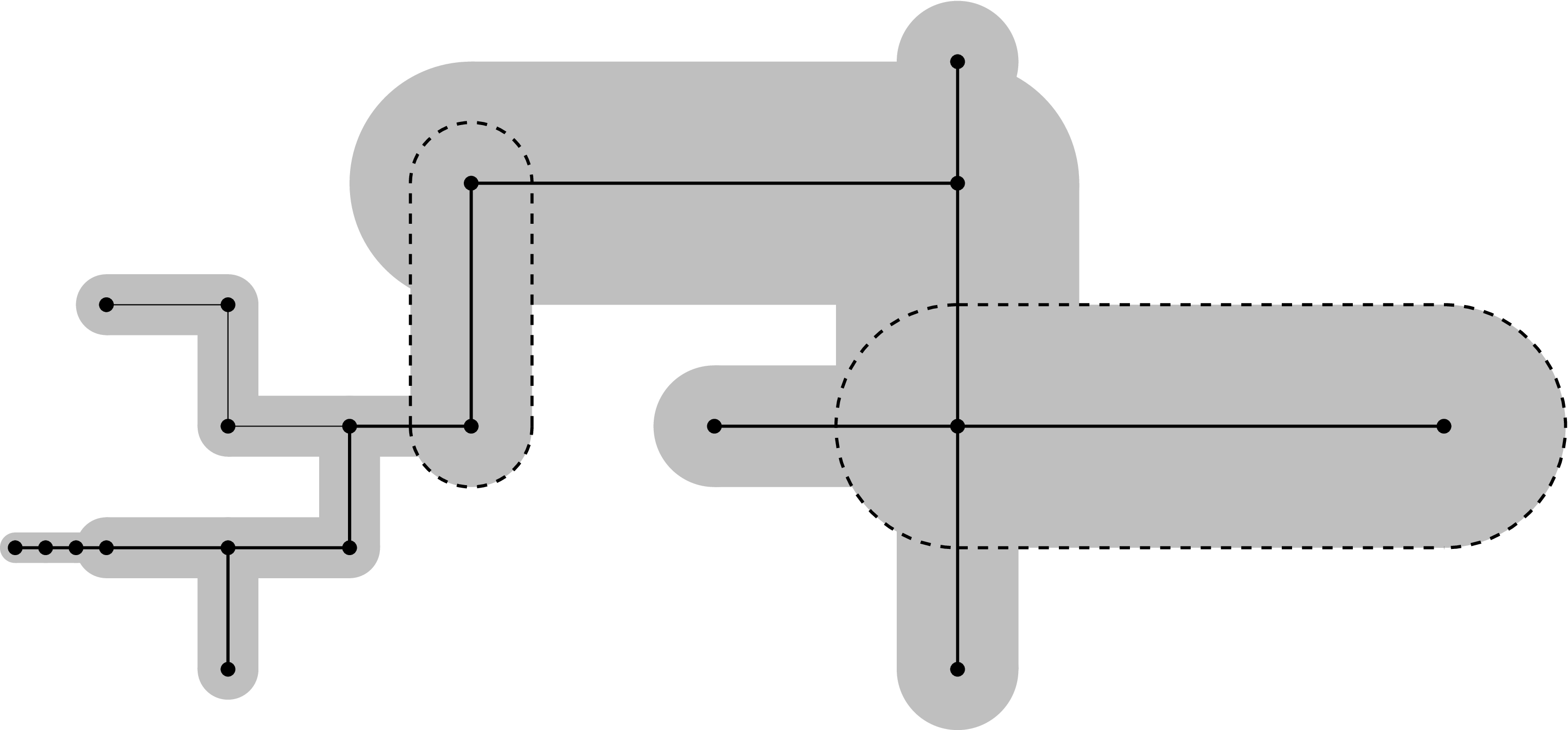}
}
\caption{ \label{Nbhd3}
The neighborhood $T(r)$ of a tree (a finite tree is shown, 
but the definition also makes sense for infinite trees and graphs).
The dashed regions show $e(r)$ for two edges.
}
\end{figure}

Now suppose we start with an 
 infinite planar  tree $T$  and 
a holomorphic map $\tau: \Omega \to \rhp$ where 
$\Omega = \complex \setminus T$ and  $\tau$ is 
conformal from each connected component $\Omega_j$ of 
$\Omega$ to $\rhp$. We  want to 
construct an $f \in {\classS}_{2,0}$
 so that $T $ approximates   $ f^{-1}([-1,1])$ and
$f $ approximates $ e^\tau$ away from $T$.
There are two basic conditions that we impose.

{\bf  (I) Bounded geometry:} this holds for $T$ if:
\begin{enumerate}
\item every edge is $C^2$ with uniform bounds on the derivatives.
\item edges meet at angles uniformly bounded away from zero.
\item any two adjacent edges have uniformly comparable lengths and their 
       union is uniformly  quasiconvex.
\item non-adjacent edges $e,f$ satisfy $\dist(e,f)/\diam(e)  
> \epsilon$ with a uniform $\epsilon >0$.
\end{enumerate} 
Here ``quasi-convex''  means that  the  arc-length distance 
between two points  $x,y$ on the curve is $O(|x-y|)$.
Note that condition (2) implies that the vertex degrees
of $T$ are uniformly bounded. 
A very  useful alternate version of (4)  comes 
from Lemma \ref{separation lemma}:
(4)  holds iff any two non-adjacent edges of $T$ are 
separated by a path family with modulus bounded 
uniformly away from zero. Because of the conformal invariance 
of modulus,  this allows us to easily verify that
under certain conditions,  conformal 
images of bounded geometry trees still have bounded geometry.
See Section \ref{build sec}. 

Later, we will also consider a bounded geometry ``forest'' 
$G$ that
is a  disjoint union of bounded geometry trees, 
where (1)-(3) hold for all edges in the forest and 
(4) holds  for all pairs of non-adjacent edges in $G$
(either from the same or from  different  
components of $G$).

Each edge $e$ of the planar tree $T$ has two sides and each 
side may be considered as a boundary arc of one 
of the complementary components $\Omega_j$  of $T$ (possibly 
both sides belong to the same component).
Conversely, the boundary of each component $\Omega_j$
is partitioned into arcs by the sides of the tree 
$T$.
 We say that two sides of $T$ are  adjacent 
if they are sides of adjacent edges of $T$ 
that are on 
the boundary of the same complementary component $\Omega_j$
and  the two sides 
correspond to adjacent intervals after conformally 
mapping $\Omega_j$ to $\rhp$. 
Two sides of $T$ can also be adjacent if they are 
opposite sides of a single edge of $T$ that 
has an endpoint of degree 1.

 When $\Omega_j$ is mapped to $\rhp$ by $\tau_j$ 
the sides of $T$  map to intervals on $\partial \rhp$.
The Euclidean length of the image 
interval is called the $\tau$-length of the corresponding 
side of $T$.  The collection of resulting intervals on 
$\partial \rhp$ form a partition, denoted ${\cal P}_j$,
 of this line.

For us, a partition of a line is a locally finite 
collection of disjoint 
open intervals whose closures cover the whole line. The 
endpoints  of the partition intervals form a countable, 
discrete set that accumulates only at $\infty$. We say that 
a partition has {\bf bounded geometry } if adjacent 
elements (i.e., partition intervals that share an endpoint)
have comparable lengths with a constant that is independent
of the intervals.  The bounded geometry constant of 
the partition is the supremum $|I|/|J|$ over all adjacent 
pairs of intervals. Occasionally we will also consider 
bounded geometry partitions of  bounded open segments or arcs 
that are defined in  the same way (adjacent  partition 
intervals have comparable lengths and accumulate only  
at the endpoints). 

If  the infinite tree 
$T$ has bounded geometry  then the partitions ${\cal P}_j$
of $\partial \rhp$,
corresponding to each complementary component $\Omega_j$,
also have bounded geometry,
%(adjacent intervals have comparable lengths),
 with a constant depending only on the bounds in the definition 
of bounded geometry (see Lemma 4.1 of \cite{Bishop-classS}).
In other words, if $T$ has bounded geometry, then adjacent 
sides of $T$ have comparable $\tau$-length.
This fact is the main way that we utilize the bounded
geometry assumption.

{\bf (II) The $\tau$-length   lower bound:}
The second condition  we require
 is  that  every side of $T$ has  
$\tau$-length $\geq \pi$ 
 (but no upper bound is assumed).

An  apparently weaker  form of this  is 
to simply require that  for each complementary 
component $\Omega_j$ of $T$, there is a strictly positive 
 lower bound  $\epsilon_j >0$ for 
the length of every interval in the partition ${\cal P}_j$.
If this weaker condition  holds, then on each 
component $\Omega_j$ of the model domain $\Omega$,
we can replace $\tau_j$ by 
a positive multiple of itself, namely  $ (\pi/\epsilon_j)
 \cdot  \tau_j$. 
This is still a 
conformal map of $\Omega_j$ to the right half-plane 
but now every partition arc on $\partial \rhp$ has 
length $ \geq \pi$. Thus if  each tract has a
positive  $\tau$-length   lower bound, we can easily 
choose a new model function for which it satisfies 
the stronger $\geq \pi$ bound. Therefore, 
in most cases, we only need to check the weaker condition.
Note that having a  positive  $\tau$-length lower bound 
 is a geometric  property of each   tract in $\Omega$; 
if each tract has such a lower bound, then
having a positive lower bound that works simultaneously 
for all the  tract depends on  the particular choice of
model function $ F = e^\tau$.

The following is the main result from \cite{Bishop-classS}.

\begin{thm}  \label{Exists}
Suppose $(T, \tau)$ 
has bounded geometry and 
 every side of $T$ has  $\tau$-length at least  $\pi$.
 Then there is a $f \in \classS_{2,0}$,  $r >0$  
and a quasiconformal $\varphi : \reals^2 \to \reals^2$
 so that $f \circ \varphi = \cosh \circ
\tau$ off $T(r)$. In addition, 
\begin{enumerate}
\item   the map $\varphi$ is conformal off $T(r)$,
\item   the function $f$ has only critical values $\pm 1$,  and no 
       finite asymptotic values,
\item    the number $r$ and  and the maximal dilatation
      $K$ of $\varphi$
       only depend on the bounded geometry constants  of $T$,  
\item the  degree of any critical point of $f$ is bounded
      by $4D$, where $D$ is the maximum degree of
    vertices in $T$.
\end{enumerate}
\end{thm}

This result is essentially  Theorem 1.1 of \cite{Bishop-classS}.
The statement there includes conclusions (1), (2) and (3) 
explicitly. Conclusion (4) follows
  from the proof  given in \cite{Bishop-classS}; 
by construction, the degree of any  critical point of $f$
 equals the 
graph degree of a corresponding vertex $v$ of a
 tree $T'$ that is 
obtained by adding finite trees to the vertices of $T$. The
bounded geometry condition implies the vertices of $T$ have
uniformly bounded degree, $D$.
Each of the added trees has maximum vertex degree   $4$ and 
has degree $3$ at the vertex that is attached to $ v \in T$.
At most $\deg(v)$ such trees are added at $v$, so the 
new tree has maximum degree at most $4D$. This gives (4).
Since the tree $T$ that we construct in the proof 
of Theorem \ref{App Omega} will have maximum vertex 
degree $3$, this gives the ``12'' in part (2) of 
Theorem \ref{App Omega}.

Conclusion (4) (and hence the estimate in Theorem \ref{App Omega})
can be improved by making some alterations 
to the folding construction in \cite{Bishop-classS},
 but we only 
sketch the possibilities here; details will appear elsewhere.
 It is fairly easy to modify the 
construction  so that the
finite trees we add to $T$ 
have degree $1$ at the vertex that is attached 
to $ T$; this requires redrawing the 
diagrams in Figure 8 of \cite{Bishop-classS} with a new 
vertical segment at endpoints of the horizontal segment and 
changing some corresponding bookkeeping in the proof.
This  change lowers the bound in (4) above
from $4D$ to $2D$.  If we replace $\tau$ by 
a positive multiple of itself, we can do even better. 
Subdivide each edge of 
$T $ into three sub-edges, so as to give a bounded 
geometry tree $T'$ and rescale $\tau$ so that each 
new edge has $\tau$-length at least $\pi$. 
We then   further modify  the 
construction in \cite{Bishop-classS}  
by adding  two extra vertices
to the horizontal edges in Figure 8 of \cite{Bishop-classS}
and make corresponding changes to the bookkeeping. This 
results in a finite tree being added to every third vertex 
of $T'$ viewed from a complementary component, and we 
can easily  arrange all these to be the ``new'' 
vertices of $T'$, and these all have degree $2$. We can add 
at most two trees to any such vertex, making the degree 
at most $4$.
This makes the bound in (4) equal to $\max(D,4)$.

As noted above, in the proof of Theorem \ref{App Omega}, 
we apply Theorem \ref{Exists} to 
a tree $T$ with maximal degree $3$. Some of the 
complementary components of $T$ correspond to 
components of $\Omega(\rho)$, and on 
these components all the $\tau$-lengths equal $\pi$.
This implies 
the folding construction does not add any trees 
inside these components.
 The other complementary components of $T$ are ``new'' 
components and we are free to choose any
positive  multiple of $\tau$ we want on these 
components.
Therefore we may assume the  $\tau$-lengths in these 
new components are all large, and hence the sketched
 argument above applies: finite trees are adjoined 
to vertices of degree $2$ in $T$, and the added trees 
have maximum degree $4$ and degree $1$ at the vertex 
attached to $T$ (so the degree $2$ vertex becomes degree
at most $4$).
Thus making the appropriate changes to the folding construction
in \cite{Bishop-classS}
will give the upper bound $4$ in part (2)
 of Theorem \ref{App Omega}.

One could improve the bound $\max(D,4)$ to $\max(D,3)$ if 
the finite trees we add have maximum vertex degree $3$.
Figure 12 of \cite{Bishop-classS}  shows how degree $4$ 
vertices arise
when adjacent trees are attached to each other.
It seems plausible that this can be avoided, 
but requires more extensive modifications and 
needs to be verified.
We would also need to  add at most one 
tree to any vertex of $T'$ (the tree obtained from $T $
by splitting each edge into three edges).
However, this is easy to arrange 
 by always attaching trees one vertex 
to the left of a vertex of $T$, when viewed from the 
corresponding complementary component; this will attach 
at most one tree to each of the two vertices of $T'$ that 
lie on any edge of $T$.
Together, these improvements
 would give the bound $3$ in part (2) of
Theorem \ref{App Omega}.

%-----------------------------------------------------------------
\section{ Glueing trees using conformal maps} 
 \label{build sec}  

In this section, we describe a way to combine two or more 
bounded geometry trees to obtain a new 
bounded geometry tree.

It is convenient to introduce a stronger version of 
bounded geometry. We say that a Jordan arc $\gamma$ is 
{\bf $\epsilon$-analytic} if there is conformal map on  
$$ \gamma(\epsilon) = \{ z: \dist(z, \gamma) < \epsilon \cdot 
 \diam(\gamma)\}$$
that maps $\gamma$ to a line segment.
We call a  bounded geometry tree $T$ 
 {\bf uniformly analytic} if there 
is an $\epsilon >0$ so that  every edge of $T$
is $ \epsilon$-analytic. We say a vertex $v$ of $T$ is 
$\epsilon$-analytic if it has degree two and the union of the 
two edges meeting at $v$ form a single  $\epsilon$-analytic 
Jordan arc. Note that vertices of a uniformly analytic 
tree need not be analytic (the  edges may meet at
various angles), but that if we  add vertices 
to the edges of a uniformly analytic tree $T_1$ to  form a new 
bounded geometry tree $ T_2$, then all the new vertices 
are analytic with the same constant as $T_1$.
 A bounded geometry forest in 
which all the edges are uniformly analytic will be called
a {\bf uniformly analytic forest}.
An important  example of such a forest is $\partial \Omega(\rho)$, 
where $\Omega$ is a model domain and the vertices are 
the usual ones, $\tau^{-1}( \rho+i \pi  \integers)$.  
In this case, all the vertices are analytic as well (with 
a uniformly bounded constant).

\begin{lemma} \label{transfer half-plane}
Suppose $T$ is a uniformly analytic  forest and suppose $\Omega$ 
is a connected component of $\complex \setminus T$. 
Suppose $ W$ is either $\disk$ or 
$\rhp$ and that $\tau: \Omega \to W$ is conformal.
 Suppose that
$T_0 \subset \overline{W}$ is a uniformly analytic tree (all 
open edges of $T_0$ are in $W$, but some vertices may lie 
on the boundary of $W$).  
We assume $\tau^{-1}(T_0)$ is locally finite in $\complex$.
Suppose there is a $M < \infty$ so that for   every 
edge  $e$ of $T_0$ either  
\begin{enumerate}
\item the edge $e$ has hyperbolic diameter at most $M$ (we call these
  the internal edges of $T_0$), or  
\item the edge  $e$  has one endpoint $x = \tau(v) \in \partial W$,
     where $v$ is an analytic vertex of $T$  and 
$$  \frac 1M \diam(\tau(I \cup J)) \leq \diam(e) \leq M  \cdot
\diam(\tau(I\cup J )),$$
 where we take
Euclidean diameters, and $I$ and $J$ are the two edges 
of $T$ adjacent to $v$. We call such an edge $e$ 
 a boundary edge of $T_0$.
\end{enumerate} 
Then $T'=T \cup \tau^{-1} (T_0)$ is a uniformly analytic  forest.
The constants for $T'$ depend only on the bounded geometry 
and uniform analyticity constants for $T$ and $T_0$. 
If a component of $W \setminus T_0$ satisfies a 
positive $\tau$-length lower bound, the same bound is satisfied 
by the image of this component under $\tau^{-1}$.
\end{lemma} 

\begin{proof}
The Koebe distortion theorem easily  implies that conditions 
(1)-(4) in the definition of bounded geometry are transferred 
from $T_0$ to $\tau^{-1}(T_0)$ for all  pairs of 
internal edges in $T_0$.
Similarly,  the images of all internal edges are clearly 
uniformly analytic. Moreover,  because of our 
assumption on the hyperbolic diameters, 
each internal edge  $e$ of $T_0$ 
is separated from $\partial W$ by a path family 
in $W$  with 
modulus bounded uniformly away from zero.  By conformal
invariance  of modulus, this also holds for   $\tau^{-1}(e)$ and 
$\partial \Omega$ and hence (4) holds whenever one edge 
corresponds to an internal edge  of $T_0$ and the 
other is an edge of $T$. A similar argument works 
for an internal edge of $T_0$ and a non-adjacent boundary edge
of $T_0$. 

Each boundary  edge $e$ of $T_0$ 
has an endpoint $x$ corresponding to an analytic vertex 
$v$  of  $T$ and by Schwarz reflection
 the map $\tau^{-1}$ extends analytically  
a uniform neighborhood of $ S=\tau(I\cup J)$ where $I, J$ 
are the edges of $T$ adjacent to $v$. Since $\diam(S) 
\gtrsim \diam(I \cup J)$, this implies $\tau^{-1}(e)$ is 
uniformly analytic. Moreover, there is  an $\epsilon \cdot 
\diam(e)$ neighborhood of $e$ where $\tau^{-1}$ extends
to be conformal and whose image hits $I$ and $J$, but 
no other edges of $T$. This (and the Koebe distortion 
theorem) implies the separation property (4) holds for 
$\tau^{-1}(e)$.
The final statement holds simply because 
having a positive lower bound for  $\tau$-lengths 
is  conformally invariant by definition. 
\end{proof}

%-----------------------------------------------------------------
\section{ Proof of Theorem  \ref{App Omega}: part 1,  bounded geometry}  
 \label{bg sec} 

It is stated in \cite{Bishop-classS} that Theorem \ref{Exists} 
reduces constructing functions in $\classS_{2,0}$ to ``drawing a picture''
of the correct tree.
One then has to verify that the tree has bounded geometry  and 
the complementary components each satisfy a positive  $\tau$-length 
lower bound.  
This is exactly what we will do to prove Theorem \ref{App Omega}.
In this section we connect  the various components 
of $\Gamma=\partial \Omega(\rho)$ to form a bounded geometry tree $T_1$. 
If $\partial \Omega(\rho)$ has  $N < \infty$ 
components, then the tree $T_1$ will have $2N$ complementary 
components;  
$N$ of these are the original 
components of $\Omega(\rho)$  and the other  $N$  are 
subdomains of $W = \complex \setminus \overline{\Omega}$.
When $N$ is finite, it is easy to make the connections if 
we are willing to allow the bounded geometry constant to 
grow. However, we shall give a more intricate construction 
that can  also deal with infinitely many components and  
gives uniformly bounded  geometry.

The new components might not satisfy a  $\tau$-length  lower 
bound condition, 
but we shall fix this in the next section by a simple
trick that  subdivides each of these new component into 
 infinitely many  components, 
each with a positive  $\tau$-length lower bound.
Later, in Section \ref{proof single comp}, we  will show 
how to replace each component by a single subdomain that 
has the desired  $\tau$-length lower bound.
See Figure \ref{FillIn}.

\begin{figure}[htb]
\centerline{
\includegraphics[height=1.2in]{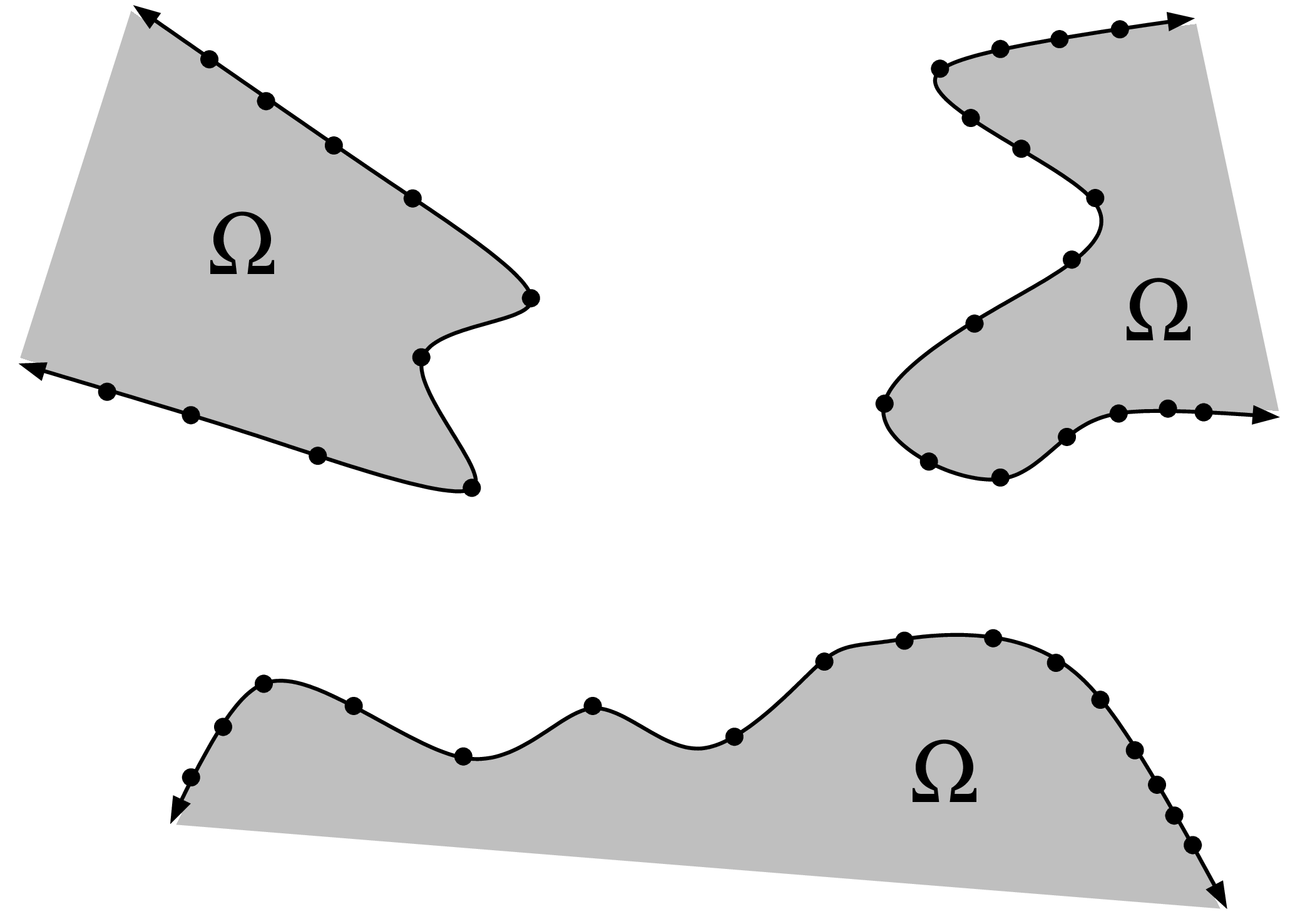} 
$\hphantom{xxx}$
\includegraphics[height=1.2in]{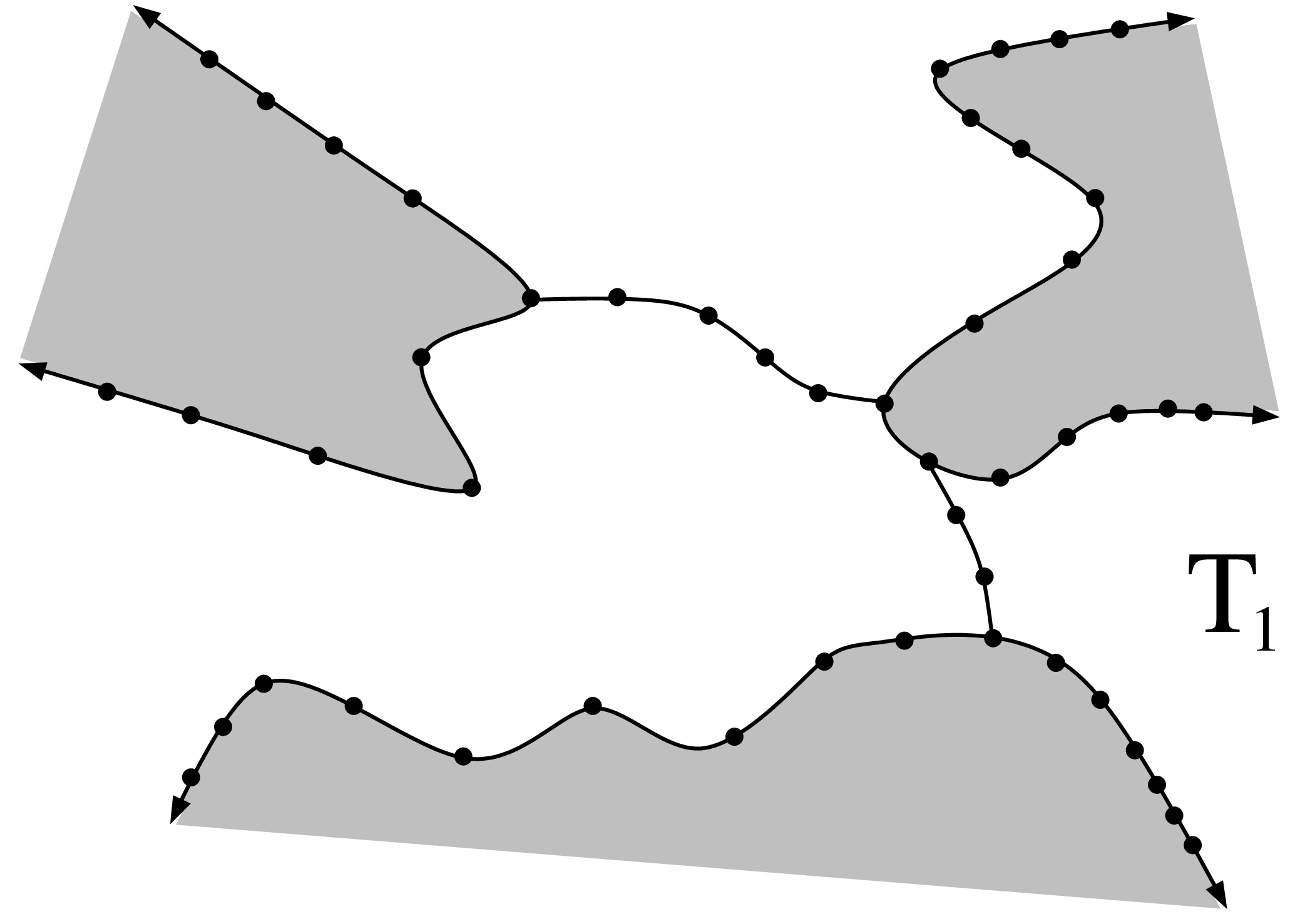}
$\hphantom{xxx}$
\includegraphics[height=1.2in]{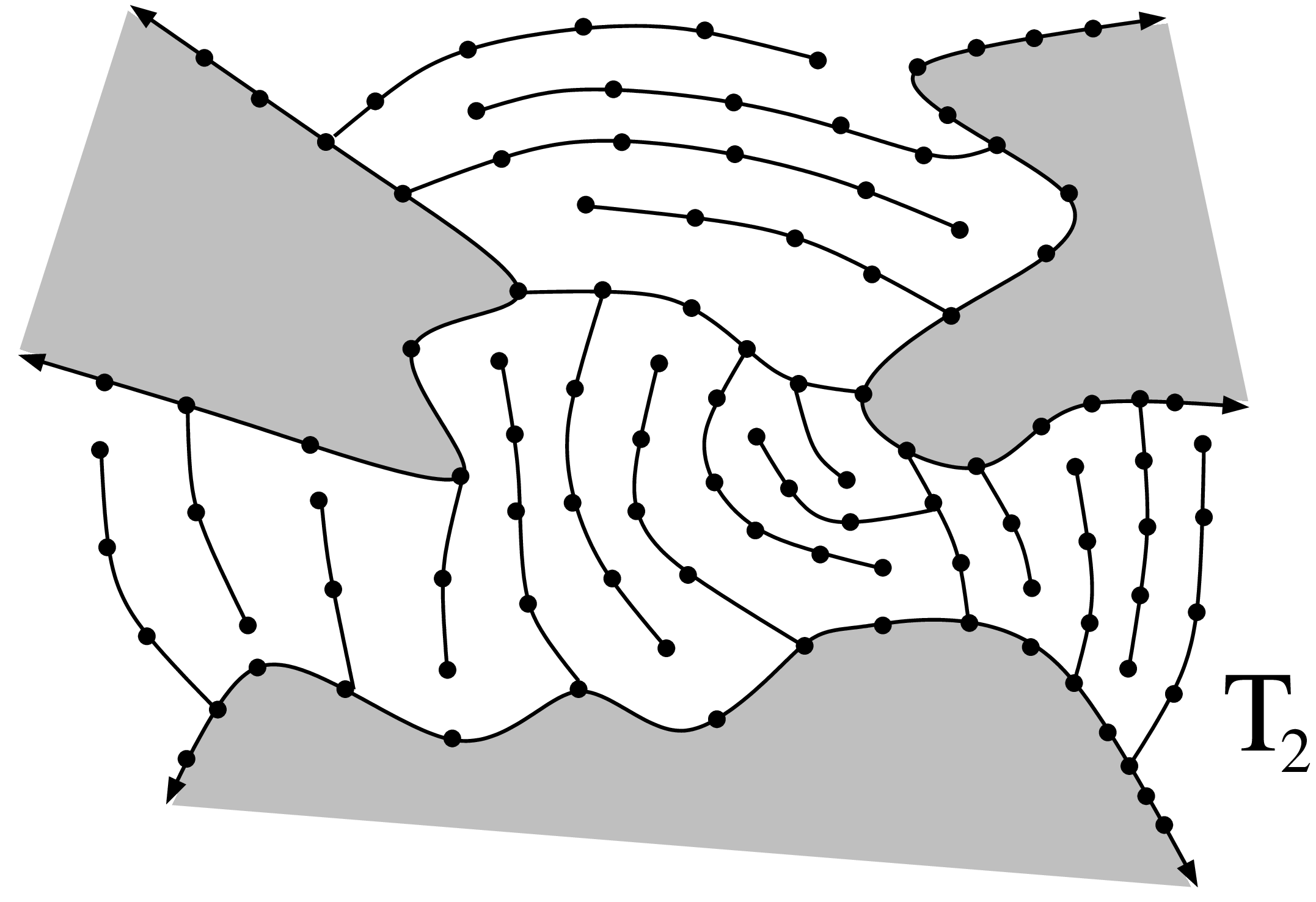}
}
\caption{ \label{FillIn}
The idea of the proof of Theorem \ref{App Omega} is
to reduce it to Theorem \ref{Exists} by  
joining the components of $\Gamma = \partial \Omega(\rho)$
to form a bounded geometry tree $T_1$ and then add extra edges 
to  the ``new''  components make a tree $T_2$ so 
that  the $\tau$-length   lower bound holds.
}
\end{figure}

We should note that the construction of $T_1$ and $T_2$ 
given in this paper was chosen for its generality, 
but it  might not be the most elegant or 
efficient choice for a particular $\Omega$ that  
arises in some application. Very likely, the geometry 
of the model domain will suggest a natural way of 
connecting the  different components of $\partial \Omega(\rho)$,
while satisfying the bounded geometry and 
 $\tau$-length   conditions.
Theorems \ref{App Omega} and \ref{details}
simply ensure that there is always  at least one way  
to accomplish this.

Now we start the construction of $T_1$.
Let  $ \Gamma= \partial \Omega(\rho)$.
This is a union of 
unbounded, analytic Jordan curves  and each curve comes with 
a set of marked points (or vertices) defined by ${\rm Im}(\tau(z)) 
\in  \pi i \integers$ (recall that this is called a 
conformal partition of $\partial \Omega(\rho)$).
 $\Gamma$ is a uniformly analytic forest
and every vertex is  analytic with  a uniform constant.

Let $W = \complex \setminus \overline{\Omega(\rho)}$.
This is a proper simply connected domain in the plane, so
by the Riemann mapping theorem there is  a conformal map 
$ \Psi: W \to \disk $. 
Each curve  $\Gamma_j=\partial \Omega_j(\rho) $
 maps to an  open arc  $I_j \subset \circle$ under $\Psi$.
We let $E = \circle \setminus \cup_j I_j$; this compact set 
corresponds to $\infty$ under $\Psi^{-1}$, hence it 
has zero Lebesgue length (even stronger, it has zero logarithmic 
capacity, but we won't need this).
The  partition   of  $\Gamma_j = \partial \Omega_j(\rho)$ 
with endpoints $\tau_j^{-1}(i \pi \integers)$
corresponds via $\Psi$ to a partition  of  $I_j$.
Because  $\partial \Omega(\rho)$ is a bounded 
geometry forest (with constant depending  only 
on $\rho$), 
adjacent intervals in the partition of $I_j$ 
have comparable lengths with a fixed 
constant,  depending  only on $\rho$.
 In particular, we can choose  
a point $v_j \in I_j$ so that the distances from 
$v_j$ to each endpoint of $I_j$ are comparable  to 
each other (just take 
an endpoint of a partition interval that contains the 
actual center of $I_j$). We call 
 $v_j$  the ``approximate center'' of $I_j$. 
The main objective of this section is to prove: 

\begin{lemma}  \label{build T1} 
Suppose notation is as above, i.e., $T_0$ is the tree 
consisting of $ \Gamma=\partial \Omega(\rho)$ with vertices 
given by the conformal partition on each component of 
$\partial \Omega(\rho)$.
There is a bounded geometry tree $T_1$   that contains
$T_0$, so that: 
\begin{enumerate}
\item 
  All new edges are in $W = \complex \setminus 
\overline{\Omega(\rho)}$.   
\item  The vertices of $T_1$ on $\partial \Omega(\rho)$ 
are exactly the vertices of $T_0$ (no new vertices are 
added  on $ \partial \Omega(\rho)$). 
\item  If $\Omega(\rho)$ has $N < \infty $ components then $ \complex 
\setminus T_1$ has  $2N$ connected components, $N $ of which are the 
connected components of $\Omega(\rho)$ and $N$ are subdomains of $W$.  
\end{enumerate}
\end{lemma}

Using Lemma \ref{transfer half-plane}, the 
proof of Lemma \ref{build T1}  reduces to the 
following construction on the disk.

\begin{lemma} \label{Build T1 in disk}
Suppose $E \subset \circle$ is closed and has length zero, 
$ \circle \setminus E = 
\cup_j I_j $ and $ v_j \in I_j$ are the approximate centers, 
as above. 
Then there is a tree $T$ in $\disk$ so that:
\begin{enumerate}
\item  The tree $T$ has bounded geometry. In fact, $T$ is uniformly 
  analytic and every edge is either a line segment or a 
  circular arc. The maximum vertex degree is $3$.
\item   For each $v_j$ there is a boundary edge of $T$ that has 
   $v_j$ as a common endpoint. The length of this edge is comparable  
   to the lengths of the two  partition arcs of $I_j$
     that   have $v_j$ as an endpoint.  This 
   edge makes an angle with $\circle$ that is bounded 
   uniformly away from zero. 
\item Every other arc of $T$ has uniformly bounded hyperbolic 
     diameter. 
\item The closure of every component of $\disk  \setminus T$  
      meets $E$ in exactly one point. In particular, if $E$ is 
      a finite set with $N$ elements, then  $\disk \setminus T$
      has $N$ components.
\end{enumerate} 
\end{lemma}

\begin{proof} 
Consider a Whitney decomposition of the disk, as illustrated in 
Figure \ref{CarlesonArcs}. The innermost part of the 
decomposition is  a central disk 
of radius $1/4$.   Outside of the central disk, 
the annulus $A_1=\{ \frac 14 < |z| < \frac 12\}$  is divided 
into eight equal sectors, the annulus $A_2=\{ \frac  12  < |z| < \frac 34 \}$
into sixteen sectors, and so on, as shown in Figure 
\ref{Carleson7}.  These sectors are called Whitney boxes.

\begin{figure}[htb]
\centerline{
\includegraphics[height=3.5in]{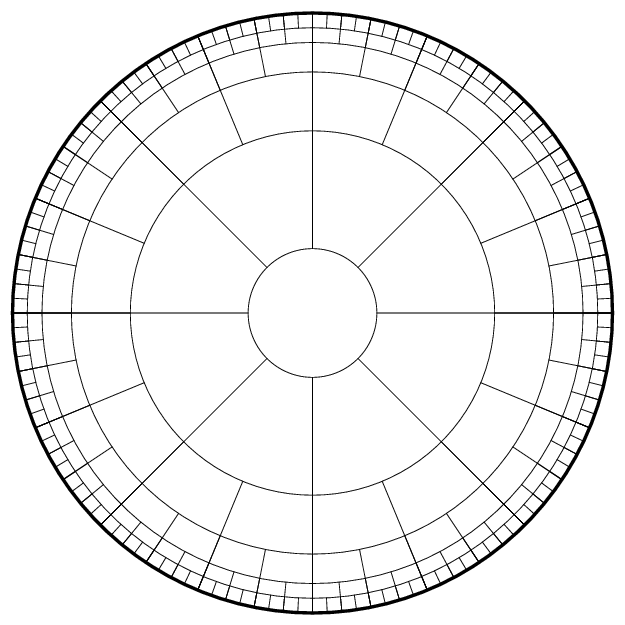}
}
\caption{ \label{Carleson7}
The  Whitney decomposition of the disk.
}
\end{figure}

%See Figure \ref{CarlesonArcs}.
Each  Whitney box has  two radial sides and two circular
arc sides concentric with the origin. The circular
 arc closer to the 
origin is called the top of the box and the arc further 
from the origin is called the bottom. Each bottom 
arc is divided into two pieces by the tops of the 
Whitney boxes below it (``below'' means between the 
given box and the unit circle).  We call these the left and right
sides of the bottom arc (left is the one   further clockwise).
The sides and bottoms of Whitney boxes we will call the 
Whitney edges, their endpoints we call Whitney vertices. The 
union of these edges and vertices forms an infinite graph 
in $\disk$ which we call the Whitney graph.
The radial projection of a  closed Whitney box $B$ onto the unit 
circle, $\circle$, is a closed  arc that we denote 
$B^*$ (this is sometimes called the ``shadow'' of $B$, 
thinking of a light source at the origin). 
The union of a closed  Whitney box $B$ and all the
closed  Whitney boxes
$B'$ so that $(B')^* \subset B^*$ is called the 
Carleson square with base $I = B^*$.

Each point on the unit circle can be connected to the central 
disk by a path  in the Whitney graph  that 
moves towards the origin whenever possible and moves 
counterclockwise otherwise.
 See Figure \ref{CarlesonArcs}.

\begin{figure}[htb]
\centerline{
\includegraphics[height=2.5in]{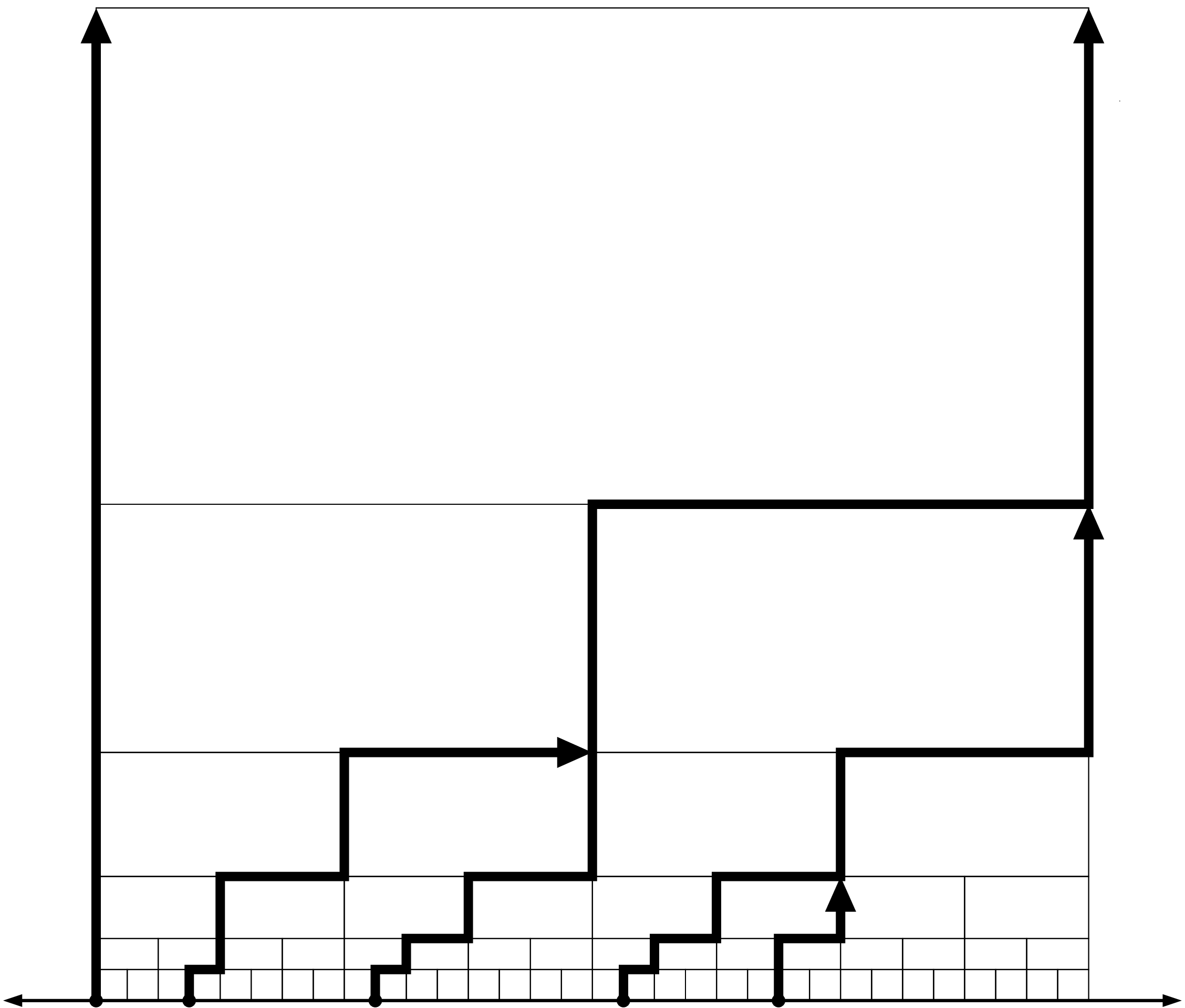}
}
\caption{ \label{CarlesonArcs}
The paths from the boundary to the central disk 
described in the text. 
Any boundary  point  can be joined to the central 
disk by a path moving along edges of Whitney boxes:
move radially towards the origin whenever possible, 
 and  move
counterclockwise (right in the picture) otherwise.
}
\end{figure}

Note that such a path never contains the ``left-half'' of 
the bottom of any Whitney box (otherwise the path would 
have moved up the left radial side of the box).
For each arc $I_j \subset \circle$ we 
 connect the approximate center $v_j$ of $I_j$  to the central disk 
by such a path. The union of all such paths, together with 
the boundary of the central disk,  is a closed 
set and divides the disk into countably many simply 
connected subdomains $\{U_j\}$.   By removing  one of
the eight arcs that bounds the central disk, 
we join  the central disk to one of the domains $U_j$.
This makes every subdomain $U_j$ 
an infinite union of Whitney boxes;  a
finite union would contain a box closest to the unit 
circle and the bottom of this box would be on a path, 
which is impossible since the left side of the bottom 
can't be on any path.

Thus every subdomain $U_j$ has a boundary that hits $\circle$, 
and $ J_j = \partial U_j \cap \circle$ must be a closed interval; if 
$J_j$  is not connected, then  
there is a component of $ \disk \setminus  U_j $ that
 is separated from the central disk by $U_j$, but this 
is impossible 
by construction (points on the boundary of this component are
on a path that continues all the way to the central disk).

 The closed   interval $J_j$ must  
hit $E = \circle \setminus \cup_j  I_j$, otherwise 
two paths were generated in the same component  $I_j$ of $\circle 
\setminus E$, contrary to the construction. Also $J_j$ must 
hit $E$ in a single point, $x_j$, otherwise $U_j$ separates some 
component  $I_k$ of $\circle \setminus E$ from the central disk, 
contradicting the fact that the approximate center of $I_k$ 
is connected to the central disk. 

We would like to turn 
the curve  $\partial U_j \setminus E$ into a tree by using the
partition vertices on $\circle \cap \partial U_j$ and using the 
Whitney vertices on $\partial U_j$, but there  are   infinitely 
many  Whitney vertices on $\partial U_j$  that accumulate at 
each approximate center $v_j$.
To fix this,   recall that
$v_j$ is the endpoint of two  partition intervals of comparable 
length. Suppose $r$ is the length of the shorter of these two and let
$w_j \in \partial U_j \cap \disk$ be a  Whitney vertex on the  path starting 
at $v$  with distance from $\circle$ between $r$ and $r/2$. 
We call this the truncation point associated to $v_j$.
The path terminating at  $v_j $
lies in a cone  in $\disk$ with radial axis,  fixed angle 
and  vertex at $v_j$  and 
there are no other paths that hit the disk 
$D(v_j, |v_j-w_j|)$, so 
we can replace the part of this path between these points 
by the line segment $[v_j,w_j]$. 
 The length of
this segment is comparable to the distance between $v_j$ and 
its neighboring partition points and is also comparable to 
the adjacent segment in the path $\partial U_j $. 
Moreover, the angle between this 
segment and the unit circle is  uniformly bounded away from 
zero, so we obtain a bounded geometry tree (even 
uniformly analytic), as desired.
%This correction makes $\partial U_j$ is a  bounded geometry tree.
See Figure \ref{CarlesonReplace}.
\end{proof} 

\begin{figure}[htb]
\centerline{
\includegraphics[height=2.15in]{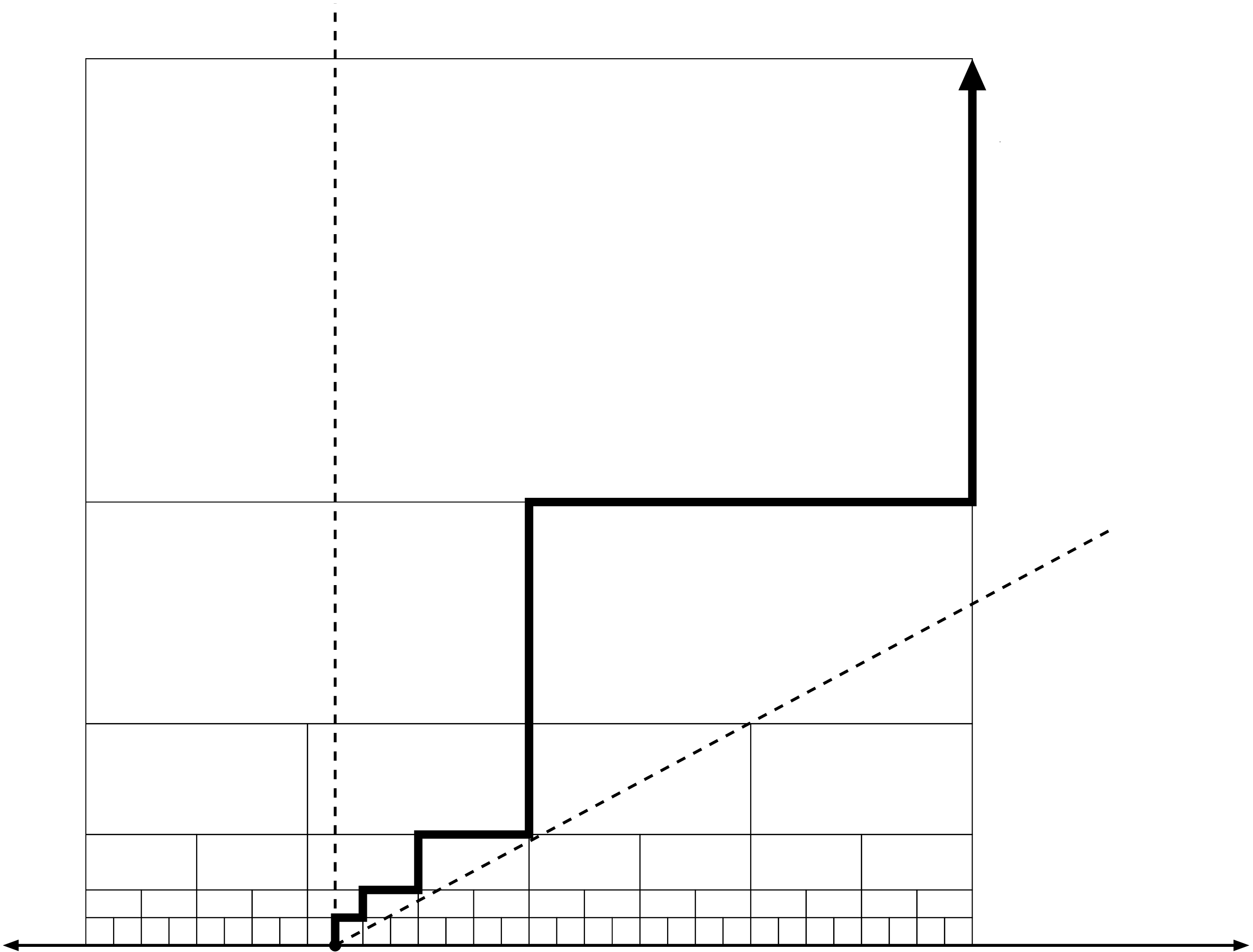}
$\hphantom{x}$
\includegraphics[height=2.15in]{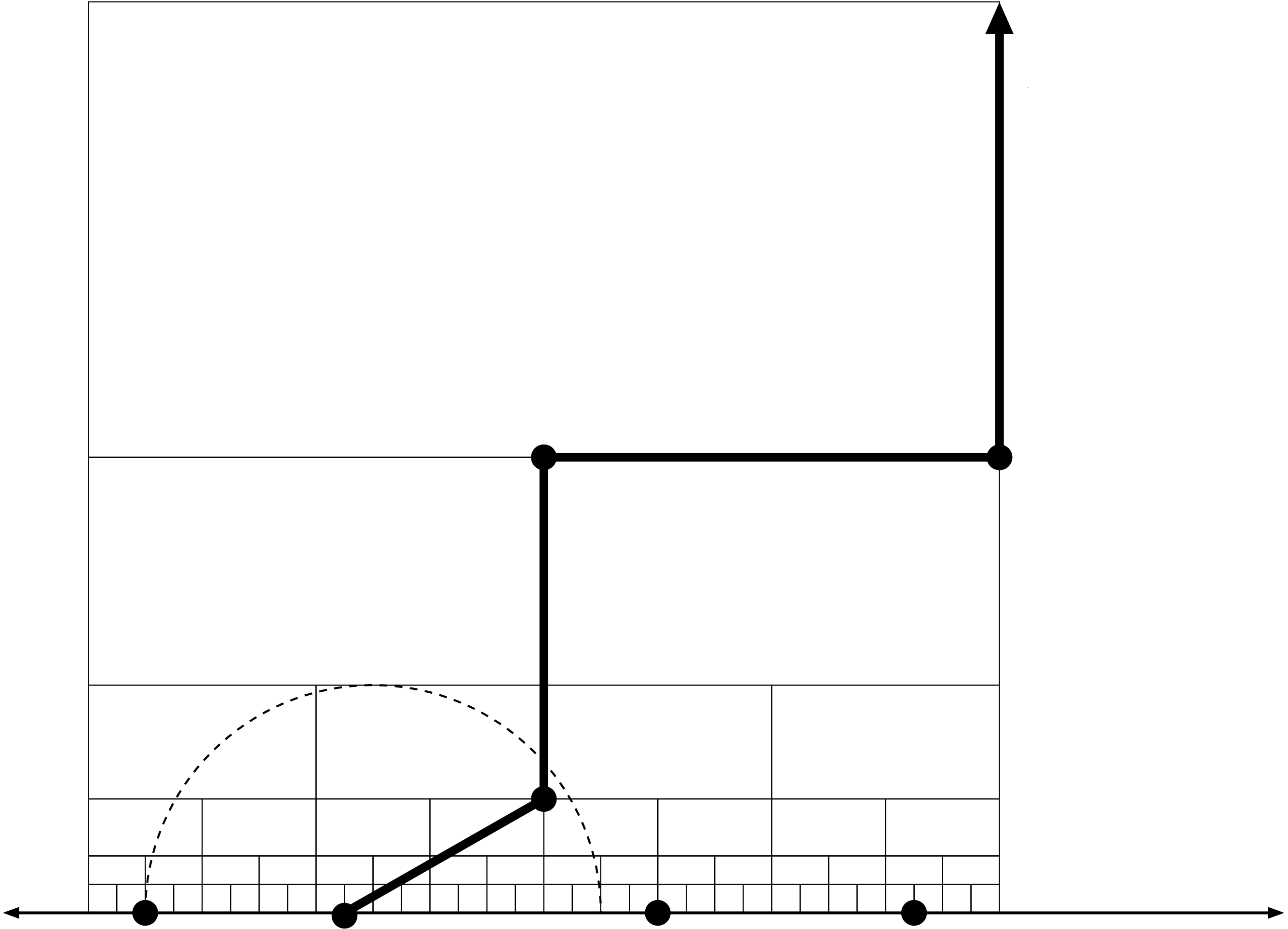}
}
\caption{ \label{CarlesonReplace}
The paths connecting  $v_j$ to the 
central disk approach  $v_j$  through a 
non-tangential cone near the boundary.
Thus if the  path to $v_j \in I_j$ is truncated at an appropriate 
scale and replaced by a line segment, this segment 
makes an angle with $\circle$ that is uniformly bounded 
away from zero.
}
\end{figure}

%-----------------------------------------------------------
\section{Proof of Theorem \ref{App Omega}: part 2, the $\tau$-length bound} 
 \label{tau sec} 

In the previous section we showed how to connect the components
of $\Gamma=\partial \Omega(\rho)$ into a single, connected, bounded 
geometry (even uniformly analytic) tree $T_1$.
In this section we modify the construction further to 
give  a positive $\tau$-length lower bound on each complementary 
component; as noted 
earlier, it is then easy to modify $\tau$ by multiplying it 
by a positive constant on each component of $\Omega(\rho)$ to 
make the lower bound $\pi$, as required in Theorem 
\ref{Exists}. 
We only have 
to  prove a lower bound on the ``new''  complementary components
that  we create;
the sides of the  components of $\Omega(\rho)$ 
have   $\tau$-length equal to $\pi$ by definition. 

Let  $T \subset \disk$ be the tree constructed 
in the Section \ref{bg sec}  and let  $T'$ be the tree 
we obtain by adding a vertex at the midpoint of 
each edge of $T$ (all the edges are segments or 
circular arcs so the midpoint is well defined). 
Note that all the ``new'' vertices  are analytic 
vertices with a uniform constant.
These vertices all have degree $2$ and later we will 
attach single edges to them, giving vertices of 
degree $3$; the resulting trees will have maximum 
degree $3$.

Let   $\{U_j \} = \disk \setminus  T$ 
be the complementary components of $T$. Let $x_j = \partial U_j 
\cap E$ be as defined in the previous section, 
 and let $\Phi_j : U_j  \to \rhp $ be conformal   with $\Phi_j(x_j) = \infty$.
The  vertices of $T'$ on $\partial U_j $ map to points on $\partial \rhp$; 
let 
 ${\cal P}_j$ be the bounded geometry 
 partition of $\partial \rhp$ induced by 
these points.  The ``new'' vertices of $T'$  induce 
a  bounded geometry partition ${\cal Q}_j$ whose 
endpoints are alternating endpoints of ${\cal P}_j$.  

\begin{lemma} \label{infinitely many}
With notation as above, fix $j$ and 
consider the union of  horizontal rays in 
$\rhp$ that start at each endpoint for the partition 
${\cal Q}_j$. Along each ray, add vertices that are 
equally spaced, with a spacing that is equal to the 
distance between that ray and the closer of the two 
adjacent rays (see Figure  \ref{AddLines}).
This is a uniformly analytic forest that we denote 
$G_j$.
Then  $T_2=T_1 \cup  \cup_j \Phi_j^{-1}(G_j)$ 
is a bounded geometry, uniformly analytic tree that satisfies 
a  positive $\tau$-length lower bound.
\end{lemma}

\begin{figure}[htb]
\centerline{
\includegraphics[height=2.5in]{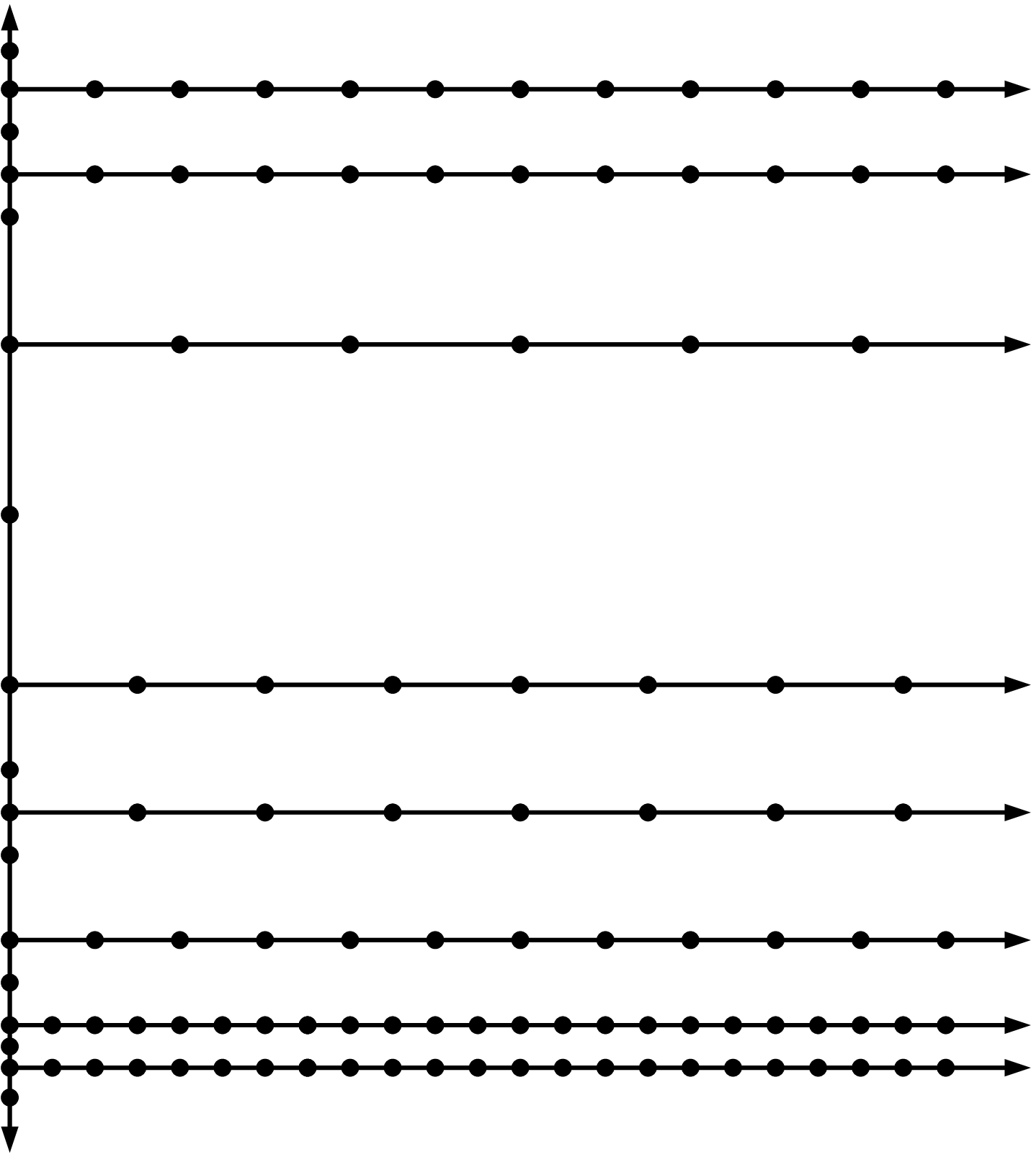}
}
\caption{ \label{AddLines}
Given a partition of $\partial \rhp$ where adjacent 
intervals have comparable lengths, add a horizontal 
ray in $\rhp$ at each partition point in $\partial \rhp$
 and place 
equally spaced  vertices on each ray, where the spacing equals 
the smaller width of the two adjacent half-strips.
It is easy to see that this gives a bounded geometry tree 
that satisfies a positive  $\tau$-length lower bound.
Note the vertices of the tree have maximum degree $3$.
}
\end{figure}

\begin{proof} 
It is obvious that $G_j$ has bounded geometry 
and is  uniformly analytic and that it satisfies all 
the other hypotheses of Lemma \ref{transfer half-plane}, 
so that 
$T_2 = T_1 \cup \cup_j \Phi_j^{-1}(G_j)$ 
is indeed  a bounded geometry, uniformly analytic tree.
To prove that  each connected component of $T_2$ 
satisfies a $\tau$-length lower bound, we simply note  
the conformal map of a half-strip to a half-plane has 
exponential growth (we can check this via an 
explicit formula involving $\sinh(z)$),
 so if the partition segments on the
boundary of the half-strip have Euclidean lengths 
bounded  below by a constant times the width of the 
strip, then the $\tau $-lengths  grow exponentially
with a uniform bound.  In particular,  the $\tau$-lengths 
are uniformly bounded away from zero.
\end{proof}

Verifying the $\tau$-length condition for a half-strip above
is simple because there is an explicit formula for
the conformal map to $\rhp$.  One could also use the 
more geometrical and more general 
Lemma \ref{tube bound}, which will be stated and proved later.

%-----------------------------------------------------------
\section{  Proof of Theorem \ref{App Omega}: part 3,  final details } 
\label{final step}

%Given a partition of $\partial \rhp$, 
The construction in the previous section 
created a  forest in $\rhp$  with  infinitely many complementary 
components, each of which satisfies a positive $\tau$-length 
lower bound. 
We now  apply  this construction to the partitions
$\{ {\cal Q}_j \}$ corresponding to the analytic vertices of 
the tree $T'$ constructed at the beginning of Section \ref{tau sec}.
Using Lemma \ref{transfer half-plane}, we can 
attach a conformal image of the forest created in 
Lemma \ref{infinitely many} to $T'$
to give a bounded geometry tree.  A positive 
 $\tau$-length 
lower bound holds automatically by the conformal invariance of 
this condition.  
As  before, we can  multiply  $\tau$ by positive constants 
on each component, so that the $\tau$-length
lower bound is $\pi$.
We then apply Theorem \ref{Exists} and 
multiply the resulting function by $e^\rho$ to get:

\begin{thm} \label{App Omega Cosh}
 Suppose  $\Omega$ is as in Theorem \ref{QR}.
 Then there is a $f \in \classS_{2,0}$
and a $K$-quasiconformal map $\phi$ of the plane so that
$ f \circ \phi = \cosh \circ \, {\tau}$   on  $\Omega$
 and $\phi$ is  conformal on $ \Omega \setminus T(r)$.
The constants $K, r < \infty $ depend on $\rho$
but are otherwise independent of $\Omega$ and $\tau$.
The function $f$  has  no finite asymptotic
values, exactly two critical values,
 $\pm e^\rho$, and   every critical
point has degree $\leq 12$.
\end{thm}

There are a few slight differences between this and 
Theorem \ref{App Omega}, but it is easy to deduce 
Theorem \ref{App Omega} from Theorem \ref{App Omega Cosh}
as follows.
 
First, Theorem \ref{App Omega Cosh} 
  uses $\cosh$ instead of $\exp$. However, these 
functions are almost the same in $\rhp$ away 
from the boundary.  Consider the 
map $z \to \frac 12 (z + \frac 1z)$; this is a 
conformal 
homeomorphism of $\{|z| >1\}$ to $ U=\complex \setminus [-1,1]$
and maps  the circle $C=\{|z| = e^\rho\}$ to some ellipse $E$. 
Define a quasiconformal map $\psi$  that equals the inverse
of this map  outside $E$ and extends it  diffeomorphically to the
interior. 
Since $\cosh = \frac 12 (e^z + e^{-z})$ we get $\exp(z) 
= \psi(\cosh(z))$ 
when  $|\exp(z)|> e^\rho$. Therefore if we use the 
measurable Riemann mapping theorem to find a quasiconformal 
$\varphi$ so that $F=\psi \circ f \circ \varphi$, 
this function  satisfies Theorem \ref{App Omega Cosh}
with $\cosh$ replaced by $\exp$.

Second, Theorem \ref{App Omega Cosh} only claims that $\varphi$ is 
conformal off $T(r)$ whereas Theorem \ref{App Omega} says it
is conformal on all of $\Omega(2\rho)$.  The first step in  verifying
this stronger condition is to prove:

\begin{lemma}
With notation as above, 
there is a $A < \infty$, depending only  on $\rho$ and $r$ so that 
$T(r) \cap \Omega(A  \cdot \rho)  = \emptyset$. 
\end{lemma} 

\begin{proof}
The proof is a modulus argument. The modulus of the 
path family in $\rhp+\rho$  separating 
 a  segment of length $\pi$  on 
$ \{ x =\rho\} $ from the vertical line $\{ x=A \rho\}$  is easily seen to 
increase to infinity as $A$ increases to infinity. 
Thus by conformal invariance of modulus and
 Lemma \ref{separation lemma} 
$$ 
\frac { \dist(I, \partial \Omega(A\rho))}{ \min(\diam(I), 
\diam (\partial \Omega (A \rho)))} = 
\frac { \dist(I, \partial \Omega(A\rho))}{ \diam(I)  } \to \infty 
$$ 
as $A \to \infty$. This proves the intersection
is empty if $A$ is large enough.

\end{proof} 

We can easily choose a quasiconformal map $H: \rhp \to \rhp$
 so that $H$ 
 \begin{enumerate}
\item  is the identity on $\{ 0 < x < \rho\}$, 
\item is of the form $(x,y) \to (ax+b, y)$  mapping
$\{ \rho < x < 2 \rho\}$ to $\{ \rho < x <  A\rho\}$, 
\item is a horizontal translation  from $\{ x > 2 \rho\}$
     to $\{ x >A \rho \}$.  
\end{enumerate}
 Defining  $G =  \tau_j^{-1} \circ H \circ \tau_j$ on each 
$\Omega_j$ and letting $G$ be the identity elsewhere gives 
a quasiconformal map of the plane to itself so that 
$  \tilde \varphi = \varphi \circ G$ is conformal off $\Omega( 2 \rho)$
and satisfies all the other conclusions of Theorem \ref{App Omega}.

Finally, the tree we have  explicitly 
constructed has maximal vertex degree $3$.
 Hence by the folding theorem (Theorem \ref{Exists}), the corresponding entire function will
have critical points of degree at most $12$.
(As noted at the end of Section \ref{QC folding}, 
this bound  can be 
improved to 4 by some modifications to the construction 
in \cite{Bishop-classS}.)

%------------------------------------------------------
\section{Two  estimates on $\tau$-length \label{tau estimate} }

In this section, we give two explicit estimates for 
$\tau$-lengths that we will use during the proof of 
Theorem \ref{details} in the next section. 

 Given two disjoint intervals $K, J$ on 
the real line, let $M(K,J) $ be the modulus of the 
path family in the upper half-plane, $\uhp$, that separates $K$ 
from $J$ (these are the paths in $\uhp$ with  two endpoints 
in $\reals \setminus (K \cup J)$, exactly one of which 
separates $K$ and $J$). This is the reciprocal of the 
modulus of the path family that joins $K$ and $J$ (paths
in $\uhp$ that have one endpoint in each of $K$ and $J$). 

The first result is helpful for domains that look 
like ``tubes'' built by attaching quadrilaterals
of bounded modulus end-to-end 
(e.g., as the half-strip is a union of 
squares joined end-to-end).

\begin{lemma} \label{tube bound}
Suppose % $I=[-1,1]$,
 $K=(-\infty,-1]$  and $\{ J_j\}$ is
a sequence of disjoint intervals in $[1, 
\infty)$ such that 
 $ M= \sup_j M(J_j, K )  < \infty$.  
Also assume the $\{J_j\}$ are in increasing order
(i.e., $J_{j+1}$ is to the right of $J_j$).
Then the lengths of $J_j$ grow 
exponentially in $j$; in particular, these 
lengths are uniformly bounded below by a constant
depending only on $M$.  
\end{lemma} 

\begin{proof}
Fix some $ J_j = [a_j, b_j]$ with $1 < a_j < b_j$.
 If $ b_j-a_j \leq \epsilon a_j$ for some $0 < \epsilon <1$, 
then $J_j$ is separated from $K$ by the annulus
$$ A = \{ z \in \uhp :  \epsilon a_j < |z-a_j| < a_j \}.$$
%See Figure \ref{ExpGrowth}.
Any path connecting different components of 
$\partial A \cap  \reals$ also separates $J$ and
$K$ so the modulus of the first family is 
a lower bound for the modulus of the second. However, 
this modulus is $ \frac 1{2 \pi} \log \frac 1\epsilon$, 
so $\epsilon \geq \epsilon_M = \exp(- 2 \pi M)$.
Hence $b_j> (1+\epsilon_M)a_j$. 
By induction $|J_j| = b_j-a_j \geq \epsilon_M (1+\epsilon_M)^{j-1} $,
as desired. 
\end{proof} 

by taking a ``tube domain'' 
and attach  ``rooms'' along the sides of the tube.

\begin{lemma} \label{add rooms}
Suppose $K=[s,t]$, $I=[x,y], J=[u,v]$ are intervals
 on the real line 
so that   $t  \leq  x  <  u$.  If  $M(J,K) \leq M(I,K)$, then 
$|I| \leq |J|$.
%(i.e., $J$ is longer than $I$).
\end{lemma}

\begin{proof}
We prove the contrapositive. Suppose  $|I| > |J|$.
After translating (if necessary) we may assume 
that $t=0$. Then dilate by $\lambda  = u/x > 1$. 
Note that  $K \subset \lambda K$ 
 and $J \subset \lambda I$ (strictly), so 
using the monotonicity  and conformal invariance of modulus, 
we deduce
$ M(J,K) >  M(\lambda I, \lambda K) = M(I,K).$
\end{proof}

%------------------------------------------------------
\section{Proof of  Theorem \ref{details} }
\label{proof single comp}

In this section, we improve Lemma \ref{infinitely many}
 by showing that instead of creating infinitely many 
complementary components, we can accomplish the 
same result using a single complementary component. 

\begin{lemma} \label{single component}
Suppose ${\cal Q} $ is a  bounded 
geometry partition of $\partial \rhp$ with 
constant $C$ 
(i.e., adjacent intervals have length ratio 
at most $C$). 
Then there is a bounded geometry, uniformly analytic 
forest  $T'' \subset \rhp$
which satisfies the hypotheses 
of Lemma \ref{transfer half-plane} and so that 
$W'' =  \rhp \setminus T''$  consists of a single
component  that  
 satisfies a positive  $\tau$-length lower bound.
The constants associated to $T''$ depend only on $C$.
Moreover, the boundary edges of $T$ meet $\partial \rhp$
exactly at the partition points of ${\cal Q}$.
\end{lemma} 

Given the lemma, we can complete the proof of Theorem \ref{details}
just as we finished the proof of Theorem \ref{App Omega} in 
Section \ref{final step}. Briefly, we had constructed a 
tree $T'$ that contained the analytic arcs $\partial \Omega(\rho)$
as well as arcs that connected the various components of 
$\partial \Omega(\rho)$. The complementary components of 
$T'$ consist of the  $N$ components of $\Omega(\rho)$ (which already 
satisfy the $\tau$-length condition because all their sides 
have $\tau$-length $\pi$ by definition)  and $N$ other 
components $\{U_j'\}$ (which might not satisfy the $\tau$-length 
condition; this is what we want to fix).
 The tree $T'$  is  uniformly analytic, and alternate 
vertices are analytic ($T'$ was obtained from an analytic 
tree $T$ by adding midpoints of edges).
 We then map each of the  $U_j'$  conformally to 
$\rhp$, take ${\cal Q}$  to be the image of the new analytic 
vertices and apply Lemma   \ref{single component} to this partition.
The resulting forest $T''$ is then mapped conformally back to 
$U_j'$ and attached to $T'$.  
Using Lemma \ref{transfer half-plane},  we see that the 
resulting tree has bounded geometry and satisfies a lower 
$\tau$-length condition. 
The rest of the proof of Theorem \ref{details} then 
exactly follows the proof of Theorem \ref{App Omega}.
Thus to prove Theorem \ref{details} 
it suffices to establish Lemma \ref{single component}.

%\begin{proof}[Proof of Lemma \ref{single component}]
We now start the proof of Lemma \ref{single component}.
Fix a partition ${\cal Q}$ of $\partial \rhp$.
Choose  a 
base interval $I_0$ in the partition. Without 
loss of generality we may assume $I_0 = [-i,i]$ 
and label the partition endpoints  $\{ z_j\} = \{ i x_j\} 
\subset \partial \rhp$ so that 
 $$ \dots  x_{-3} <  x_{-2} <  x_{-1} =-1 < x_1=1 <  x_2 <  x_3 \dots.$$
%as shown in Figure \ref{TauFill}.  
Note that the elements
of ${\cal Q}$ are labeled by $\integers$, but the endpoints 
are labeled by $\integers^* = \integers \setminus \{0\}$.
It will be convenient to define $k^* = k+1$ if $k >0$
and $k^*=k-1$ if $k<0$; for $k  \in \integers^*$, $k^*$
 is the integer adjacent to $k$, 
but farther from $0$. With this notation, we can write 
$I_k =(x_k, x_{k^*})$ without needing to have special cases 
for $k >0$ and $k <0$ (although we accept that an interval can 
be written as either $(a,b)$ or $(b,a)$).

  Define the 
central  region in $\rhp$  as the union of the  rectangle $[0, 1] 
\times [-1,1]$ and the sector  $\{x+iy \in \rhp: |y| < x\}$. 
This region is illustrated in the 
left part of Figure \ref{JoinTrap5B}. 
The boundary of the central region consists of the segment 
$I_0  \subset \partial \rhp$ and two  infinite paths  in $\rhp$ that 
we will call the upper and lower boundaries.

\begin{figure}[htb]
\centerline{
\includegraphics[height=2.8in]{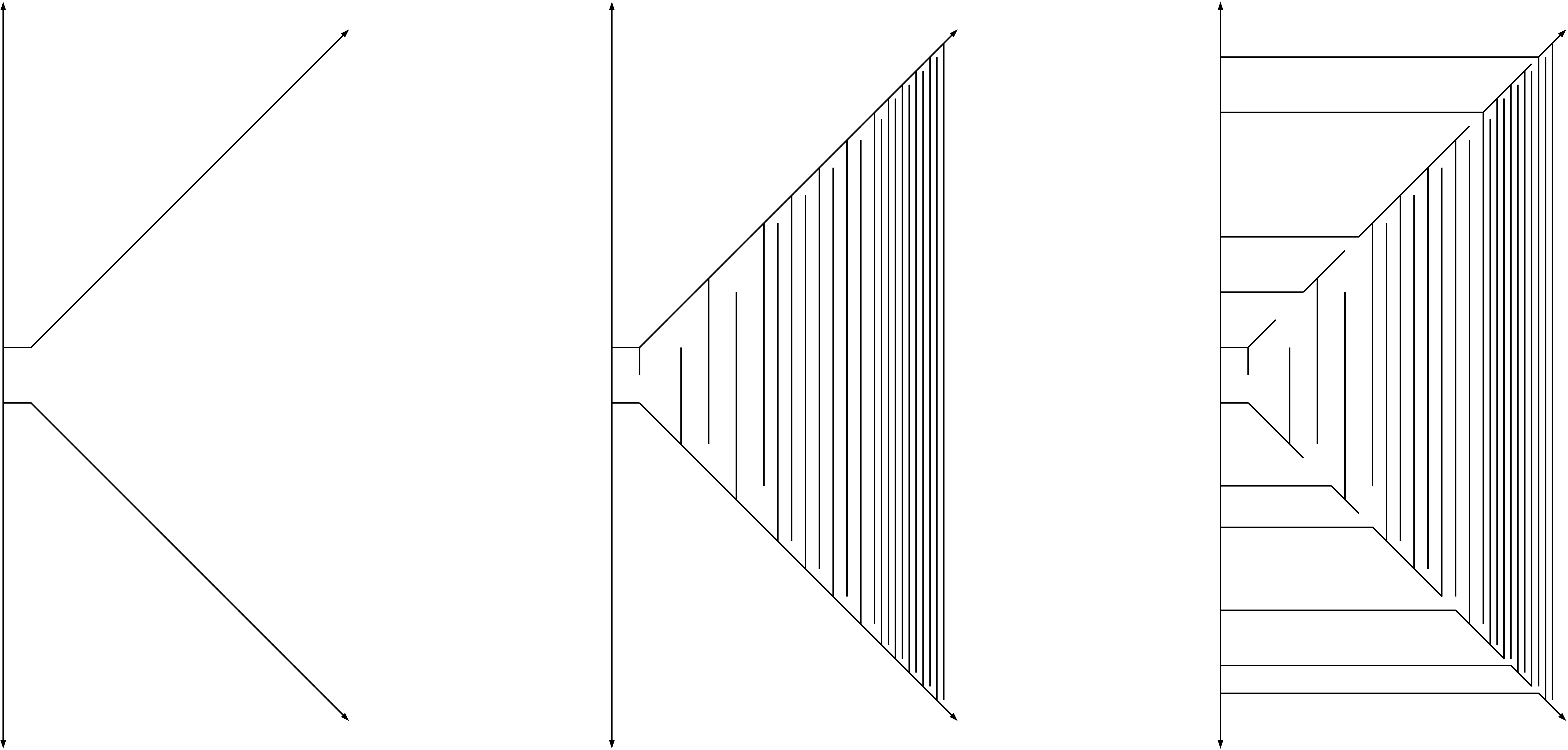}
}
\caption{ \label{JoinTrap5B}
This figure encapsulates the proof of Lemma \ref{single component}
by showing the relevant tree.
On the left is the central region, in the center is the 
central tube and on the right is the tree $T''$.
We can easily 
add vertices to make this a bounded geometry tree, and 
 we will use extremal length estimates to show $W''$ 
satisfies a positive  $\tau$-length lower bound.
}
\end{figure}

It would be very convenient for us if the partition 
${\cal Q}$ was symmetric with respect to the origin, 
i.e., $J_{-k} = - J_k$.  Since this need not be the 
case, we will build a new partition ${\cal Q}'$ that is symmetric, 
has bounded geometry  and is ``finer'' than ${\cal Q}$ in 
a certain sense.  More precisely:

\begin{lemma} \label{symmetric partition}
Given a bounded geometry partition ${\cal Q}$ of 
$\partial \rhp$  normalized 
as above,  and a real number $1 < M < \infty$, 
there is a bounded geometry partition ${\cal Q}'
= \{ I_j'\}_{\integers}$ of $\partial \rhp$
so that 
\begin{enumerate}
\item  the interval $I_0'=I_0$ is an element of ${\cal Q}'$, 
\item  the partition is symmetric, that is, $I_{-j}' = - I_j'$,
\item the length of $I_j'$ is a non-increasing function of $|j|$,
\item the length of every $I_j$ is an integer power of $2$,
\item \label{cond 5} if any interval $I' \in {\cal Q}' \setminus \{ I_0\}$ 
        intersects 
       an interval $I \in {\cal Q}$, then $|I'|< |I|/M$,
\item the bounded geometry constant of ${\cal Q}'$  is 
      bounded above depending only on $M$ and the bounded 
      geometry constant  $C$ of ${\cal Q}$. 
\end{enumerate}
\end{lemma}

\begin{proof}
Cover  $I_1$ by  a collection ${\cal D}_1$ of
closed  dyadic intervals
that all hit $I_1$ and that all have 
lengths strictly  less than $|I_1|/4M$ and greater or 
equal to $|I_1|/8M$ (since there is exactly one power of 
two  in this range, all the chosen intervals have the same 
length, call it $\ell_1$).  In general, suppose we have already covered 
$I_1 \cup \dots \cup I_{j-1}$ by a  collection ${\cal D}_{j-1}$
of closed dyadic intervals  such that
\begin{enumerate} 
\item the interiors are disjoint, 
\item every $J\in{\cal D}_{j-1}$ hits some $I_k$, $1 \leq k < j$, 
\item the lengths are non-increasing, 
 \item if $J \in {\cal D}_{j-1}$ hits 
$I_k$, $ 1 \leq k  < j$ then $|J|< |I_k|/M$.
\item adjacent dyadic intervals have comparable lengths.
\end{enumerate}
Let $\ell_{j-1}$ be the length of the last (rightmost) 
dyadic interval  in ${\cal D}_{j-1}$.
Cover $I_{j}$ by a collection ${\cal C}_j$ of 
 dyadic intervals all with  the same
length $\ell_j$,  where $\ell_j$ is the
integer  power of $2$ satisfying 
$|I_j|/8M \leq \ell_j <  |I_j|/4M$. 

First suppose  $\ell_j \leq \ell_{j-1}$.  Remove the 
last interval in ${\cal D}_{j-1}$ and replace it by 
its   dyadic subintervals 
of length $\ell_j$ that don't hit $I_j$. Also add the 
dyadic intervals in ${\cal C}_j$ to ${\cal D}_{j-1}$
 to get the collection ${\cal D_j}$. 
Clearly  (1)-(4) all hold. 
Moreover, 
$$ \ell_{j-1} \geq \ell_j \geq |I_j|/8M 
\geq |I_{j-1}|/8MC \geq \ell_{j-1}/2C,$$
where $C$ is the bounded geometry constant of ${\cal Q}$. 
Thus (5) also holds. 

Next, suppose $\ell_j > \ell_{j-1}$. Then subdivide each 
dyadic interval in the cover ${\cal C}_j$
 of $I_j$ into dyadic subintervals 
of length $\ell_{j-1}$ and redefine $\ell_j  = \ell_{j-1}$.
Add these intervals to ${\cal D}_{j-1}$ to give ${\cal D}_j$ 
(except possibly the first interval, if it is already in 
the collection). Then (1)-(5) are all obvious.

Next, do  the analogous  construction for $j <0$ and reflect the resulting 
dyadic cover of $(-\infty ,-1]$ across zero to get a 
dyadic covering of  $[1, \infty)$. By taking the shortest 
interval covering each point we get a dyadic covering of 
$[1, \infty)$   which satisfies all the desired conditions.
\end{proof}

Fix $M \geq  8C$ (recall $C$ is the bounded geometry
 constant of ${\cal Q}$)
and  apply Lemma \ref{symmetric partition} to get a symmetric 
partition ${\cal Q}' = \{ I_j\}_{j \in \integers^*} $. 
We will use the partition ${\cal Q}'$ to fill the central region
with a meandering tube. Let $\{a_j\} \subset [1, \infty)$
 be the  positive endpoints of ${\cal Q}'$ ($I_j=(i a_j,  i a_{j+1})$).
Let $\eta_j = a_{j+1}-a_j = |I_j|$.
Now add the  vertical segments 
$$V_j = \left\{ x+iy: x =a_j,  \quad -a_j+ {\eta_j}\frac {1+(-1)^j}2 
\leq y \leq a_j - {\eta_j}\frac{ 1+(-1)^{j+1}}{2}   \right\}$$ 
inside the central region; more geometrically, we are
adding segments on the  
vertical  lines $\{ x = a_n\}$ that lie inside the 
central region so that one endpoint lies on the boundary 
of the central region and the other is distance $\eta_j$ below 
or above the boundary. Alternate segments  alternately  touch 
the ``top'' and ``bottom'' sides  of the central region.
This defines a simply connected subregion of the central 
region that we call the central tube.
The boundary of the tube consists of $I_0'$ and two connected 
components that we call the upper and lower components.
 The central tube is 
illustrated in the center of  Figure \ref{JoinTrap5B}.

We make  the boundary of the central tube
 into a tree by adding vertices on the 
vertical segments at their endpoints and at points 
spaced  $\eta_j $ apart on $V_j$. 
On the upper and lower boundaries of the central region 
we add a vertex at all the points $\{ a_j \pm i a_j\}_{j \geq 1}$.
It is easy to see that this makes the boundary of the 
central tube into a bounded geometry tree (since $\eta_j \simeq 
\eta_{j+1}$).

As before, let  $\{ z_j\} =\{i x_j\}$ be the endpoints of the 
partition ${\cal Q}=\{I_j\}$. 
For  each $k  \in \integers^*$, 
choose a   $y_k = a_{n(k)}$  so that 
$|x_k-y_k|<  |I_k|/M$. This is possible 
by  condition (\ref{cond 5}) in 
Lemma \ref{symmetric partition}.
Define the segment 
$H_k= [i x_k,  w_k ]$ where $w_{k} = |y_k|  +iy_k$.
%(see Figure \ref{DefnTrap}). 
This segment  
connects $i x_k$  to a  vertex  on the boundary of 
the central tube and is close to horizontal (the 
absolute value of its slope is $\leq 1/M$;
by abusing notation, we will refer to these segments as ``horizontal''.
Doing this for every $k$ divides the complement of the 
central region into quadrilaterals that look like 
trapezoids, and which we will call trapezoids by an 
another abuse of notation. 

We can also make the choice above
 so $n(k)$ is even if $k <0$ 
and is odd if $k>0$. This means that the right-hand 
vertex  $w_k$ of the segment $H_k$  is a degree $2$ 
vertex of the central tube, and hence forms a degree 
$3$ vertex when the segment $H_k$ is added. This is 
important because we want the final tree to have maximum 
degree $3$ (otherwise we would end up with the upper 
bound 16 instead of 12 in part (2) of Theorem 
\ref{App Omega}). 

\begin{figure}[htb]
\centerline{
\includegraphics[height=2.5in]{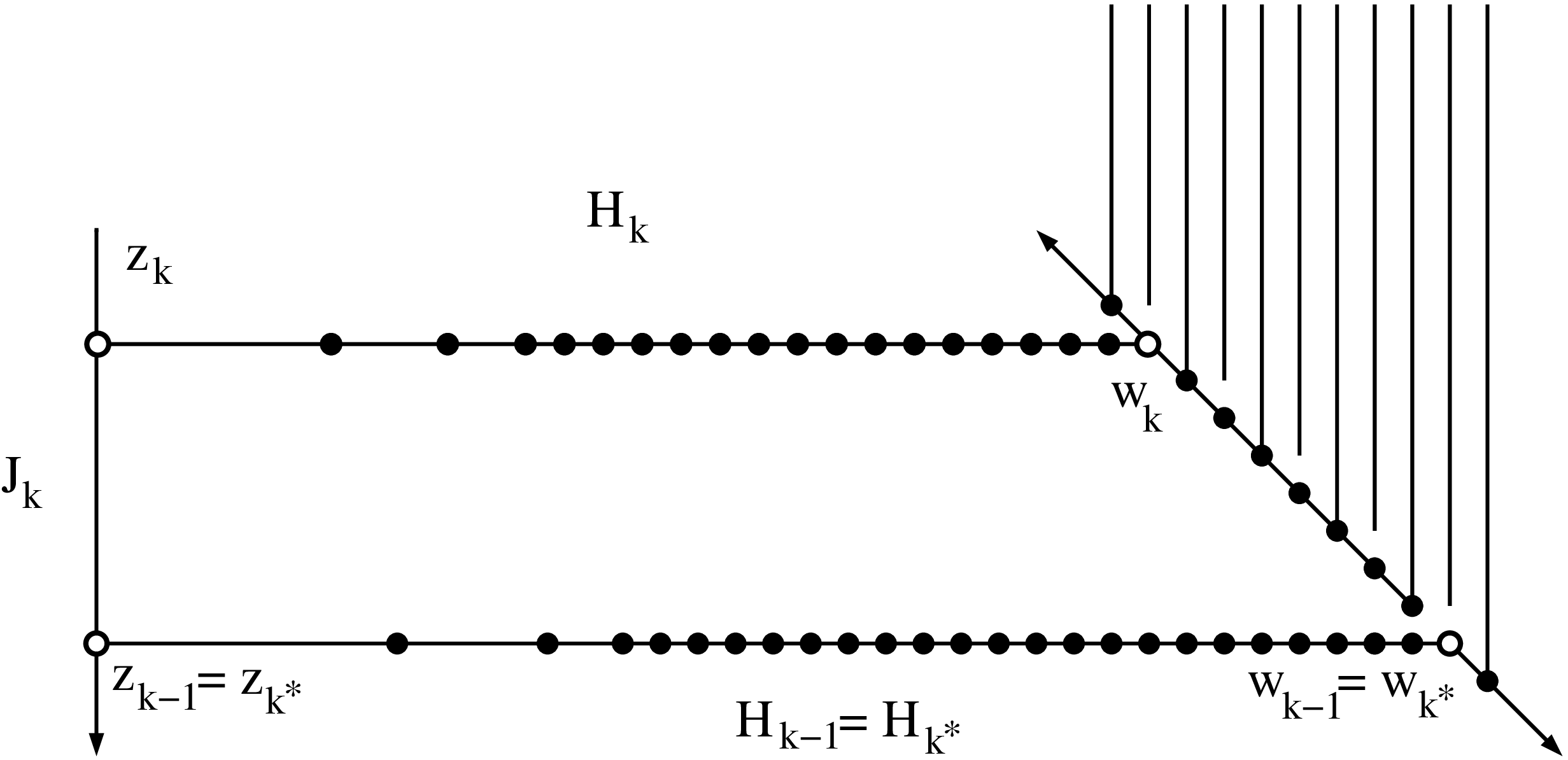}
}
\caption{ \label{JoinTrap1}
We build  approximate trapezoids by joining 
partition points on $\partial \rhp$ by 
almost horizontal lines to partition vertices 
on the boundary of the central region.
A small segment is then removed from the 
boundary of the central region so that 
the interior of the trapezoid is joined 
to the central tube.
 In this picture, $k<0$, so $k^* = k-1$.
}
\end{figure}

For $k \in \integers^*$, the $k$th trapezoid has 
 left side $ I_k \subset \partial \rhp$, two  
``horizontal'' sides $H_k$ and $H_{k^*}$,
 and a right-hand side
of slope $\pm 1$ along the boundary of the central region.
See Figure \ref{JoinTrap1}.
Vertices are added to the ``horizontal'' sides $H_k$
that break  $H_k$ into segments that 
start at the left with length comparable to the 
length of $I_k$ and end on the right with lengths
comparable to (but larger than)
 $ \eta_n =a_{n+1} - a_{n}$ where $n=n(k)$. Thus the 
sub-segment of $H_k$ that meets the boundary of the central 
tube has length comparable to the edges of the 
central tube at the meeting point.  Hence we have 
a bounded geometry forest in $\rhp$.
However, this forest  cuts the plane into 
infinitely many components (the trapezoids and the central tube).
We want to form a single component by removing 
some segments.

For each $k \in \integers^*$, choose $n = n(k)$ so 
$w_{k^*} = a_n \pm i a_n$ then remove the 
open segment $(a_{n-1} \pm i a_{n-1}, w_{k^*})$ from
the boundary of the central region; this is 
the segment on the boundary of the central 
tube   that is also on the boundary of the $k$th 
trapezoid and that  has $w_{k^*}$ as one endpoint.
See Figure \ref{JoinTrap1}.
Removing this segment connects the $k$th trapezoid
to the central tube. When we have removed all 
such segments, we have a bounded geometry forest
$T''$ with a single complementary component $W''$  in $\rhp$.
See the right side of Figure \ref{JoinTrap5B}.

All that remains is to prove that $W''$ satisfies a 
positive $\tau$-length lower bound.
First consider  sides of $W''$ that are 
also sides of the central tube. If such a side
lies on the upper boundary, it can clearly be
separated from the lower boundary by a path family 
with uniformly bounded modulus. Thus 
 Lemma \ref{tube bound}  implies that the 
$\tau$-lengths of  such sides grow 
exponentially and hence are bounded away  from zero.
A similar argument applies to sides on the lower 
boundary of the central tube.

Next we have to consider sides of $W''$ that are sides
of the $k$th trapezoid. We want to use Lemma
\ref{add rooms} with $I=I_0$, $J$ a 
side of $k$th trapezoid and $K$ a side of the 
central tube, chosen as shown in Figure \ref{JoinTrap2}.
To prove $M(J,K) \leq M(I,K)$ we will first give
an upper bound for $M(J,K)$ and  then give 
a lower bound for $M(I,K)$  that is larger than this bound.

Replacing $J$ by a sub-interval only increases 
$M(J,K)$, and every side of the $k$th trapezoid has 
length at least $\eta_{n^*}$ where  $n^*=n(k^*)$ is defined
by the relation $y_{k^*} = a_{n^*}$. 
So we assume that $J$ is any interval of length 
$\eta_{n^*}$  on the side of the $k$th trapezoid. 
Any such  $J$  is one side of a generalized 
quadrilateral  $Q \subset W''$ whose opposite side is $K$ and 
so that the two remaining sides are 
at least distance $\eta_{n^*}$ apart. See 
Figure  \ref{JoinTrap2}. 

\begin{figure}[htb]
\centerline{
\includegraphics[height=2.0in]{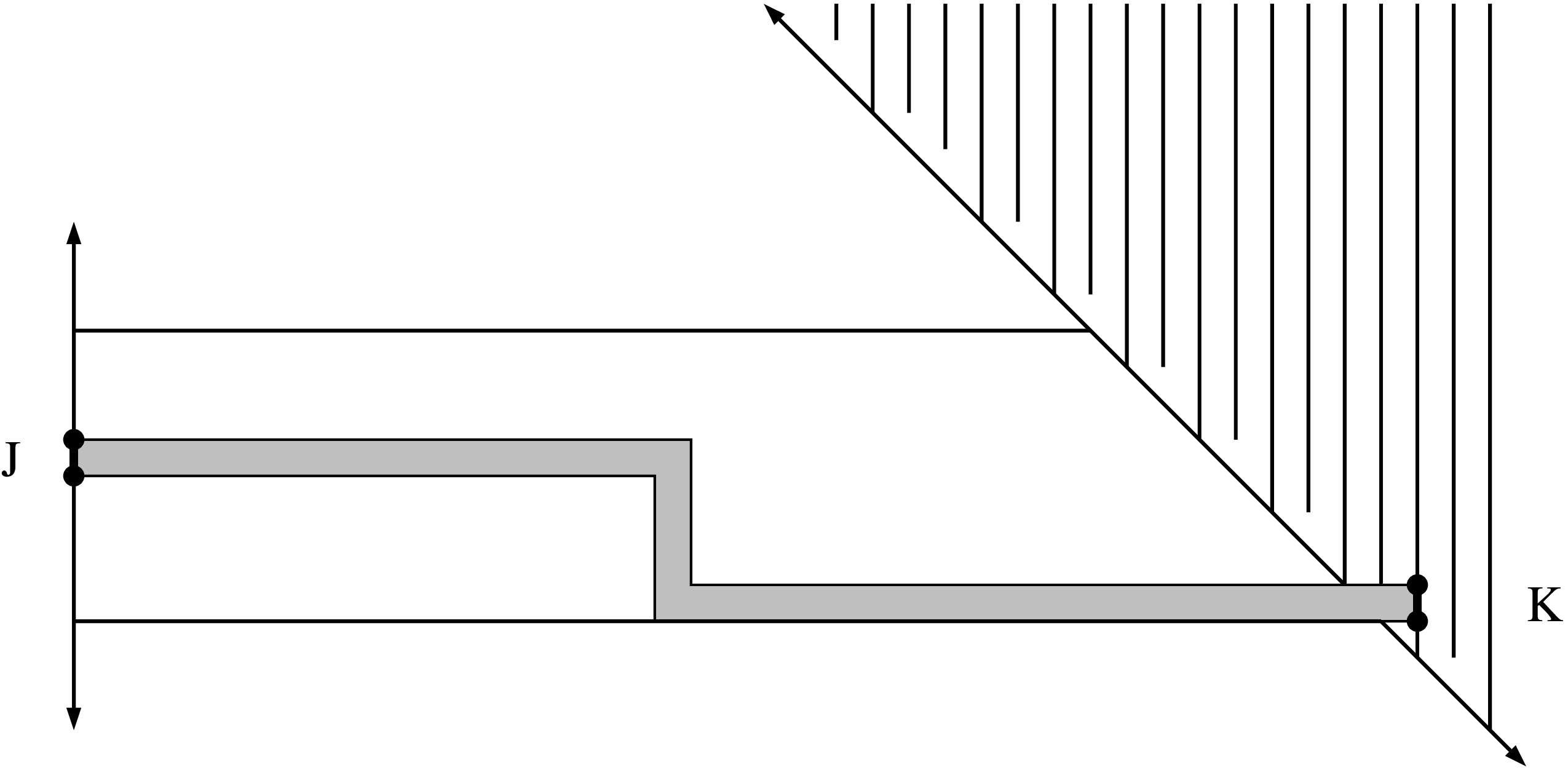}
}
\vskip.3in
\centerline{
\includegraphics[height=2.15in]{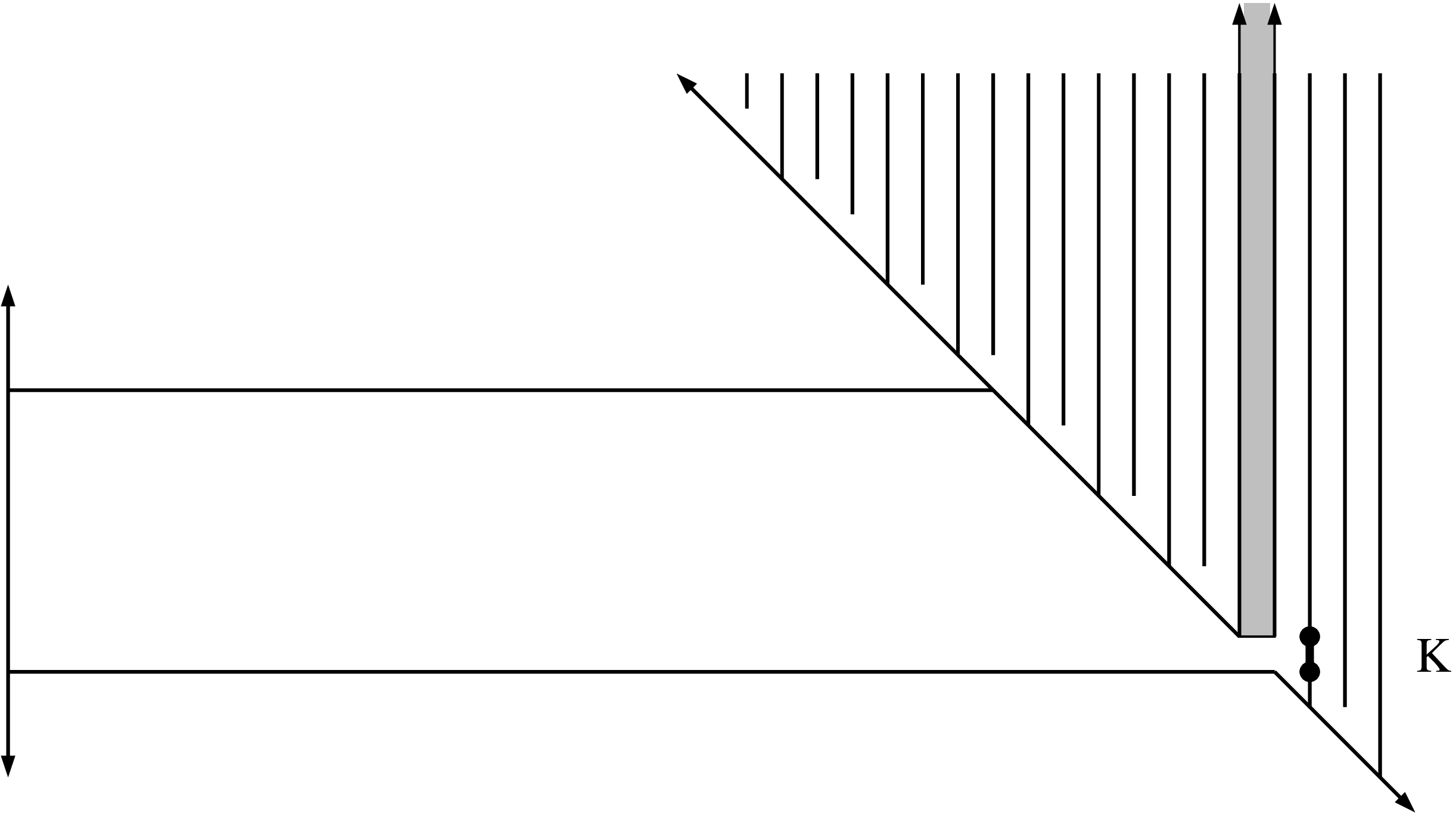}
}
\caption{ \label{JoinTrap2}
The modulus of the  shaded tube in the top picture 
gives an upper bound for $M(J,K)$.
The interval $J$ is on the side of the trapezoid, 
and $K$ is on the boundary of the central tube.
The modulus of the shaded rectangle in the bottom picture 
gives a lower bound for $M(I_0, K)$.
}
\end{figure}

Moreover,  we can choose 
$Q$ so its area is at  most $(|I_k|+y_{k^*}) \eta_{n^*}$.
Therefore the modulus of the path family separating 
$J$ from $K$ in the quadrilateral is at most 
$(|I_k|+y_{k^*})/ \eta_{n^*}$ (just take the constant 
metric $\rho = 1/\eta_{n^*}$ in the definition of 
modulus; any admissible metric gives an upper bound 
for the modulus.)
Any path separating $J$ and $K$ in $W''$ contains 
a sub-path that separates them in $Q$, so by the 
extension principle 
$$M(J,K) \leq  (|I_k|+y_{k^*})/ \eta_{n^*}.$$

Now we give a lower bound for $M(I,K)$.
Note that there is  
$\eta_{n^*} \times (2y_{k^*}-2\eta_{n^*})$ rectangle that 
separates $K$ from $I_0$; it is contained in 
the vertical section of the central tube just 
above the opening to the $k$th trapezoid and 
the lower portion of the rectangle
is shown as a shaded region in the bottom picture 
of Figure \ref{JoinTrap2} (the rectangle extends upwards
almost to the upper boundary of the central region).
 The horizontal segments 
that cross this rectangle  separate $K$ from $I_0$, so
by the extension principle again, we see that the 
modulus of these horizontal segments is a lower
bound for $M(I,K)$, hence (since the modulus of a 
rectangle is the ratio of its sides), 
 \begin{eqnarray} \label{MIK est}
 M(I,K) \geq  \frac{2y_{k^*}}{\eta_{n^*} } -2.
\end{eqnarray}

Now we have to compare our lower bound for $M(I,K)$ 
to our upper bound for $M(J,K)$. 
 By the definition of the bounded geometry constant
$C$ for ${\cal Q}$, if $|k|=1$ then 
$I_k \in {\cal Q}$ is distance 
$1$ from the origin and $|I_k| \leq C|I_0| \leq 2C$, so 
 $1 \geq |I_k|/2C$. If $|k|>1$, 
then $I_k \in {\cal Q}$ is  separated from the origin by another 
partition interval of length at least $ |I_k|/C$.
So in either case  
$$y_{k^*} \geq x_{k^*}  \geq |I_k| +  |I_k|/2C,$$
so that  
 $|I_k|  \leq (1-\frac 1{4C}) y_{k^*}$ (here we use 
that $(1+\epsilon)^{-1} < 1-\epsilon/2$ if $0 < \epsilon <1$).
Hence
\begin{eqnarray*}
M(J,K)
&\leq&  \frac {|I_k|+y_{k^*}}{\eta_{n^*}} 
\leq   (2-\frac 1{4C}) \frac {y_{k^*}}{\eta_{n^*}} \\
&\leq&  \frac   {2y_{k^*}}{\eta_{n^*}} - \frac{y_{k^*}}{4C\eta_{n^*}}   
\leq  \frac   {2y_{k^*}}{\eta_{n^*}} - \frac{|I_k|}{4C\eta_{n^*}}   \\
&\leq& \frac   {2y_{k^*}}{\eta_{n^*}} - \frac{M}{4C}  .
\end{eqnarray*}
Since we assumed $M  \geq   8C$, and using (\ref{MIK est}), 
we have 
$$ M(J,K) \leq  \frac   {2y_{k^*}}{\eta_{n^*}} - 2  \leq    M(I, K).$$
Thus by Lemma \ref{add rooms}, the $\tau$-length of $J$ 
is greater than that of $ I=I_0$ and hence is bounded
uniformly from below.
As before, it is easy to check that the constructed tree 
has maximum vertex degree $3$.
 This completes the proof of 
Lemma \ref{single component} and hence the proof of 
Theorem \ref{details} (the number of tracts in the approximation 
is at most double the number of tracts in the model).

%----------------------------------------------------------------
\section{Geometric restrictions on Speiser models }  \label{near constant sec}

So far, this paper has dealt with methods for building Speiser
class functions. The remainder of the paper is devoted 
to placing limits on what can be accomplished in this 
direction. 
 In this section, we show that 
the choice of $\tau$ on different tracts of a Speiser
class function must satisfy certain constraints; no 
such restriction need hold for the Eremenko-Lyubich class
by the results of \cite{Bishop-EL-models}.
This is a clear difference between the two classes.

Suppose $f \in \classS$
and $S(f) \subset \disk$.  Since $S(f)$ is finite, 
there is an  $\epsilon>0$ is so  that 
$$ \dist(S(f), \partial \disk) > 4 \epsilon$$
 and
$$ \min \{ |a-b|: a,b\in S(f), a \ne b \} > 4 \epsilon.$$
For $a \in S(f)$ 
 let $D_a= D(a, \epsilon)$. and let $2D_a=D(a, 2 \epsilon)$.
The open  disks  $D_a$ are evidently pairwise disjoint and all 
lie inside $\disk$.

For each $a \in S(f)$, the set $W(a,\epsilon)=f^{-1}(D_a)$ 
only  has simply connected components,
and on each such component $U$  the map $f: U \to D_j $ acts either as
\begin{enumerate}
\item    a $1$-to-$1$ map onto $D_j$,  
\item  a  finite-to-$1$  branched cover   of $D_j $ with
       a single  critical value  at   $a$ or
\item an $\infty$-to-$1$ cover of $D_a \setminus \{a\}$.
\end{enumerate}
In the first two cases $U$ is bounded, and in the 
third  case it is unbounded and contains a
 path to $\infty$ along which $f$ 
has  asymptotic value $a$. 
Moreover, the pre-images $f^{-1}(a+\epsilon)$
  partition  $\partial W(a,\epsilon)$ into arcs.
Let $X = \overline{\disk} \setminus \bigcup_{a \in S(f)} D(a, \epsilon)$.

Recall that given $r >0$ and an  arc $I$, we define a neighborhood 
of $I$ by 
$$ I(r) =  \{ z: \dist(z, I) <  r \cdot\diam(I) \} .$$

\begin{thm} \label{near constant}
There is an $r < \infty$, depending only on $\epsilon$, 
so that for each partition arc $I$  of
 $\partial W(a,\epsilon)$   there is
a  edge $J$ of $\partial \Omega$  with
$I \subset J(r)$ and $J \subset I(r)$.  Moreover, 
$ \diam(I)  \simeq  \diam(J) \simeq \dist(I,J)  $
and the lengths of $I$ and $J$ are comparable 
to their diameters. 
\end{thm}

\begin{proof}
Fix $a \in S(f)$. Then $\partial D_a$ can be covered by a 
uniformly bounded number of disks whose doubles don't 
hit $S(f)$ and so $f^{-1}$ is conformal from each such 
disk to any of its pre-images under $f$. If $I$ 
is a partition arc of $\partial W(a,\epsilon)$, this 
fact and Koebe's distortion theorem  
imply that  $I$ has bounded geometry and its diameter is 
comparable to its length with constants that are 
uniform (the constants  only depend on the number of 
disks covering $\partial D_a$, which is uniformly bounded).

Similarly, $ \circle = \partial \disk$  is covered by 
$O(\epsilon^{-1})$ disks of radius $\epsilon$ whose doubles
don't intersect $S(f)$, so the same argument shows 
that a partition arc $J$ of $\partial \Omega$ has 
bounded geometry, but now with constants that depend on 
$\epsilon$. 

Finally, $\partial D_a$ and $\circle$ can be joined by a 
curve $\gamma$ that is never closer than $\epsilon$ to 
any point  of $S(f)$ (use a straight line segment and 
replace its intersection with any $D_a$ by the shorter 
arc of $\partial D_a$ connecting the same two points; see 
Figure \ref{AvoidingPath}).
This arc can also be covered
by $O(\epsilon^{-1})$ distinct
 $\epsilon$-disks whose doubles miss $S(f)$.  

\begin{figure}[htb]
\centerline{
\includegraphics[height=2.6in]{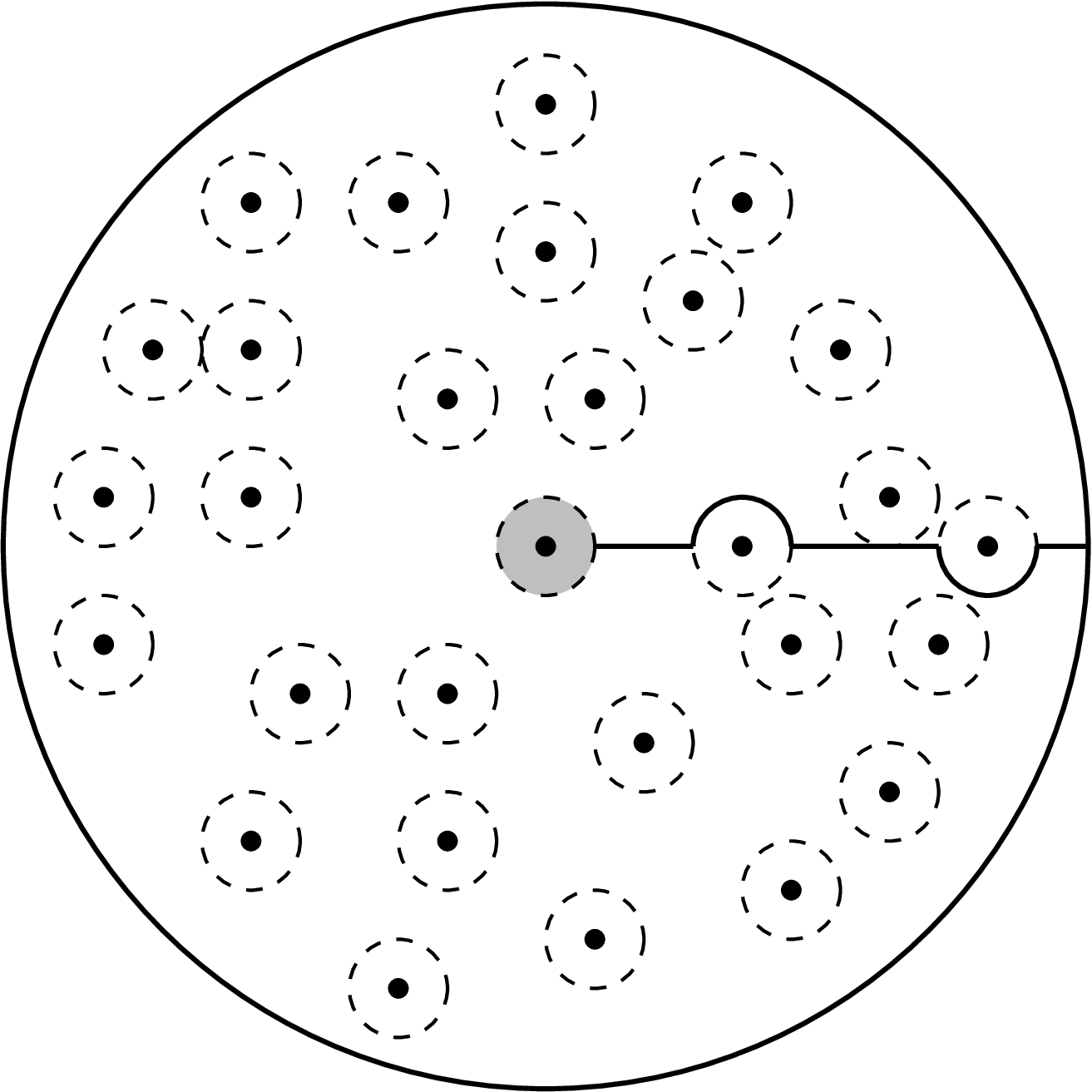}
}
\caption{ \label{AvoidingPath}
Each $D(a, \epsilon)$  can be connected to $\{|z|=1\}$ 
by a path  $\gamma$ of length $\leq \pi$ that stays at least
distance $\epsilon$ from every singular point. Such a 
path can be covered by $O(\epsilon^{-1})$ disks, 
each which has a double that misses the singular set.
Therefore the derivative  of 
any branch of $f^{-1}$ on $\gamma$ has 
comparable sizes at any two points of $\gamma$
(with a constant depending only on $\epsilon$).
}
\end{figure}

Therefore all the lifts of $\gamma$ via $f$ have bounded 
geometry with constants depending only on $\epsilon$. 
Thus if we take $I$ as above and a lift of $\gamma$ with 
one endpoint on $I$, then the other endpoint of the lifted 
curve  $\gamma'$ is on a partition arc $J$ of $\partial \Omega$. 
By Koebe's theorem   $|f'|$ has comparable values at all 
points of $I \cup \gamma' \cup J$  with constant that 
depends only on $\epsilon$. Since $\partial D_a$, $\gamma$ and 
$\circle$ all have comparable lengths (within a factor of 
$O(\epsilon^{-1})$), so do $I$, $\gamma'$ and $J$. 
This proves the lemma.  See Figure \ref{LiftedPath}. 
\end{proof}

\begin{figure}[htb]
\centerline{
\includegraphics[height=2.4in]{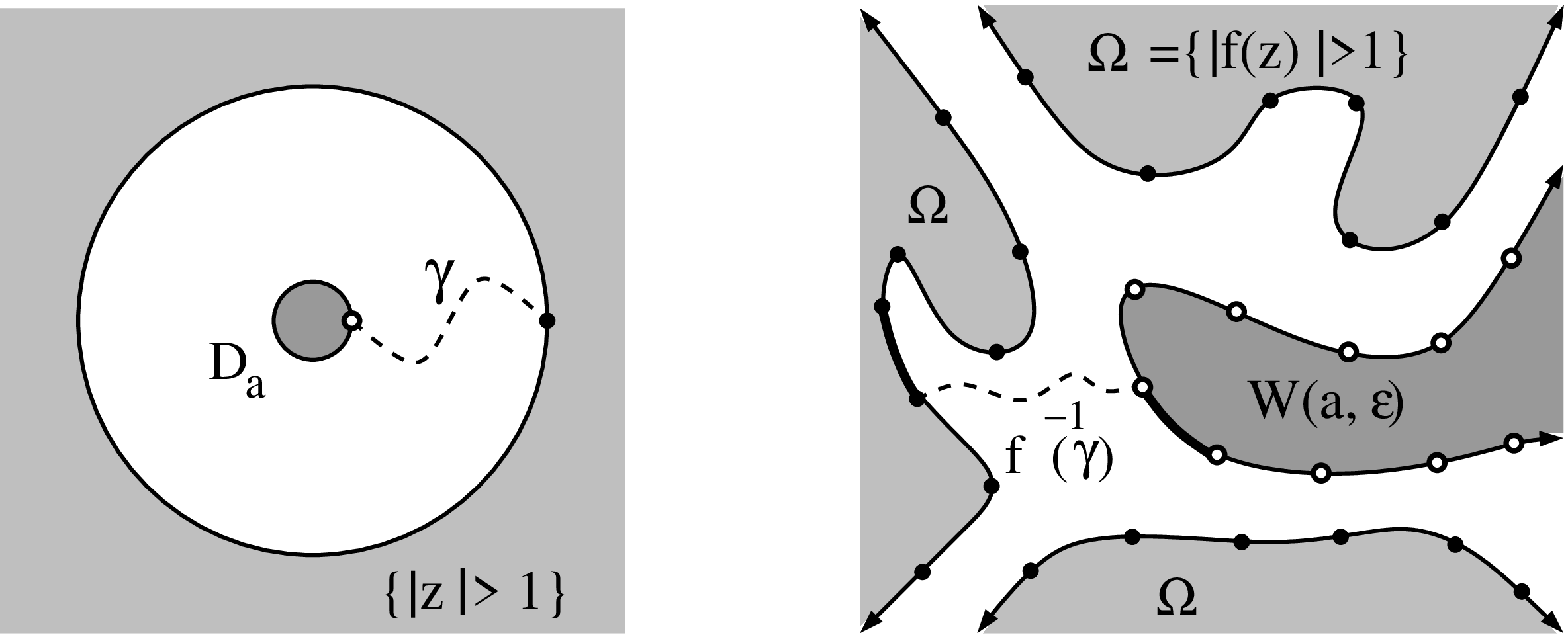}
}
\caption{ \label{LiftedPath}
As noted earlier, 
each $D(a, \epsilon)$  can be connected to $\{|z|=1\}$ 
by a path  $\gamma$  that stays at least distance 
$\epsilon$ away from  the singular set. 
A lift of this path via $f$ is a path that 
connects partition arcs of $\partial \Omega$
and $\partial W(a, \epsilon)$. As explained in 
the text, the lengths and  diameters
of the lifted curve and the arcs its connects are  all
comparable with a constant depending only on $\epsilon$. 
}
\end{figure}

\begin{cor}
With notation as above, there is a $r < \infty$,
 depending only on $\epsilon$, 
so that $f^{-1}(X) \subset T_\Omega(r)$.
\end{cor} 

\begin{proof}
Note that  $\partial W(a, \epsilon) \subset T(r)$ and this implies
$f^{-1}(X) \subset T(r)$, as claimed.
\end{proof}

Theorem \ref{near constant}  immediately implies:

\begin{cor}
Suppose $f$ is in the Speiser class and  $S(f) \subset \disk$. 
For every $\epsilon >0$ there is a $r < \infty$ so that 
each  connected component of $\complex \setminus (\Omega\cup
T_{\Gamma}(r))$ (where $\Gamma = \partial \Omega$)
  maps under $f$ into some disk $D_a$
for some $a \in S(f)$.  If the connected component is unbounded, 
then $a$ must be an asymptotic value of $f$.
\end{cor} 

If the critical points of $f$ have uniformly bounded degree $D$, then the
components of $W(a,\epsilon)$ containing critical points  have boundaries 
with  at most $D$ partition arcs, each with diameter comparable to the 
whole component (the constant depending only on $D$). Since one of these 
arcs in contained in some $J(r)$ the whole component will be contained 
in $J( Cr)$ for some $C$ depending only of $D$.   
Thus 

\begin{cor}
If $f \in \classS$,  $S(f) \subset \disk$ and $f$
 has no finite asymptotic values and every 
critical point has uniformly bounded degree $D$, then there 
is a $r> 0$ so that  $\complex = \Omega \cup T(r)$ 
(as before, $\Omega = \{ |f| > 1\}$, 
and $r$ depends only on  $D$ and $\epsilon$, where 
$\epsilon$ is the minimal separation 
between different points of $S(f)$ and between  $S(f)$ and 
$\circle$). 
\end{cor}

For the half-strip $S = \{ x+iy: x> 0, |y| < 1\}$ the partition
elements decay exponentially fast and it is obvious
that $\Omega \cup T_\Omega(r)$ is not the whole plane for
any finite $r$.  Thus this model can't be approximated in the 
sense of Theorem \ref{App Omega} without using extra tracts.
In the remainder of the paper we will show that something 
even stronger is true: the approximation of a half-strip   by 
a Speiser  model domain with a  single tract  is not
 possible  even if we allow: 
\begin{enumerate}
\item finite asymptotic values,
\item  high degree critical points, and 
\item  $\varphi$ to be non-conformal everywhere in the plane. 
\end{enumerate}

Very briefly, the problem with the half-strip is that 
the $\tau$-lengths for it and its complement behave 
so differently, that the comparability implied 
by Theorem \ref{near constant} cannot hold. 
See Figure \ref{TauBound}.

\begin{figure}[htb]
\centerline{
\includegraphics[height=.9in]{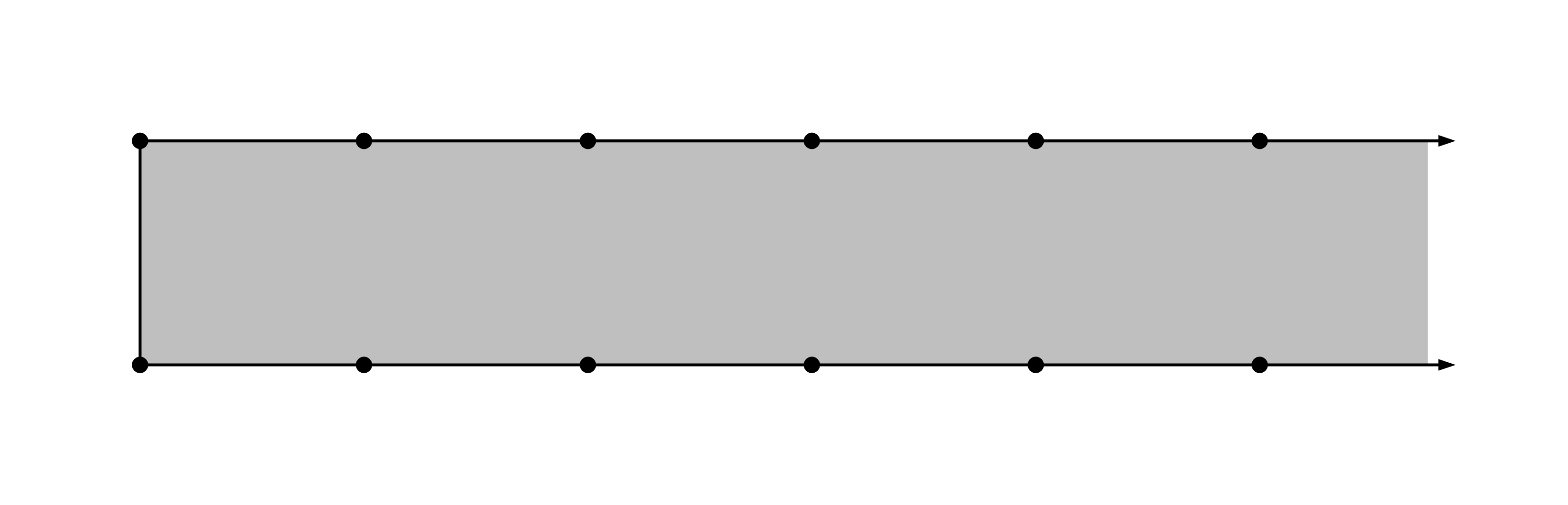}
$\hphantom{xxx}$
\includegraphics[height=.9in]{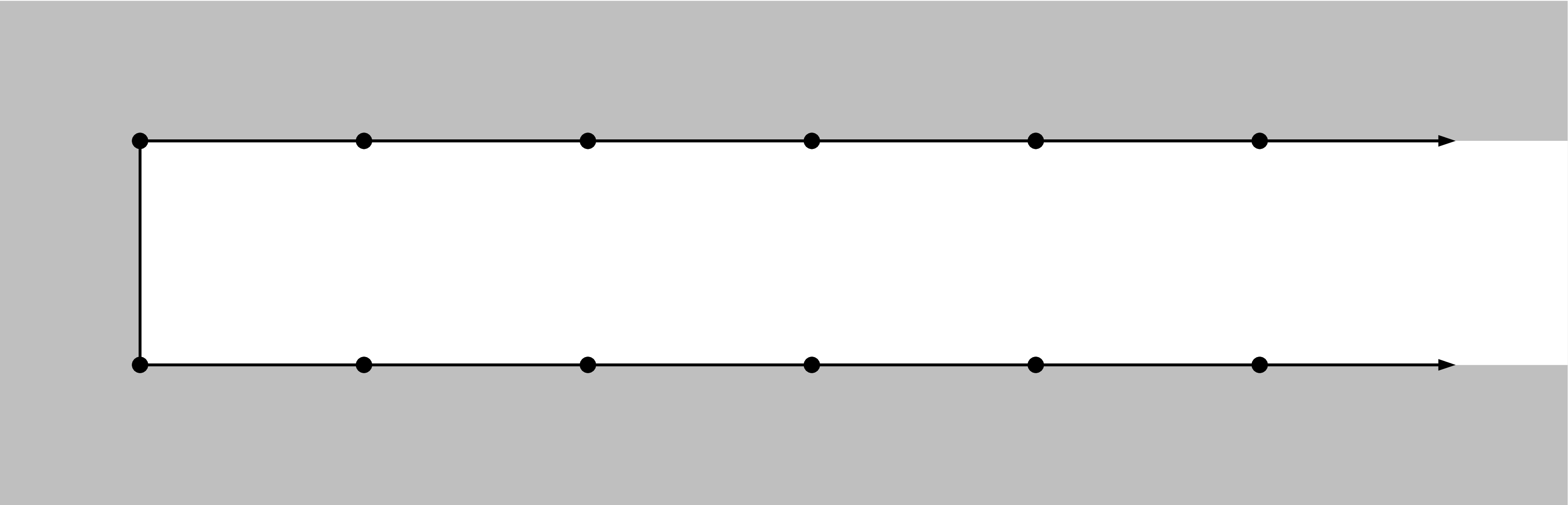}
}
\caption{ \label{TauBound}
Suppose $T$ is  the boundary of a  half-strip with unit 
spacing of the vertices, as shown above.
The conformal map of the half-strip (left) to 
a half-plane is $\sinh(z)$, and it
 expands exponentially, so the  unit 
segments shown each contain exponentially 
many conformal partition elements for the half-strip.
However,  the conformal map of the exterior domain  (right)  
to $\rhp$ 
behaves like $z^{1/2}$ near $\infty$, so  the   unit segments
are much smaller than conformal partitions elements 
should be.
This ``imbalance'' of $\tau$-sizes is what prevents the 
half-strip (or any quasiconformal image of it) from 
being a Speiser model domain.
The following sections will make this precise.
}
\end{figure}

%----------------------------------------------------------------
\section{ A polynomial lower bound for thick tracts} \label{lower sec}

Suppose $\Omega$ is a simply connected planar domain bounded 
by a Jordan curve on the sphere that passes through $\infty$. 
Suppose $\tau:\Omega \to \rhp$ is 
conformal, maps $\infty$ to $\infty$ and  ${\cal J}$ is the partition 
of $\partial \Omega$  that corresponds via $\tau $ to the 
partition of $\partial \rhp$ with endpoints $i \pi \integers$
(recall that this is called a conformal partition of $\partial \Omega$).
In this section we want to prove that if a tract
$\Omega$  
is ``large'' in  a certain sense, then  the size of 
elements in  ${ \cal J}$ cannot tend to zero 
too quickly.  By ``large'' we will mean that $\Omega$ 
contains an unbounded quasidisk. For example, 
a sector $W_\theta = \{ z: |\arg(z) | < \theta  \}$, $0 < \theta < \pi$
 is an example of an unbounded quasidisk, so if $\Omega$ 
contains a sector, its partition elements cannot  have 
diameters that decrease exponentially quickly.

\begin{lemma} \label{HM 1} 
Suppose $\Omega$ is bounded by a Jordan curve  through $\infty$
and $\{\cal J\}$ is a  conformal partition of $\partial \Omega$.
Suppose $\Omega$ contains an unbounded quasidisk $W$. Then 
there is a  $R_0 < \infty$ so that any 
partition arc $J$ that hits a circle $\{|z|=R\}$ 
with $R > R_0$  satisfies
$ \diam (J)\geq C R^{- \sigma}$
for some $C>0$, $\sigma < \infty$, independent of $J$.
\end{lemma}

\begin{proof}
Since $\Omega$ contains an unbounded quasidisk, 
Lemma \ref{unbounded quasidisk} says there is a 
curve $\gamma \subset \Omega$ that  connects 
some point $z_0 \in \Omega$  to 
$\infty$ and has the property that  there is a 
$C_1 < \infty$ so that 
$  \dist(z, \partial \Omega ) \geq C_1 |z|,$
for all $z \in \Gamma$. 
See Figure \ref{LowerBoundProof}.

\begin{figure}[htb]
\centerline{
\includegraphics[height=2.3in]{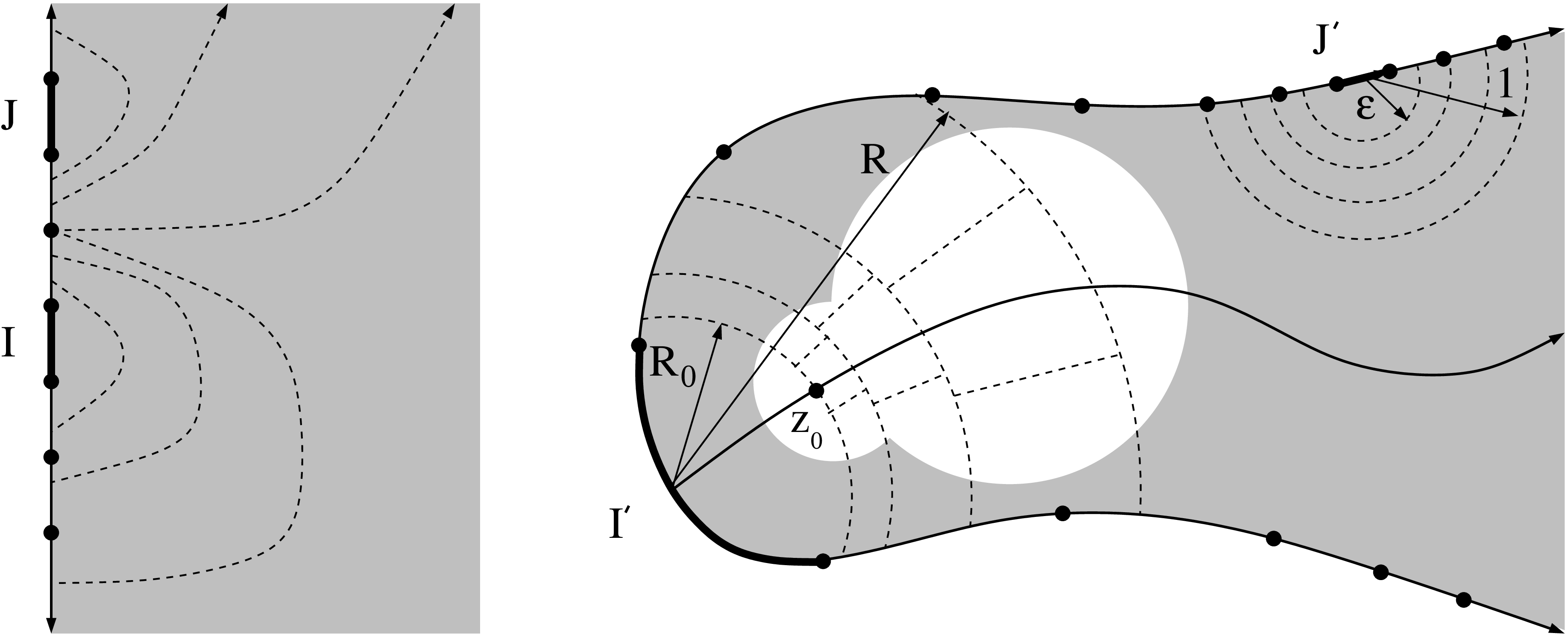}
}
\caption{ \label{LowerBoundProof}
The family  of all separating curves 
for two  unit intervals distance $n$ 
apart  on $ \partial
\rhp$ has modulus $\simeq \log n$. Thus  in 
$\Omega$ any separating family must have 
modulus $\gtrsim \log n$. Applying this 
to two families as in the text says that 
image arcs are separated by distance 
$R \lesssim n^\alpha$ and the second arc
has diameter $\epsilon \gtrsim n^{-\beta}$, 
for some finite, positive $\alpha$ and $\beta$. 
}
\end{figure}

Choose $R_0$ large enough so that  $|z_0| < R_0$ 
and so that $\{|z|< R_0\}$ contains  
 some partition element 
 $I' = \tau^{-1}(I)  \in  \partial \Omega \cap {\cal J}$.
Fix $R \gg   4R_0  $ and 
choose a partition element $J$ so that  $J'=\tau^{-1} 
(J) \subset \partial \Omega$ hits $\{|z|=R\}$. If $J'$ 
has diameter greater than $1$, there is nothing to 
do, so we may assume $ \epsilon=\diam(J') \leq 1$ and hence 
$J' \subset  \{ z: R-1 \leq |z| \leq R+1\}$.
See Figure \ref{LowerBoundProof}.

Suppose  $I$ and $J$    (which are
equal length intervals on $\partial \rhp$)   are separated by $n$ 
other  partition elements. 
By Lemma \ref{two intervals}, $M(I,J) \simeq \log n$,  (recall 
that $M(I,J)$ is the  modulus of the path family   
 that separates $I$ and $J $ in $\rhp$).

We consider two other  path 
families in $\Omega$. 
 First, let  $\Gamma_1$ be the family of circular arcs 
in $\Omega \cap \{ z: R_0 < |z| < R/2\} $ 
concentric with $0$ that connect  the two components
of $\partial \Omega \setminus I'$.
 See Figure \ref{LowerBoundProof}.
 Let $z_{J'}$ be any point of $J'$ and 
let   $\Gamma_2$  be the path family  consisting of  circular  
arcs in $\{ z : \diam(J') <   |z-z_{J'}| < 1\}$ that are 
concentric with $z_{J'}$ and connect different components of 
$ \partial \Omega \setminus J'$. 
See Figure \ref{LowerBoundProof}. 

Each path in $\Gamma_1$ and $\Gamma_2$ separates $I'$ from 
$J'$ in $\Omega$, so by conformal invariance and the 
parallel rule for modulus, we have 
$$ \log n \simeq M(\Gamma_0) \geq M(\Gamma_1) + M(\Gamma_2),$$
and thus both terms on the right are bounded by $C_2 \log n$ 
for some $C_2 < \infty$.

Because  $\gamma$ crosses each element of $\Gamma_1$
and each crossing point $z$ is at least distance $C_1 |z|$
from $\partial \Omega$, we deduce that there is a  
quadrilateral region 
$$Q = \{ w:  |\arg(w)-\arg(z) | \leq  A, 
1 \leq |w/z| \leq B\}$$
 contained in $\Omega$   for some
$A >0$, $B > 1$ depending only on the constant $C_1$. 
The path family in $Q$ connecting the two radial sides 
of $Q$ has fixed modulus $M_Q$ and this modulus is a lower bound
for the modulus of the path family in 
$\Omega\cap\{ |z| < w < B|z|\}$ connecting different 
components of $\partial \Omega \setminus I'$.  
Thus, by the parallel rule,  the 
modulus of $\Gamma_1$ is bounded below by $M_Q   \cdot 
\lfloor\log_B (R/R_0)\rfloor$. In other words, 
$ \log n \gtrsim M(\Gamma_1) \gtrsim \log R, $ for $R$ large.
Hence, when $R$ is large, we have 
$R \leq n^\alpha$ for some $\alpha$ that only depends on 
the constant $C_1$. 

On the other hand, the usual estimate of the modulus of 
an annulus says that 
$$ \log  \frac 1{\diam(J') }  \lesssim \log n ,$$
so 
$ \diam(J') \gtrsim n^{-\beta},$
for some $\beta >0$.  Thus 
$$\diam(J') \gtrsim (R^{1/\alpha})^{-\beta}
\gtrsim R^{-\beta/\alpha} = R^{-\sigma},$$
 as desired. 
\end{proof}

%---------------------------------------------------------------
\section{An exponential upper bound for thin tracts}  \label{upper sec}

We now want to do the opposite of the previous section: 
show that if $\Omega$ is  ``thin'',
then the  diameters of partition elements decay 
faster than any polynomial.

\begin{lemma} \label{HM 2} 
Suppose $\Omega$ is the image of the  half-strip 
$S = \{ x+iy: x >0, |y| < \frac 12\}$ under a $K$-quasiconformal 
map  $\phi$ of the plane fixing $0$ and $1$ 
and ${\cal J}$ is a conformal partition of $\partial \Omega$. 
Then  all the partition elements  satisfy
$$ \diam (J)  \leq C_1 \exp(-C_2\dist(J,0)^\alpha ) ,$$
for  some  finite, positive constants $C_1, C_2, \alpha$ depending on $K$. 
\end{lemma} 

\begin{proof}
Using the measurable Riemann mapping theorem, we 
can write $\phi$ as a composition  of two 
$K$-quasiconformal maps $\phi=g_2 \circ g_1$, where 
both maps also fix  both $0$ and $1$,
 $g_1$ is conformal outside $S$ and $g_2$ 
is conformal on  $W=g_1(S)$.
% (see e.g., \cite{MR2241787}).

Consider  the square  $S_n$  inside  $S$ between 
the vertical lines 
$\{ x= n\}$ and $\{ x = n+1\}$. Then  $W_n=g_1(S_n)$ 
is a quasidisk and its image in $\rhp$ under the conformal 
map $\tau: \rhp \to W$ is generalized quadrilateral $Q_n$ with 
two sides on $\partial \rhp$ and modulus bounded above and below, 
depending only on $K$.
See Figure \ref{ExpDecayProof}. 

\begin{figure}[htb]
\centerline{
\includegraphics[height=3.0in]{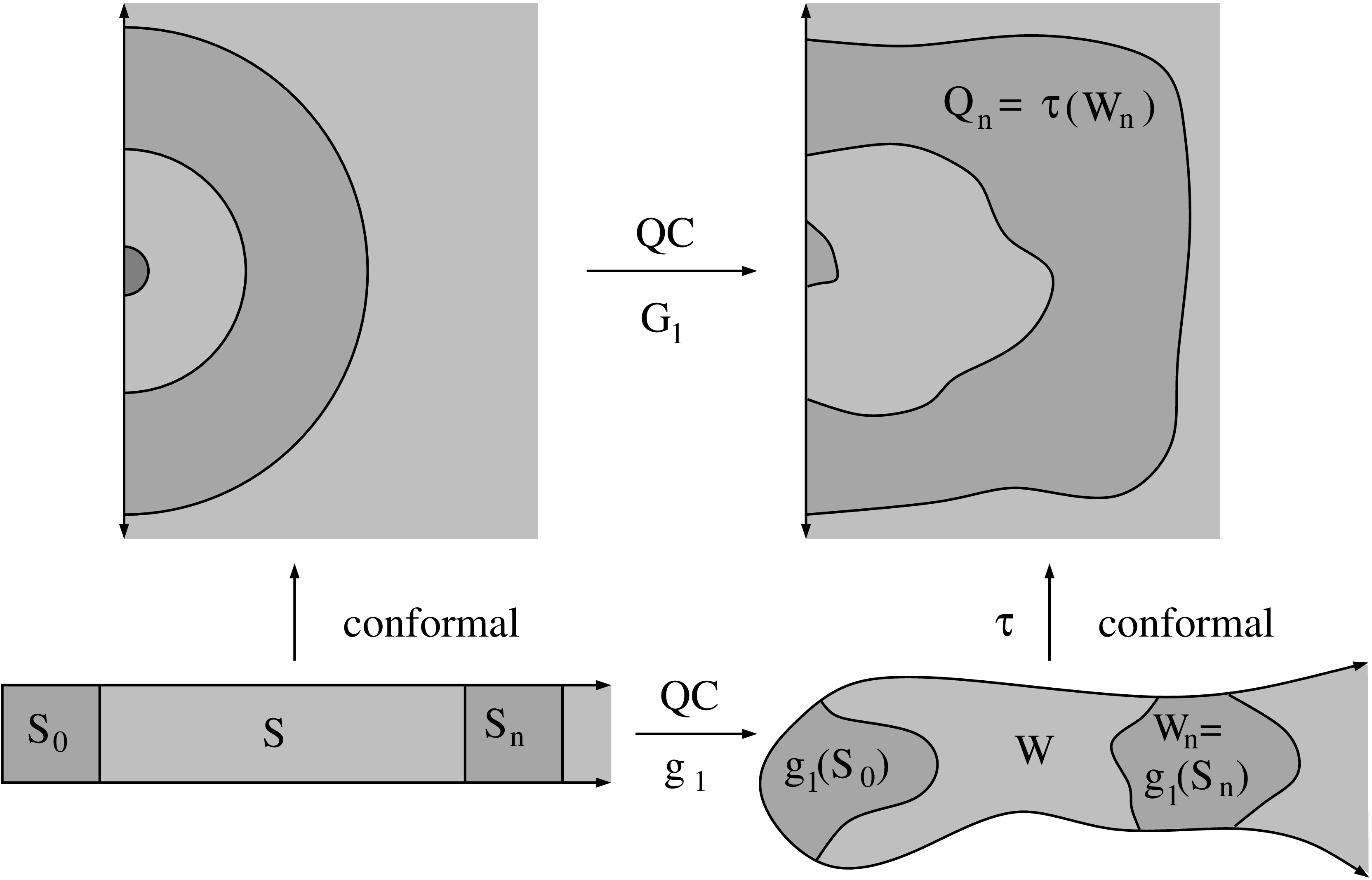}
}
\caption{ \label{ExpDecayProof}
The modulus separating $S_0$ and $S_n$ in the 
half-strip is comparable to $n$, so  by quasi-invariance
the same is 
true for $Q_0$ and $Q_n$ in $\rhp$.
This implies the Euclidean 
diameter of $Q_n$ grows exponentially, hence 
$\partial Q_n \cap \partial \rhp$ contains exponentially 
many partition points. The map from $Q_n$ to $W_n$ is 
conformal, so the same is true of $W_n$ and partition 
arcs for $W = g_1(S)$. 
}
\end{figure}

Since the extremal length between $S_0$ and $S_n$ in $S$ is 
$\simeq n$, the same is true for the extremal distance 
between $Q_0$ and $Q_n$ in $\rhp$in $\rhp$  (with a constant depending on the 
maximal dilatation  $g_1$). This implies that the  
 diameters of $Q_n$ must grow exponentially, as 
do the component intervals of $\partial Q_n \cap 
\partial \rhp$ (this is the 
same argument as in the proof of Lemma \ref{tube bound}). 
Thus  $\partial Q_n$ hits $\geq c e^{cn}$ partition intervals
on $\partial \rhp$  
for some fixed $c >0$ (depending only on $K$). Hence $W_n$ hits the same 
number of partition arcs on $\partial W$. Because $W_n$ is 
the image of $U_n$ under a  bi-H{\"o}lder  map (since it 
is a quasidisk), each of these partition arcs has diameter bounded 
by $C \diam(W_n) \exp( - an)$ for another constant $a$ depending 
only on $K$. 

 Since  $g_1$ has dilatation supported in the half-strip $S$ and 
$$   \int_S \frac {dxdy}{1+x^2+y^2}   <  \infty,$$
Theorem  \ref{TWB thm} implies that 
$|g_1(z)/z|$ has a limit as $z \to \infty$.
Thus if  $R_n = \dist(W_n,0)$ we have   $ R_n \simeq n$ 
and $\diam(W_n) \lesssim n$. Thus  
all the partition elements hitting $W_n$ have 
diameters less than $ c n e^{-an}$ where $a,c$ are 
positive constants that depend only on $K$.

Since $g_2$ is conformal on $W$, the partition for $\Omega$ is 
just the image of the partition for $W$ under $g_2$, and since
$g_2$ is bi-H{\"o}lder (with exponent depending only on $K$),
the estimate in the lemma follows.
\end{proof} 

The proof can be applied to other tracts that look like
thin tubes, e.g.,  
$$ \Omega = \{ x+iy: x>0, |y| \leq \eta(x)\},
$$ 
if $\eta(x), \eta'(x)  \to 0$ as $x \to \infty$.
However, the proof does not work for all subdomains
of the half-strip, since adding ``rooms'' to the sides
of a half-strip can create partition arcs whose diameters 
do not tend to zero (in fact, we used a similar 
construction in the proof of Lemma \ref{single component}.)

%-----------------------------------------------------------
\section{The half-strip is not the QC-image of a 
         Speiser model domain} \label{not}

Before starting the proof of Theorem \ref{not thm}, we 
record the following result that is immediate from 
Lemma \ref{QC edge nbhds}.

\begin{cor} \label{QC image}
Suppose $\Omega$ is a model domain.
If $\phi$ is a $K$-quasiconformal map of the plane that is conformal 
on $\Omega$, then 
$$ 
T_{\phi(\Omega)}(t) \subset \phi(T_\Omega(r)) \subset T_{\phi(\Omega)}(s),$$
 where $t,s$ depend only on $r$ and $K$. 
\end{cor} 

We can now prove Theorem \ref{not thm}: 
the half-strip $S = \{ x+iy : x>0, |y|< 1\}$ cannot be mapped to any  
Speiser class  model domain by any quasiconformal homeomorphism
of the plane.

\begin{proof} 
Suppose there were  a $K$-quasiconformal 
map $\phi$  of the plane taking $S$ to the tract $\Omega=\{ z: |f(z)| >R\}$
of some  $f \in \classS$.
Choose $\epsilon$ as in Theorem \ref{near constant} and let 
$r$ be as given by that theorem. Let $s$ be as given 
by Corollary \ref{QC image}.

As  in the proof of Lemma \ref{HM 2},
 write $\phi= g_2 \circ g_1$  where $g_1$ 
is conformal off $S$, $g_1$ is conformal on 
$W=g_1(S)$.  
Using Theorem \ref{TWB thm} again implies that 
that we can choose $g_1$ so that 
$$W \cup T_W(s)  \subset V = \{z:|z|< R\} \cup \{z: |\arg(z)| < 
\pi /4\} ,$$
if $R$ is large enough (depending on $s$). Note that 
$V^c$ is a quasidisk and hence $V'=g_2(V^c)$ is a quasidisk 
as well. 
By Lemma \ref{QC image}, this domain is contained in the complement 
of $\Omega \cup T_\Omega(t)$. Therefore  $V'$  is contained inside 
some component $U$  of $W(a, \epsilon)$ for $a \in S(f)$.

Lemma \ref{HM 1} applies to $U$ and Lemma \ref{HM 2} applies 
to $\Omega$, giving estimates that contradict the conclusion 
of Theorem \ref{near constant} (partition elements for 
$\partial U$  are contained in $r$-neighborhoods of partition 
elements for $\partial \Omega$). This proves that $\Omega$ could 
not have been the tract of any $f \in \classS$.
\end{proof}

Although the half-strip  cannot be approximated by 
Speiser class  model domains with a single tract, 
Figure \ref{TwoExtra} shows 
how it can be approximated by  functions in $ \classS_{2,0}$ with 
two, three  or infinitely many tracts.
To apply Theorem \ref{Exists}, 
 we must  add vertices so that the 
bounded geometry and $\tau$-length conditions  are
satisfied. Bounded geometry is trivial in these 
pictures using vertices from the obvious lattice 
and  Lemma \ref{tube bound} easily gives a strictly 
positive $\tau$-length lower bound for each tract 
(as usual, we can then multiply $\tau$ on each 
component by a positive constant to attain a $\tau$-length 
lower bound of $\pi$). 

\begin{figure}[htb]
\centerline{
\includegraphics[height=1.7in]{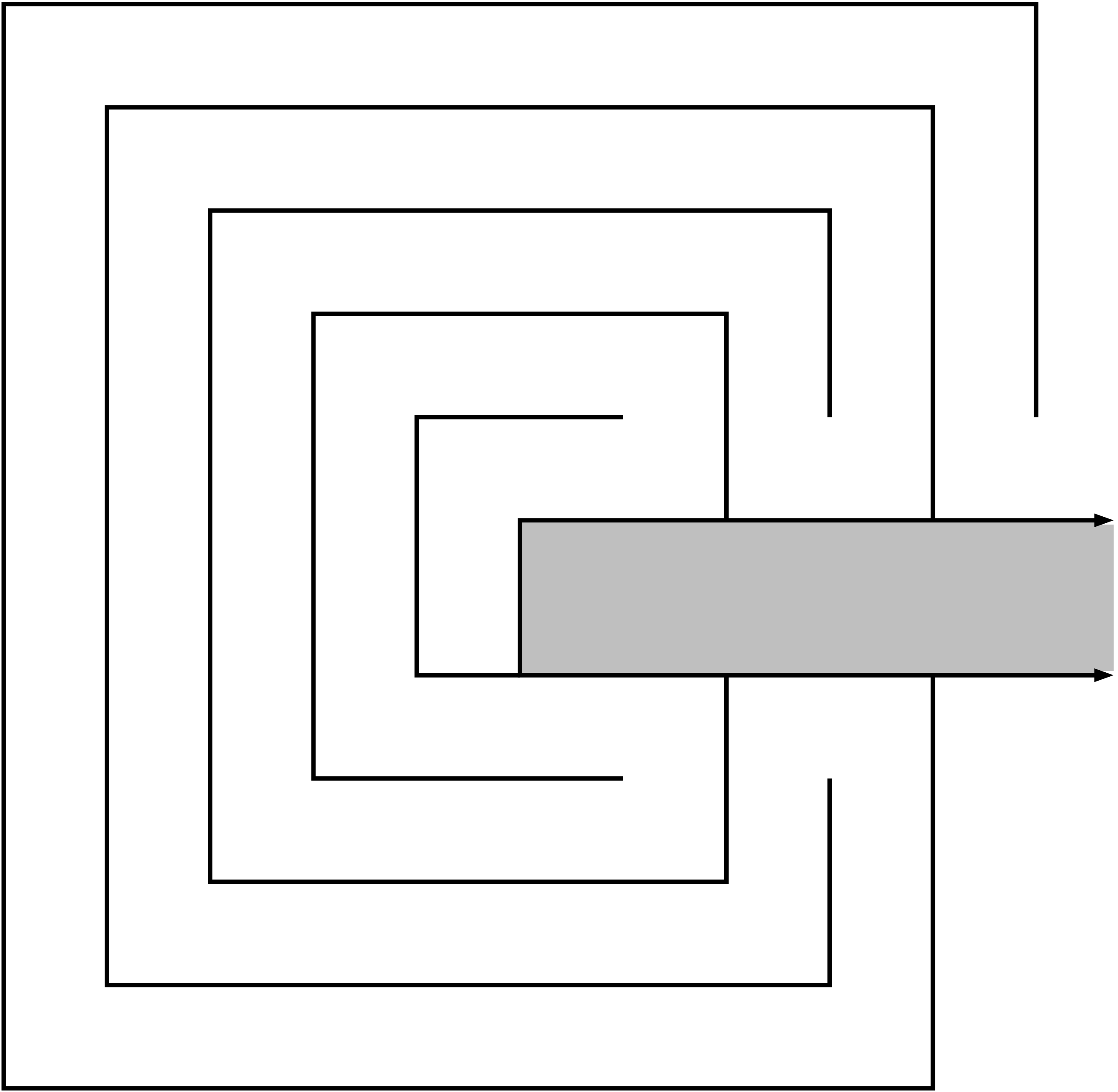}
\hphantom{xxxxx} 
\includegraphics[height=1.7in]{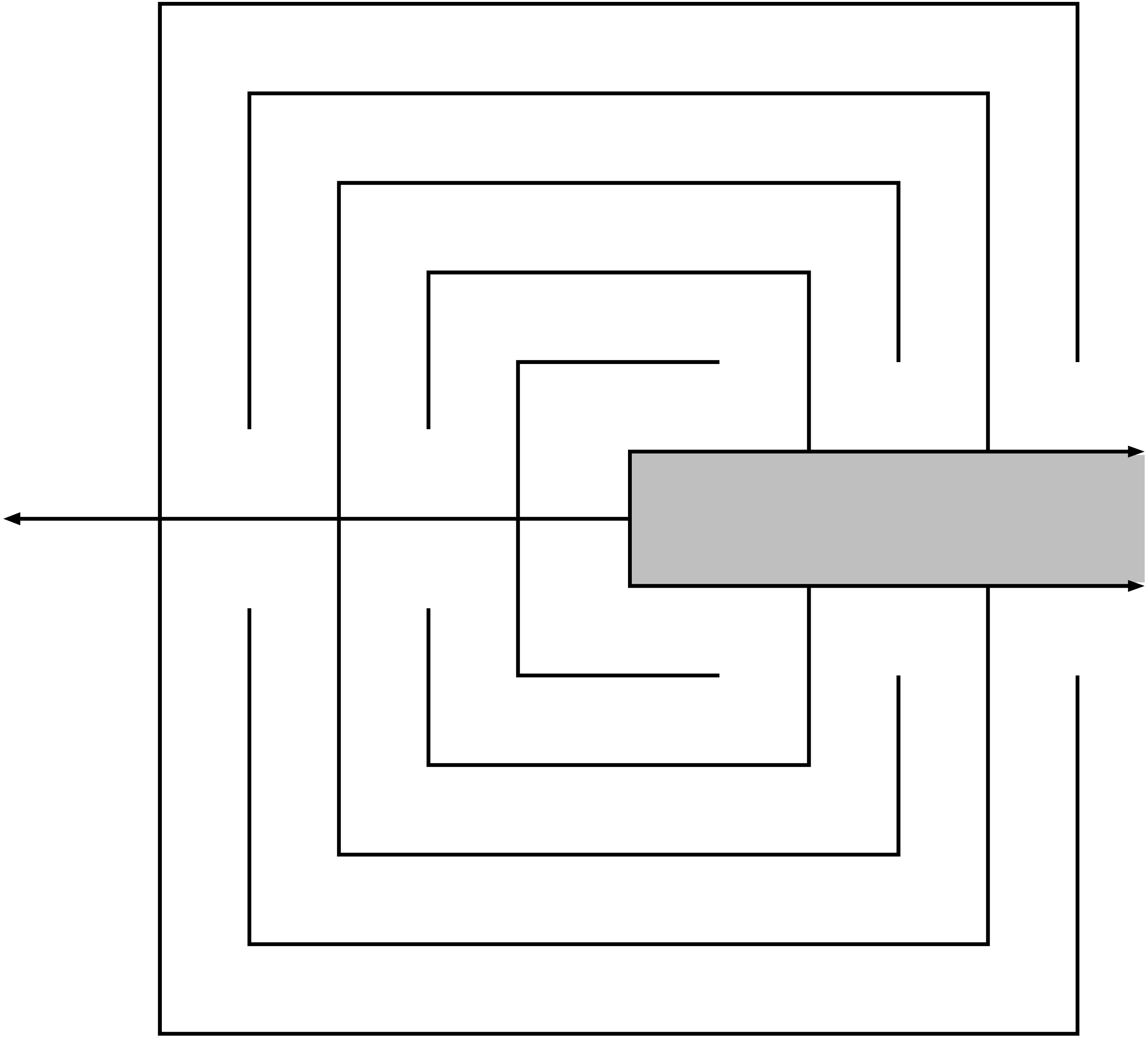}
\hphantom{xxxxx} 
\includegraphics[height=1.7in]{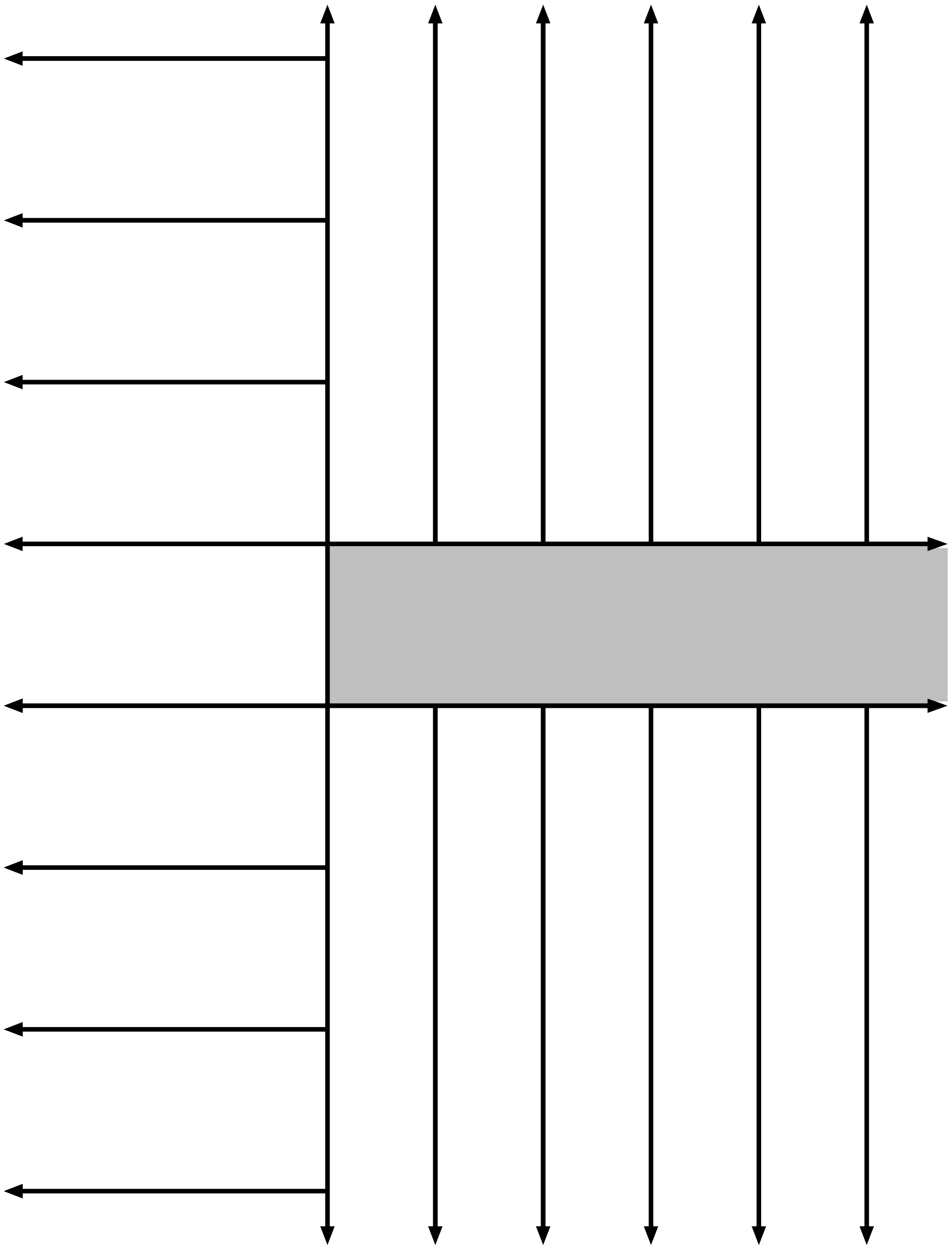}
}
\caption{ \label{TwoExtra}
Although the half-strip  cannot be a Speiser class
model domain (or even  quasiconformally mapped 
to a Speiser class model domain), it can be approximated
by a tract of such a model. These pictures show 
some ways this can occur using 
two, three  or infinitely many tracts.
}
\end{figure}

Lemma \ref{tube bound} also implies that the conformal partition of 
each of the drawn tracts has elements whose sizes decay 
exponentially quickly. This implies that $T(r)$ (the set 
where the quasiconformal correction map $\varphi$ in 
Theorem \ref{Exists} has its dilatation supported) has 
finite Lebesgue area.  If we replace $\tau$ by 
a large positive  integer multiple of itself, say 
$N \cdot \tau$, each edge of the conformal partition is 
divided into $N$ segments. This implies that the 
area of  $T(r)$ 
decreases to zero as  $N \nearrow \infty$. 
However, the maximal dilatation of $\varphi$ remains bounded, 
independent of $N$. 
A standard argument then shows that 
the corresponding map 
$\varphi$  tends uniformly on compact sets (or uniformly 
on the Riemann sphere)  
to the identity map as $N$ increases.
A more careful argument shows that we can normalize the 
correction maps so that they tend  to the 
identity  uniformly with respect to the Euclidean metric 
on the whole plane 
(e.g., see Theorem 1.1 of \cite{Bishop-small-support};
 our examples satisfy the $(\epsilon, \varphi)$-thin 
hypothesis  of that result).
 Thus the tracts of the resulting  Speiser functions
can approximate the tracts  in Figure \ref{TwoExtra} 
as closely as we wish in the Hausdorff metric on the plane. 

In general, it seems that 
the shapes of  individual tracts of Speiser model 
domains  and 
Eremenko-Lyubich model domains  do not differ significantly.
However, Speiser models only allow   disjoint
tracts to be combined in certain ways depending on the 
choice of $\tau_j$ in each tract, whereas 
Eremenko-Lyubich models allow  disjoint tracts to be combined
arbitrarily, and $\tau_j$ can be chosen on each tract 
independently of the choices in other tracts.

Roughly speaking,  each tract of  a Speiser 
class function can ``see''  other nearby tracts in 
the sense that it can be connected to such tracts by 
path families that avoid the singular set and 
 come with nice geometric estimates
(see the proof of Theorem \ref{near constant}).
However, the singularities of an Eremenko-Lyubich function 
can effectively ``block the view'' between different 
tracts. It would be reasonable to expect that if the 
singular set of an Eremenko-Lyubich functions was 
infinite but ``sparse'' in an appropriate sense, then 
nearby tracts would be forced to satisfy compatibility 
conditions similar to Speiser class functions.
What can be said about entire functions where the 
singular set is restricted to lie in a given 
closed set $E$, or satisfies some bound on its 
area, dimension or  capacity, or satisfies some 
other natural geometric restriction?
%------------------------------------------------------------------

\bibliography{models}
\bibliographystyle{plain}

\end{document}